\documentclass[reqno,fleqn, noinfoline]{imsart}
\usepackage{amsmath, amsthm, amssymb}
\usepackage{verbatim,enumerate,graphicx,bbm}
\usepackage{natbib,hyperref}
\usepackage{bigints, relsize}
\usepackage{mathabx}
\usepackage{color}
\usepackage{enumitem}
\usepackage{multirow}
\usepackage{nccmath}
\allowdisplaybreaks

\usepackage{tikz, pgfplots}
\usetikzlibrary{calc,decorations.markings}

\newcommand*\circled[1]{\tikz[baseline=(char.base)]{
		\node[shape=circle,draw,inner sep=2pt] (char) {#1};}}

\usepackage{empheq}

\newcommand\pushright[1]{\noindent\makebox[0.97\textwidth]{\hfill$\displaystyle#1$}\vspace{2ex}}
\newcommand\pushrightn[1]{\noindent\makebox[0.84\textwidth]{\hfill$\displaystyle#1$}\vspace{2ex}}

\def\N{\mathbb{N}}

\def\R{\mathbb{R}}
\def\C{\mathbb{C}}
\def\S{\mathbb{S}}

\numberwithin{equation}{section}
\newtheorem{theorem}{Theorem}
\newtheorem{corollary}{Corollary}
\newtheorem{lemma}{Lemma}
\newtheorem{proposition}{Proposition}
\theoremstyle{remark}

\theoremstyle{definition}
\newtheorem{definition}{Definition}

\newcommand{\1}[1]{\mathbbm{1}\!\left[#1\right]}

\newcommand{\E}[1]{\operatorname{E}\!\left[#1\right]}
\newcommand{\Var}[1]{\operatorname{Var}\!\left[#1\right]}
\newcommand{\tr}[1]{\operatorname{tr}\!\left(#1\right)}
\newcommand{\Tr}{\operatorname{tr}}
\newcommand{\Etr}{\operatorname{etr}}

\newcommand{\Log}{\operatorname{Log}}

\newcommand{\Kodd}{\1{K\,\mathrm{odd}} }
\newcommand{\Keven}{\1{K\,\mathrm{even}} }
\newcommand{\Lodd}{\1{L\,\mathrm{odd}} }

\newcommand{\kodd}{\1{k\,\mathrm{odd}} }
\newcommand{\keven}{\1{k\,\mathrm{even}} }

\begin{document}

\begin{frontmatter}
\title{The middle-scale asymptotics of Wishart matrices}
\runtitle{Mid-scale Wishart asymptotics}
\maketitle

\begin{aug}
  \author{\fnms{Didier}  \snm{Ch\'etelat}
  \ead[label=e1]{didier.chetelat@polymtl.ca}}
  \and
  \author{\fnms{Martin T.} \snm{Wells}
  \corref{}\ead[label=e2]{mtw1@cornell.edu}}

  \runauthor{Ch\'etelat and Wells}

  \affiliation{\'Ecole Polytechnique de Montr\'eal and Cornell University}

  \address{Department of Applied Mathematics\\
		   and Industrial Engineering \\
		  \'Ecole Polytechnique \\
		  Montr\'eal, Qu\'ebec H3T 1J4 \\
		  Canada \\
          \printead{e1}}

  \address{Department of Statistical Science\\
		  Cornell University\\
		  1198 Comstock Hall\\
		  Ithaca, NY 14853-3801\\
		  USA\\
          \printead{e2}}
\end{aug}

\begin{abstract}
We study the behavior of a real $p$-dimensional Wishart random matrix with $n$ degrees of freedom when $n,p\rightarrow\infty$ but $p/n\rightarrow0$. We establish the existence of phase transitions when $p$ grows at the order $n^{(K+1)/(K+3)}$ for every $k\in\N$, and derive expressions for approximating densities between every two phase transitions. To do this, we make use of a novel tool we call the G-transform of a distribution, which is closely related to the characteristic function. We also derive an extension of the $t$-distribution to the real symmetric matrices, which naturally appears as the conjugate distribution to the Wishart under a G-transformation, and show its empirical spectral distribution obeys a semicircle law when $p/n\rightarrow0$. Finally, we discuss how the phase transitions of the Wishart distribution might originate from changes in rates of convergence of symmetric $t$ statistics.
\end{abstract}


\end{frontmatter}


\section{Introduction}\label{sec:introduction}

The roots of random matrix theory lies in statistics, with the work of \citet{wishart28} and \cite{bartlett33}, and in numerical analysis, with the work of \citet{vonneumann47}. In this early period, many well-known matrix distributions were introduced. This includes the real Gaussian matrix ensemble $\text{G}(p,q)$, a $p\times q$ matrix with independent standard Gaussian entries, the Gaussian orthogonal ensemble $\text{GOE}(p)$, the distribution of a symmetric matrix $(X+X^t)/\sqrt{2}$ with $X\sim\text{G}(p,p)$, and the Wishart (also known as Laguerre) distribution $\text{W}_p(n, I_p/n)$, the distribution of a symmetric matrix $XX^t/n$ with $X\sim\text{G}(p, n)$. During that time, the main concern was to derive properties of these distributions for a fixed dimension. Some asymptotics of the Wishart distribution were considered, but only as $n\rightarrow\infty$ for fixed $p$.

Starting with the pioneering work of \citet{wigner51, wigner55, wigner57}, \citet{porter60}, \citet{gaudin61} and \citet{mehta1960a, mehta1960b}, researchers began investigating the asymptotics of Gaussian ensembles as their dimension grew to infinity. As a result of decades of work, the behavior of a $\text{GOE}(p)$ matrix is now well understood both in the classical setting where $p$ is fixed, and in the setting where $p\rightarrow\infty$.

However, the situation asymptotics of the Wishart distribution is more complicated, as it depends on two parameters, $n$ and $p$, and initial progress was slow. The work of \citet{marchenko67} clearly established that the analogue of a Gaussian orthogonal ensemble matrix whose dimension $p$ grows to infinity is a Wishart matrix whose degrees of freedom $n$ and dimension $p$ jointly grow to infinity in such a way that $p/n\rightarrow c\in(0,1)$. Since then, we gained a very good understanding of the behavior of Wishart matrices in this regime.

But this body of work left open the question as to what happens to a Wishart matrix when $n,p\rightarrow\infty$ with $p/n\rightarrow0$. Since such asymptotics are middle-scale between the \textit{classical} regime where $p$ is fixed as $n\rightarrow\infty$ and the \textit{high-dimensional} regime where $p/n\rightarrow c\in(0,1)$, we might refer to them as \textit{middle-scale} regimes. Hence, we might ask: what is the asymptotic behavior of a Wishart matrix $\text{W}_p(n, I_p/n)$ in the middle-scale regimes? This question is addressed this article.

To gain some intuition, it is instructive to look at the eigenvalues $\lambda_1>\dots>\lambda_p>0$ of a $\text{W}_p(n, I_p/n)$ Wishart matrix. In the classical regime where $p$ is fixed as $n\rightarrow\infty$, the eigenvalues must all almost surely tend to $1$ by the strong law of large numbers. In constrast, in the high-dimensional regime where both $n,p\rightarrow\infty$ with $p/n\rightarrow c\in(0,1)$, the Marchenko-Pastur law states that for any bounded, continuous $f$,
\begin{ceqn}
{\setlength{\abovedisplayskip}{7pt}	
\setlength{\belowdisplayskip}{7pt}\begin{align*}
\frac1n\sum_{i=1}^n f(\lambda_i) \;\rightarrow\; \int_{c_-}^{c_+} f(l)\frac{\sqrt{(c_+-l)(c_--l)}}{2\pi cl}dl
\qquad\text{a.s.},
\end{align*}}%
\end{ceqn}
where $c_\pm = (1\pm\sqrt{c})^2$. Thus the eigenvalues do not all tend to $1$, but rather distribute themselves in the shape of a Marchenko-Pastur law with parameter $c$.

What happens between these two extremes? When $c\rightarrow0$, the Marchenko-Pastur law converges weakly to a Dirac measure with mass at $1$. This suggests that whenever $n,p\rightarrow\infty$ with $p/n\rightarrow0$ 
\begin{ceqn}
{\setlength{\abovedisplayskip}{5pt}	
\setlength{\belowdisplayskip}{5pt}\begin{align*}
\frac1n\sum_{i=1}^n f(\lambda_i) \;\rightarrow\; f(1)
\qquad\text{a.s.},
\end{align*}}%
\end{ceqn}
or in other words that the eigenvalues converge almost surely to $1$, as in the classical case.

This motivates a binary view of Wishart asymptotics. It appears that the behavior of a Wishart matrix in the middle-scale regimes is the same as in the classical regime, and therefore that there really are only two regimes: low-dimensional where $p/n\rightarrow0$, and high-dimensional where $p/n\rightarrow c\in(0,1)$.

This binary view has very concrete repercussions. For example, in statistics, many covariance matrix estimators have been developed that leverage high-dimensional Wishart asymptotics (see \citet{pourahmadi13} for a review). When faced with a problem where $p$ is large with respect to $n$, it has been argued that the high-dimensional asymptotics, rather than the classical, constitute the correct model. The binary view provides a useful rule of thumb: small $p$'s call for classical covariance estimators, while large $p$'s call for high-dimensional covariance estimators.

Unfortunately, recent results establish that this binary view is incorrect. In the classical regime where $p$ is fixed, the central limit theorem implies that
\begin{ceqn}
{\setlength{\abovedisplayskip}{7pt}	
\setlength{\belowdisplayskip}{7pt}\begin{align*}
\sqrt{n}\Big[W_p(n, I_p/n) - I_p\Big] \;\Rightarrow\;  \text{GOE}(p),
\end{align*}}%
\end{ceqn}
as $n\rightarrow\infty$, where the arrow stands for weak convergence. In fact, something better is known: recent work has extended this result to the case where $p$ tends to infinity. Recall that the total variation distance between two absolutely continuous distributions $F_1$ and $F_2$ with densities $f$ and $g$ is given by $\mathrm{d}_\text{TV}(F_1, F_2)=\mathrm{d}_\text{TV}(f_1, f_2)=\int|f_1(x)-f_2(x)|dx$. With different approaches, \cite{jiang15} and \citet{bubeck16a} independently established that
\begin{ceqn}
{\setlength{\abovedisplayskip}{7pt}	
\setlength{\belowdisplayskip}{7pt}\begin{align*}
\mathrm{d}_\text{TV}\bigg(\sqrt{n}\Big[W_p(n, I_p/n) - I_p\Big],  \,\text{GOE}(p)\bigg) 
\rightarrow 0
\end{align*}}%
\end{ceqn}
whenever $p^3/n\rightarrow 0$. Thus, when $p^3/n\rightarrow 0$, the same asymptotics hold as in the $p$ fixed case, and we might regard these regimes as rightfully belonging to the classical setting.

The surprising part is that the converse is true! When $p^3/n\nrightarrow0$, results of \citet{bubeck16a} and \citet{racz16} show that
\begin{ceqn}
{\setlength{\abovedisplayskip}{7pt}	
\setlength{\belowdisplayskip}{7pt}\begin{align*}
\mathrm{d}_\text{TV}\bigg(\sqrt{n}\Big[W_p(n, I_p/n) - I_p\Big],  \,\text{GOE}(p)\bigg) 
\nrightarrow 0.
\end{align*}}%
\end{ceqn}
Thus a phase transition occurs when $p$ is of the order $\sqrt[3]{n}$. This begs the question: if a normal approximation fails to hold when $p$ grows faster than $\sqrt[3]{n}$, what asymptotics hold? Is there a uniform asymptotic behavior that holds whenever $p/n\rightarrow0$ with $p^3/n\nrightarrow0$, or are there further phase transitions as the growth rate of $p$ is increased?

The results of this paper offers a mostly complete answer to this question. Namely, we establish that when $p^3/n\nrightarrow0$ but $p^2/n\rightarrow0$, 
\begin{ceqn}
{\setlength{\abovedisplayskip}{7pt}	
\setlength{\belowdisplayskip}{7pt}\begin{align*}
\mathrm{d}_\text{TV}\bigg(\sqrt{n}\Big[W_p(n, I_p/n) - I_p\Big], \,F_1\bigg) \rightarrow 0,
\end{align*}}%
\end{ceqn}
where $F_1$ is a continuous distribution on the space of real symmetric matrices whose density is given when $n\geq 3p-3$ by
{\setlength{\mathindent}{10pt}
\setlength{\abovedisplayskip}{5pt}	
\setlength{\belowdisplayskip}{7pt}\begin{align}&
f_1(X)  
\;\;\propto\;\; \bigg|\text{E}\bigg[\exp\bigg\{
\frac{i}{\sqrt{8}}\Tr XZ
- \frac{i}{3\sqrt{2n}}\Tr Z^3
+ \frac{1}{4n}\Tr Z^4 
+ \frac{\sqrt{2}i}{5n^{3/2}}\Tr Z^5
\notag\\&\hspace{0pt}
- \frac{1}{3n^2}\Tr Z^6
+ \frac{i(p+1)}{2\sqrt{2n}}\Tr Z
- \frac{p+1}{4n}\Tr Z^2
- \frac{4\sqrt{2}i(p+1)}{3n^{3/2}}\Tr Z^3
\bigg\}\bigg]\bigg|^2,
\label{eqn:f1-as-expectation}
\end{align}}%
for a $Z\sim \text{GOE}(p)$. When $p$ grows like $\sqrt{n}$, another phase transition occurs.  Namely, we establish that when $p^2/n\nrightarrow0$ but $p^{5/3}/n\rightarrow0$,
\begin{ceqn}
{\setlength{\abovedisplayskip}{7pt}	
\setlength{\belowdisplayskip}{7pt}\begin{align*}
\mathrm{d}_\text{TV}\bigg(\sqrt{n}\Big[W_p(n, I_p/n) - I_p\Big], \,F_2\bigg) \rightarrow 0,
\end{align*}}%
\end{ceqn}
where $F_2$ is a continuous distribution on the space of real symmetric matrices whose density is given when $n\geq 3p-3$ by
{\setlength{\mathindent}{10pt}
\setlength{\abovedisplayskip}{5pt}	
\setlength{\belowdisplayskip}{5pt}\begin{align}&
f_2(X)
\;\;\propto\;\; \bigg|\text{E}\bigg[\exp\bigg\{
\frac{i}{\sqrt{8}}\Tr XZ
-\frac{i}{3\sqrt{2n}}\!\Tr Z^3 
+\frac{1}{4n}\!\Tr Z^4
+\frac{\sqrt{2}i}{5n^{3/2}}\!\Tr Z^5
\notag\\&\pushright{
-\frac{1}{3n^2}\!\Tr Z^6
-\frac{2\sqrt{2}i}{7n^{5/2}}\!\Tr Z^7
+\frac{i(p\!+\!1)}{2\sqrt{2n}}\!\Tr Z 
-\frac{p\!+\!1}{4n}\!\Tr Z^2
-\frac{4\sqrt{2}i(p\!+\!1)}{3n^{3/2}}\!\Tr Z^3
}\notag\\&\hspace{50pt}
+\frac{p\!+\!1}{4n^2}\!\Tr Z^4
+i\frac{512(p\!+\!1)}{5n^{5/2}}\!\Tr Z^5
-\frac{1024(p\!+\!1)}{3n^3}\!\Tr Z^6
\bigg\}\bigg]\bigg|^2,
\label{eqn:f2-as-expectation}
\end{align}}%
again for a $Z\sim\text{GOE}(p)$.

\begin{figure}
\includegraphics[scale=0.25, trim={0 40pt 0 10pt}, clip]{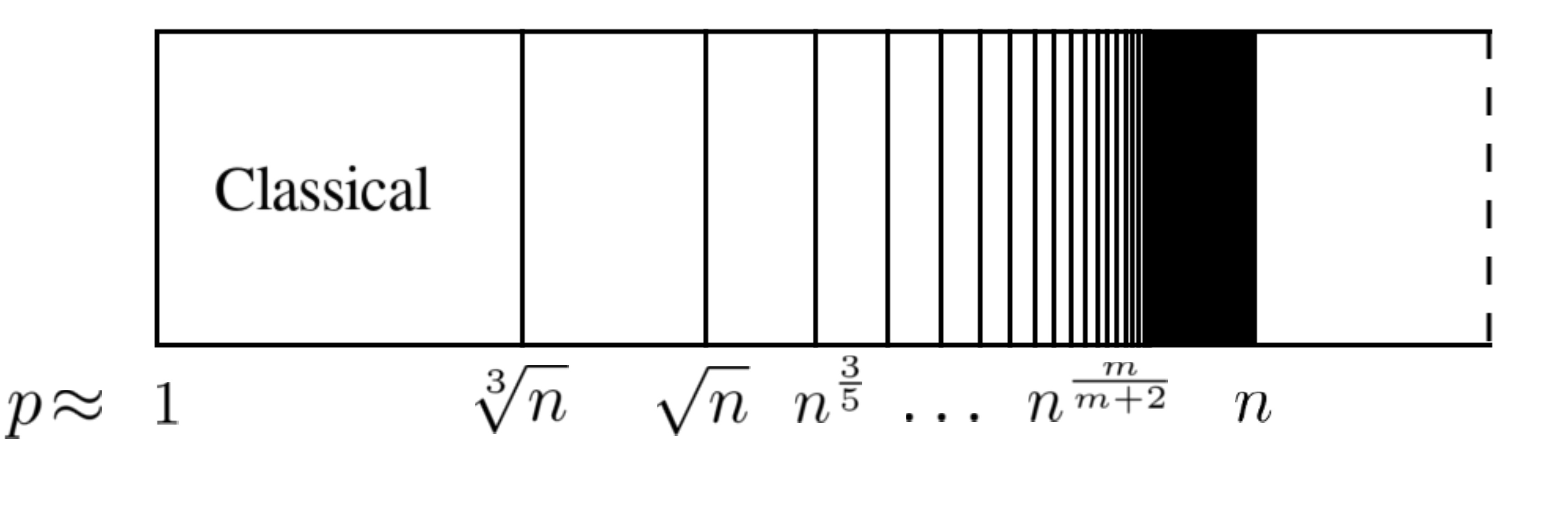}
\caption{Correct picture of Wishart asymptotics. This contrasts with the binary view, where no phase transitions occur between $p$ held constant and $p$ growing like $n$.}
\label{fig:regimes-full}
\end{figure}

In general, for every $K\in\N$ we find a continuous distribution $F_K$ on the space of real symmetric matrices, with density given when $n\geq 3p-3$ by
{\setlength{\mathindent}{20pt}
\setlength{\abovedisplayskip}{0pt}	
\setlength{\belowdisplayskip}{3pt}\begin{align}&
f_K(X)
\;\;\propto\;\; \Bigg|
\text{E}\Bigg[
\exp\Bigg\{
	\frac{i}{\sqrt{8}}\Tr (XZ)
	+ \frac{n}4
	\hspace{-2pt}\mathlarger{\mathlarger{\sum}}_{k=3}^{\substack{2K+3+\\ \Kodd}}\hspace{-4pt}
	i^k\Big(\frac{2}{n}\Big)^{\frac{k}2}
	\frac{\Tr Z^k}{k}
\notag\\[-5pt]&\pushrightn{
	+ \frac{p+\!1}4\hspace{-2pt}\mathlarger{\mathlarger{\sum}}_{k=1}^{\substack{2K+2-\\ \Kodd}}\hspace{-4pt}
	i^k\Big(\frac{2}{n}\Big)^{\frac{k}2}
	\frac{\Tr Z^k}{k}
	\Bigg\}
\Bigg]\Bigg|^2
}\label{eqn:fK-as-expectation}
\end{align}}%
for a $Z\sim\text{GOE}(p)$, which approximates the normalized Wishart distribution in some (but not all) middle-scale regimes. Namely, we prove the following, which can be regarded as the main result of this paper.

\begin{theorem}\label{thm:existence-densities}\hspace{-10pt}
For any $K\in\N$, the total variation distance between the the normalized Wishart distribution $\sqrt{n}[\text{W}_p(n,I_p/n)-I_p]$ and the $K^\text{th}$ degree density $f_K$ satisfies
\begin{ceqn}
{\setlength{\abovedisplayskip}{3pt}
\setlength{\belowdisplayskip}{3pt}\begin{align*}
\mathrm{d}_\text{TV}\bigg(\sqrt{n}\Big[W_p(n, I_p/n) - I_p\Big], \,F_K\bigg) \rightarrow 0
\end{align*}}%
\end{ceqn}
as $n\rightarrow\infty$ with $p^{K+3}/n^{K+1}\rightarrow0$.
\end{theorem}

The definition of $f_K$ and proof of Theorem \ref{thm:existence-densities} are found in Section \ref{sec:wishart-density}, and follow from definitions and results from Sections \ref{sec:gtransforms}, \ref{sec:symmetrict} and \ref{sec:wishart-G-transform} that constitute the bulk of this paper.

The main consequence of this theorem is the existence of an infinite countable number of phase transitions, occurring when $p$ grows like $n^{(K+1)/(K+3)}$ for $K\in\N$. A diagram is provided at Figure \ref{fig:regimes-full}. This naturally groups the middle-scale regimes satisfying  $\lim\limits_{n\rightarrow\infty}\frac{\log p}{\log n}<1$ by which semi-open interval $\big[\frac{K}{K+2}, \frac{K+1}{K+3}\big)$ their limit $\lim\limits_{n\rightarrow\infty}\frac{\log p}{\log n}$ belongs to. We might refer to this grouping as the \textit{degree} of the regime. In other words, we will say an middle-scale regime satisfying $\lim\limits_{n\rightarrow\infty}\frac{\log p}{\log n}<1$ has degree $K$ when $\lim\limits_{n\rightarrow\infty}\frac{\log p}{\log n}\in \big[\frac{K}{K+2}, \frac{K+1}{K+3}\big)$. 

The main result of this paper, Theorem \ref{thm:existence-densities}, may then be summarized as saying that the normalized Wishart distribution can be approximated by the distribution with density $f_K$ in every middle-scale regime of degree $K$ or less. The $0^\text{th}$ degree case corresponds to the classical setting, while the higher degrees correspond to previously unknown behavior. In fact, we show that our $0^\text{th}$ degree approximation $F_0$ is asymptotically equivalent to the Gaussian orthogonal ensemble. The results of this paper can therefore be regarded as a wide generalization of the Wishart asymptotics results of \cite{jiang15}, \cite{bubeck16a, bubeck16b} and \cite{racz16}.

Our approach relies on a novel technical tool we call the G-transform. It turns out that to understand middle-scale regime behavior of Wishart matrices, densities are less clear than characteristic functions (that is, Fourier transforms of densities). Unfortunately, characteristic functions are difficult to relate to metrics like the total variation distance. To remedy this problem, we develop the G-transform and some associated theory in Section \ref{sec:gtransforms}. An interesting aspect of G-transform theory is that to every distribution we can associate a closely related distribution called its G-conjugate. In fact, the G-conjugate of a Wishart matrix is essentially a generalization of the $t$ distribution to real symmetric matrices. In Section \ref{sec:symmetrict}, we define and derive several results concerning this new distribution, including a semicircle law. From these results, we derive in Section \ref{sec:wishart-G-transform} approximations to the Wishart distribution for middle-scale regimes of every degree. Since these approximations are given using the language of G-transforms, we derive in Section \ref{sec:wishart-density} density approximations, from which Theorem \ref{thm:existence-densities} follows. We briefly discuss what concrete effects the phase transitions might have on Wishart asymptotics in Section \ref{sec:explanation}. Finally, we compile auxiliary results in Section \ref{sec:auxiliary}, while we discuss in Section \ref{sec:conclusion} open questions that arise from these results.

Although the results of this paper explain a large part of the behavior of Wishart matrices when $p/n\rightarrow0$, there exists regimes for which $p/n\rightarrow0$ yet $p\not\in O(n^{(K+1)/(K+3)})$ for all $K\in\N$, or in other words for which $\lim\limits_{n\rightarrow\infty}\frac{\log p}{\log n}=1$. An example is when $p$ grows at the order $n^{1-1/\sqrt{\log n}}$. Although the results of our paper characterize almost all middle-scale regimes in the sense that among those regimes satisfying $\lim\limits_{n\rightarrow\infty}\frac{\log p}{\log n}\leq1$, those such that $\lim\limits_{n\rightarrow\infty}\frac{\log p}{\log n}=1$ represent a negligible set, they nonetheless exist. One might regard regimes such as those as having infinite degree. Beyond this, however, it is difficult to say anything about the behavior of Wishart matrices in these regimes. More work in that direction is clearly needed.
\section{Notation and definitions}\label{sec:notation}

The transpose of a matrix is denoted ${}^t$, and the identity matrix of dimension $p$ is $I_p$. As is standard, we take the trace operator to have lower priority than the power operator: thus for a matrix $X$, $\Tr X^k$ means the trace of $X^k$. We will write $\Tr^kX$ when we mean the $k^\text{th}$ power of the trace of $X$. The Kronecker delta is the symbol $\delta_{kl}=\1{k=l}$.

The space of all real-valued symmetric matrices is denoted $\S_p(\R) = \{X\in\mathbb{M}_p(\R) \vert X=X^t\}$. For a symmetric matrix $X$, we define the symmetric differentiation operator $\partial_\text{s}/\partial_\text{s}X_{kl}$ by
\begin{align*}
\frac{\partial_\text{s}}{\partial_\text{s}X_{kl}} = \frac{1+\delta_{kl}}2\frac{\partial}{\partial X_{kl}}
= \begin{cases}
\frac12\frac{\partial}{\partial X_{kl}} & \text{ for }k\not=l\\
\frac{\partial}{\partial X_{kk}} & \text{ for }k=l.
\end{cases}
\end{align*}
This operator has the elegant property that $\frac{\partial_\text{S}}{\partial_\text{S}X_{kl}}\Tr (XY) = Y_{kl}$ for any two symmetric matrices $X$, $Y$.

The space of symmetric matrices $\S_p(\R)$ can be assimilated to $\R^{p(p+1)/2}$ by mapping a symmetric matrix to its upper triangle. By integration over $\S_p(\R)$, we mean integration with respect to the pullback Lebesgue measure under this isomorphism, that is,
{\setlength{\mathindent}{10pt}
\setlength{\abovedisplayskip}{3pt}	
\setlength{\belowdisplayskip}{3pt}\begin{align*}&\qquad
\bigintsss_{\S_p(\R)} \hspace{-15pt}f(X) \,dX 
= \bigintsss_{\R^{p(p+1)/2}}\hspace{-25pt} 
	f\big(X\big)\prod_{i\leq j}^p \,dX_{ij}.
\end{align*}}%

We say a real symmetric matrix follows the Gaussian orthogonal ensemble $\text{GOE}(p)$ distribution if $X_{kl}$, $k\leq l$ are all independent, with diagonal elements $X_{kk}\sim\text{N}(0, 2)$ and off-diagonal elements $X_{kl}\sim\text{N}(0, 1)$.

Let $X$ be a $n\times p$ matrix of i.i.d. $\text{N}(0,1)$ random variables, and let $\Sigma$ be a $p\times p$ positive-definite matrix. The Wishart distribution $\text{W}_p(n, \Sigma)$ is the distribution of the random matrix $\Sigma^{\frac12}X^tX\Sigma^{\frac12}$. This is a special case of the matrix gamma distribution. Following \citet[Section 3.6]{gupta99}, we say a positive-definite matrix $X$ has a matrix gamma distribution $\text{G}_p(\alpha, \Sigma)$ with shape parameter $\alpha>(p-1)/2$ and scale parameter $\Sigma$  if it has density over $\S_p(\R)$ given by
{\setlength{\abovedisplayskip}{5pt}	
\setlength{\belowdisplayskip}{5pt}\begin{align*}
f(X) = \frac1{|\Sigma|^{\alpha}\Gamma_p\left(\alpha\right)}\big|X\big|^{\alpha-\frac{p+1}2}\exp\Big\{-\Tr (\Sigma^{-1}X)\Big\}\1{X>0},
\end{align*}}%
where $\Gamma_p$ is the multivariate gamma function. With this definition, the Wishart distribution $W_p(n, \Sigma)$ is a matrix gamma with shape $\frac{n}2$ and scale $2\Sigma$.

While studying the Wishart distribution, the expression $n-p-1$ comes up so often that it makes sense to give it its own symbol. We will therefore write $m=n-p-1$.

The Hellinger distance is metric between absolutely continuous probability measures. For two distributions $F_1$ and $F_2$ with densities $f_1$ and $f_2$, their Hellinger distance is defined as
{\setlength{\abovedisplayskip}{3pt}	
\setlength{\belowdisplayskip}{3pt}\begin{align*}
\mathrm{H}(F_1, F_2) = \mathrm{H}(f_1, f_2) = \Big(\int \Big|f_1^{1/2}(x)-f_2^{1/2}(x)\Big|^2dx\Big)^{\frac12}.
\end{align*}}%
The Hellinger distance is closely related to the total variation distance by the inequalities
{\setlength{\abovedisplayskip}{5pt}	
\setlength{\belowdisplayskip}{5pt}\begin{align}
\frac12 \mathrm{d}_\text{TV}(f_1,f_2) \;\leq\; \mathrm{H}(f_1, f_2) \;\leq\; \mathrm{d}^{1/2}_\text{TV}(f_1,f_2).
\label{eqn:HvsTV-densities}
\end{align}}%
In particular, $\mathrm{H}(f_1, f_2)\rightarrow0$ if and only if $\mathrm{d}_\text{TV}(f_1,f_2) \rightarrow0$. Thus they can be seen as inducing the same topology on absolutely continuous probability measures, called the \textit{strong topology}, in contrast to the topology induced by weak convergence of measures called the \textit{weak topology}. One can show that if a sequence of measures converges in the strong sense (i.e. in the $\mathrm{d}_\text{TV}$ or $\mathrm{H}$ metrics), then it converges weakly.

\section{G-transforms}\label{sec:gtransforms}

Our analysis of Wishart matrices relies heavily on a tool we call the G-transform of a probability measure. To do so, we first need to define the Fourier transform over symmetric matrices.

In Section \ref{sec:notation}, we clarified what we meant by integration over $\S_p(\R)$. For a function $f: \S_p(\R)\rightarrow\C$ in $L^1(\S_p(\R))$, we define its Fourier transform to be
{\setlength{\mathindent}{10pt}
\setlength{\abovedisplayskip}{3pt}	
\setlength{\belowdisplayskip}{3pt}\begin{align}&\qquad
\mathcal{F}\{f\}(T) 
= \frac1{2^{\frac{p}2}\pi^{\frac{p(p+1)}4}} \bigintsss_{\S_p(\R)}\hspace{-15pt}
	e^{-i\Tr (TX)}f(X) \,dX.
\label{def:fourier-transform}
\end{align}}%
It is more common to define the Fourier transform on symmetric matrices with the integrand $\exp\big\{-i\sum_{k\leq l}T_{kl}X_{kl}\big\}$, but choosing $\exp\big\{i\Tr (TX)\big\}$ considerably simplifies our computations.

We extend this definition to $f\in L^r(\S_p(\R))$, $1<r\leq2$ in the usual manner. Because of the specific normalization chosen, this definition obeys a simple version of Plancherel's theorem, namely
{\setlength{\mathindent}{10pt}
\setlength{\abovedisplayskip}{3pt}	
\setlength{\belowdisplayskip}{3pt}\begin{align*}&\qquad
\bigintsss_{\S_p(\R)}\hspace{-15pt} f(X)\widebar{g}(X) \,dX 
= \bigintsss_{\S_p(\R)}\hspace{-15pt} \mathcal{F}\{f\}(T)\widebar{\mathcal{F}\{g\}}(T) \,dT.
\end{align*}}%

We now define the G-transform. In itself, the definition has nothing to do with symmetric matrices and could have been perfectly well defined on any other space endowed with a Fourier transform.

\begin{definition}[G-transform of a density]\label{def:gtransform}
Let $f$ be an integrable function $\S_p(\R)\rightarrow\C$. Its G-transform is the complex-valued function $\mathcal{G}\{f\}: \S_p(\R)\rightarrow\C$ defined by
{\setlength{\mathindent}{10pt}
\setlength{\abovedisplayskip}{5pt}	
\setlength{\belowdisplayskip}{5pt}\begin{align}\label{eqn:gtransform}&\qquad
\mathcal{G}\{f\} = \mathcal{F}\{f^{1/2}\}^2,
\end{align}}%
where $z^{1/2}$ stands for the principal branch of the complex logarithm.
\end{definition}

In the same way that the Fourier transform maps $L^2(\S_p(\R))$ to itself, the G-transform maps $L^1(\S_p(\R))$ to itself. 

By extension, for an absolutely continuous distribution on $\S_p(\R)$ with density $f$, we will define its G-transform to be the G-transform of its density. (This usage mirrors other transforms, such as the Stietjes transform.) We will usually denote the G-transform of $f$ by $\psi$. Since a density is integrable, this is always well-defined. Moreover, $f = \mathcal{F}^{-1}\{\psi^{1/2}\}^2$, so the density can be recovered from the G-transform, and therefore to understand a distribution it is equivalent to study its density or its G-transform.

Two comments are in order. First, for many densities, $f^{1/2}\in L^1(\S_p(\R))$. In this case, the G-transform can be written explicitly as
{\setlength{\mathindent}{10pt}
\setlength{\abovedisplayskip}{3pt}	
\setlength{\belowdisplayskip}{3pt}\begin{align}\label{eqn:gtransform-explicit}&\qquad
\psi(T) = \mathcal{G}\{f\}(T) 
= \frac1{2^p\pi^{\frac{p(p+1)}2}}
	\bigg(\bigintsss_{\S_p(\R)}\hspace{-15pt} e^{-i\Tr(TX)} f^{1/2}(X) \,dX\bigg)^2.
\end{align}}%
Second, throughout this article we will often talk about ``the'' square root of a G-transform. To be clear, by $\psi^{1/2}$ we will always mean $\mathcal{F}\{f^{1/2}\}$.

Now, in many ways, the G-transform behaves similarly to the characteristic function (Fourier transform of a density), but it has unique features. First, Plancherel's theorem yields that
{\setlength{\mathindent}{10pt}
\setlength{\abovedisplayskip}{3pt}	
\setlength{\belowdisplayskip}{3pt}\begin{align}\label{eqn:gconjugate} \qquad
\bigintsss_{\S_p(\R)}\hspace{-15pt} |\psi(T)| \,dT
&= \bigintsss_{\S_p(\R)}\hspace{-15pt} |\psi^{1/2}(T)|^2 \,dT
\notag\\&
= \bigintsss_{\S_p(\R)}\hspace{-15pt} |f^{1/2}(X)|^2 \,dX
= \bigintsss_{\S_p(\R)}\hspace{-15pt} |f(X)| \,dX
= 1.
\end{align}}%
Thus $|\psi|$ is itself a density, which we will call the \textit{G-conjugate} of $f$. (In particular, $\psi^{1/2}$ is much like a quantum-mechanical wavefunction.) We will also use an asterisk notation, so that the G-conjugate of a $\text{N}(0,1)$ distribution will be denoted $\text{N}(0,1)^*$. For example, straightforward computations yield that $\text{N}(0,1)^*= \text{N}(0, 1/8)$, $\chi^{2*}_n=\frac1{\sqrt{8n}}t_{n/2}$ (where $\chi^2_\nu$ and $t_\nu$ are the univariate $\chi^2$ and $t$ distributions with $\nu$ degrees of freedom, respectively)  and $(aF+b)^* = a^{-1}F^*$ for any distribution $F$ and scalars $a\not=0$, $b\in\R$. Studying the G-conjugate of the Wishart distribution will play a key part in deriving results about the Wishart distribution itself. We should note that, in general, the double G-conjugate $F^{**}$ is not the same as $F$. For example, $\chi^{2**}_n$ is a density involving modified Bessel functions of the first kind, not a $\chi^2_n$.

A second feature that distinguishes G-transforms from characteristic functions is that they are easy to relate to the Hellinger distance between probability measures. Consider two densities $f_1, f_2$ with G-transforms $\psi_1, \psi_2$. By analogy, we could define the ``total variation'' and ``Hellinger'' distances of $\psi_1$ and $\psi_2$ by
{\setlength{\mathindent}{0pt}
\setlength{\abovedisplayskip}{3pt}	
\setlength{\belowdisplayskip}{3pt}\begin{align}\qquad
\mathrm{d}_\text{TV}(\psi_1, \psi_2) &= \bigintsss_{\S_p(\R)}\hspace{-15pt} |\psi_1(T) - \psi_2(T)| \,dT
\label{eqn:tv-Gtransforms},
\\
\text{and}\qquad
\mathrm{H}(\psi_1, \psi_2) &= \Big(\bigintsss_{\S_p(\R)}\hspace{-15pt} |\psi_1^{1/2}(T) - \psi_2^{1/2}(T)|^2 \,dT\Big)^{\frac12}.
\label{eqn:hellinger-Gtransforms}
\end{align}}%
Since the modulus of the G-transforms integrate to one, their total variation and Hellinger distances are related to each other in the same way as in Equation \ref{eqn:HvsTV-densities} for densities, namely
{\setlength{\abovedisplayskip}{5pt}	
\setlength{\belowdisplayskip}{5pt}\begin{align}
\frac12 \mathrm{d}_\text{TV}(\psi_1, \psi_2) \;\leq\; \mathrm{H}(\psi_1, \psi_2) \;\leq\; \mathrm{d}^{1/2}_\text{TV}(\psi_1, \psi_2).
\label{eqn:HvsTV-Gtransforms}
\end{align}}%
Thus $\mathrm{d}_\text{TV}(\psi_1, \psi_2)\rightarrow0$ if and only if $\mathrm{H}(\psi_1, \psi_2)\rightarrow0$. 
But the Hellinger distance between G-transforms is much more useful. Indeed, by the Plancherel theorem, for any two densities $f_1, f_2$ with G-transforms $\psi_1, \psi_2$, their Hellinger distance satisfies
{\setlength{\mathindent}{10pt}
\setlength{\abovedisplayskip}{3pt}	
\setlength{\belowdisplayskip}{3pt}\begin{align} \qquad
\mathrm{H}^2(f_1, f_2)
&= \bigintsss_{\S_p(\R)}\hspace{-15pt} |f_1^{1/2}(X) - f_2^{1/2}(X)|^2 \,dX
\notag\\&
= \bigintsss_{\S_p(\R)}\hspace{-15pt} |\psi_1^{1/2}(T) - \psi_2^{1/2}(T)|^2 \,dT
\hspace{5pt}
= \mathrm{H}^2(\psi_1, \psi_2).
\label{eqn:hellinger-equivalence}
\end{align}}%
Thus to compute the Hellinger distance $\mathrm{H}^2(f_1, f_2)$ between two densities, we can instead compute the Hellinger distance $\mathrm{H}^2(\psi_1, \psi_2)$ of their G-transforms. In contrast, there is no explicit way to express the Hellinger distance in terms of characteristic functions. And no such connection exists between the total variation distances of densities and G-transforms.

The G-transform does have some disadvantages compared to the Fourier transform. It is a non-linear transformation (and therefore not a true transform), and it does not behave well with respect to convolution. For our purposes, however, the advantages listed above outweigh these problems.

In practice, it is not aways easy to control the Hellinger distance directly, and one often focuses on the Kullback-Leibler divergence instead. The two quantities are related through the well known inequality
{\setlength{\mathindent}{10pt}
\setlength{\abovedisplayskip}{5pt}	
\setlength{\belowdisplayskip}{5pt}\begin{align*}\;\;
\mathrm{H}^2(f_1, f_2)
\;\;\leq\;\;
\E{\log\frac{f_1(X)}{f_2(X)}}\; \; \; \text{for }X\sim f_1.
\end{align*}}%
For G-transforms, the following analog holds, which clarifies our interest in G-conjugates:
{\setlength{\mathindent}{10pt}
\setlength{\abovedisplayskip}{3pt}	
\setlength{\belowdisplayskip}{5pt}\begin{align*}
\mathrm{H}^2(\psi_1, \psi_2) \;\;\leq\;\; \E{\Re\Log\frac{\psi_1(T)}{\psi_2(T)}} + 2\!\E{\left|\Im\Log\frac{\psi_1(T)}{\psi_2(T)}\right|}^2
\; \; \; \text{ for }T \sim F_1^*,
\end{align*}}%
where $\Log$ stands for the principal branch of the complex logarithm. In fact, in this article we will need a further generalization, where $\psi_2$ does not need to be a G-transform of a density.

\begin{proposition}[Kullback-Leibler inequality for G-transforms]\label{prop:generalizedkl}
Let $\psi_1$ be the G-transform of an absolutely continuous distribution $F_1$ on $\S_p(\R)$, and let $\psi_2$ be an integrable function $\S_p(\R)\rightarrow\C$. Then
{\setlength{\mathindent}{10pt}
\setlength{\abovedisplayskip}{3pt}	
\setlength{\belowdisplayskip}{3pt}\begin{align*}&
\mathrm{H}^2(\psi_1, \psi_2) 
\;\;\leq\;\; 
	\bigg[\!\! \bigintsss_{\S_p(\R)}\hspace{-15pt}   |\psi_2|(T)\,dT - 1\bigg]
	+ \E{\Re\Log\frac{\psi_1(T)}{\psi_2(T)}} 
\\[-4pt]&\hspace{140pt}
	+ 2\bigintsss_{\S_p(\R)}\hspace{-15pt}   |\psi_2|(T)\,dT^{\frac12} \cdot
		\E{\left|\Im\Log\frac{\psi_1(T)}{\psi_2(T)}\right|}^{\frac12}
\end{align*}}%
for $T\sim F_1^*$, where $\Log$ stands for the principal branch of the complex logarithm.
\end{proposition}
\begin{proof}
We can write
{\setlength{\mathindent}{10pt}
\setlength{\abovedisplayskip}{5pt}	
\setlength{\belowdisplayskip}{5pt}\begin{align*}&
\mathrm{H}^2(\psi_1, \psi_2) 
= \bigintss_{\S_p(\R)}\hspace{-20pt} 
	|\psi_1|(T) 
	+ |\psi_2|(T)|
	- \widebar{\psi}_1^{1/2} \psi_2^{1/2} 
	- \psi_1^{1/2}\widebar{\psi}_2^{1/2} 
	\,dT
\\[-5pt]&\hspace{10pt}
=\;\; \bigg[\!\bigintss_{\S_p(\R)}\hspace{-20pt} |\psi_2|(T)| - 1\bigg]
	+ 2
	- \hspace{-5pt}\bigint_{\S_p(\R)}\hspace{-20pt}
		\left[\frac{\psi_2^{1/2}(T)}{\psi_1^{1/2}(T)}
		+ \frac{\bar{\psi}_2^{1/2}(T)}{\bar{\psi}_1^{1/2}(T)}\right]
		\!|\psi_1|(T) \,dT
\\&\hspace{10pt}
=\;\; \bigg[\!\bigintss_{\S_p(\R)}\hspace{-20pt} |\psi_2|(T)| - 1\bigg]
	+ 2\bigg[1
	- \bigintss_{\S_p(\R)}\hspace{-15pt}
	\Re\left\{\frac{\psi_2^{1/2}(T)}{\psi_1^{1/2}(T)}\right\}
	\!|\psi_1|(T) \,dT\bigg]
\\&\hspace{10pt}
=\;\; \bigg[\!\bigintss_{\S_p(\R)}\hspace{-20pt} |\psi_2|(T)| - 1\bigg]
	+ 2\bigg[ 1
		- \bigintss_{\S_p(\R)}\hspace{-20pt}
		\exp\bigg\{\!\!
			-\!\frac12\Re\Log\frac{\psi_1(T)}{\psi_2(T)}
		\bigg\}
\\&\hspace{140pt}\cdot
		\cos\bigg(
			\frac12\Im\Log\frac{\psi_1(T)}{\psi_2(T)}
		\bigg)
		|\psi_1|(T) \,dT
	\bigg].
\end{align*}}%
Now using the inequality $-\cos(x)\leq -1 + \sqrt{2|x|}$ that holds for any $x\in\R$. The last quantity is bounded as
{\setlength{\mathindent}{10pt}
\setlength{\abovedisplayskip}{5pt}	
\setlength{\belowdisplayskip}{5pt}\begin{align*}&\hspace{10pt}
\leq\;\; \bigg[\!\bigintss_{\S_p(\R)}\hspace{-20pt} |\psi_2|(T)| - 1\bigg]
	+ 2\bigg[ 1
	- \bigintss_{\S_p(\R)}\hspace{-20pt}
	\exp\bigg\{\!\!
	-\!\frac12\Re\Log\frac{\psi_1(T)}{\psi_2(T)}
	\bigg\}
	|\psi_1|(T) \,dT\bigg]
\\&\hspace{40pt}
 + 2\bigintss_{\S_p(\R)}\hspace{-20pt}
	 \exp\bigg\{\!\!
		 -\!\frac12\Re\Log\frac{\psi_1(T)}{\psi_2(T)}
	\bigg\}
	\sqrt{\bigg|
		\Im\Log\frac{\psi_1(T)}{\psi_2(T)}
	\bigg|}
	|\psi_1|(T) \,dT.
\end{align*}}%
In the second term, use $1-x\leq -\log(x)$ for $x\geq0$, while in the third term, use the Cauchy-Schwarz inequality to obtain
{\setlength{\mathindent}{10pt}
\setlength{\abovedisplayskip}{5pt}	
\setlength{\belowdisplayskip}{5pt}\begin{align*}&\hspace{10pt}
\leq\;\; \bigg[\!\bigintss_{\S_p(\R)}\hspace{-20pt} |\psi_2|(T)| - 1\bigg]
	- 2\log \bigintss_{\S_p(\R)}\hspace{-20pt}
	\exp\bigg\{\!\!
	-\!\frac12\Re\Log\frac{\psi_1(T)}{\psi_2(T)}
	\bigg\}
	|\psi_1|(T) \,dT
\\&\hspace{20pt}
	+ 2
	\bigintss_{\S_p(\R)}\hspace{-20pt}
	\exp\bigg\{\!\!
	-\!\Re\Log\frac{\psi_1(T)}{\psi_2(T)}
	\bigg\}
	\!|\psi_1|(T) \,dT^{\frac12}
	\bigintss_{\S_p(\R)}\hspace{-20pt}
	\Big|
		\Im\Log\frac{\psi_1(T)}{\psi_2(T)}
	\Big|
	|\psi_1|(T) \,dT^{\frac12}.
\end{align*}}%
Now use Jensen's inequality in the second term and the algebraic identity $\exp\{-\Re\Log\psi_1(T)/\psi_2(T)\}\allowbreak=|\psi_2|(T)/|\psi_1|(T)$ in the third term to obtain
{\setlength{\mathindent}{10pt}
\setlength{\abovedisplayskip}{5pt}	
\setlength{\belowdisplayskip}{5pt}\begin{align*}&\hspace{10pt}
\leq\;\; \bigg[\!\bigintss_{\S_p(\R)}\hspace{-20pt} |\psi_2|(T)| - 1\bigg]
	+\bigintss_{\S_p(\R)}\hspace{-20pt}
		\Re\Log\frac{\psi_1(T)}{\psi_2(T)}
		|\psi_1|(T) \,dT
\\&\hspace{110pt}
	+ 2 \bigintss_{\S_p(\R)}\hspace{-20pt}
		|\psi_2|(T) \,dT^{\frac12}
		\bigintss_{\S_p(\R)}\hspace{-20pt}
		\Big|
			\Im\Log\frac{\psi_1(T)}{\psi_2(T)}
		\Big|
		|\psi_1|(T) \,dT^{\frac12},
\end{align*}}%
as desired.
\end{proof}

Let us now compute the G-transform of the Gaussian Orthogonal Ensemble and the normalized Wishart distribution, which will be needed in our proofs. The density of a $\text{GOE}(p)$ matrix over $\S_p(\R)$ is
{\setlength{\mathindent}{10pt}
\setlength{\abovedisplayskip}{5pt}	
\setlength{\belowdisplayskip}{5pt}\begin{align}\label{eqn:goe-density} &\qquad
f_{\text{GOE}}(X) = \frac1{2^{\frac{p(p+3)}4}\pi^{\frac{p(p+1)}4}} \exp\Big\{-\frac14\Tr X^2\Big\}.
\end{align}}%
To compute its G-transform, we will make use of the fact that the elements of a $\text{GOE}(p)$ matrix are independent to reduce the expression to a product of characteristic functions.

\begin{proposition}\label{prop:gtransform-goe} The G-transform of the Gaussian Orthogonal Ensemble density on $\S_p(\R)$ is
{\setlength{\mathindent}{5pt}
\setlength{\abovedisplayskip}{3pt}	
\setlength{\belowdisplayskip}{3pt}\begin{align*}&\qquad
\psi_{\text{GOE}}(T) 
= \frac{2^{\frac{p(3p+1)}4}}{\pi^{\frac{p(p+1)}4}} 
\exp\Big\{-4\Tr T^2\Big\}.
\end{align*}}%
\end{proposition}
\begin{proof}
From Equation \eqref{eqn:goe-density}, $f_{\text{GOE}}^{1/2}$ is proportional to the density of the $\sqrt{2}\,\text{GOE}(p)$ distribution, so it is integrable. Therefore, we can apply Equation \eqref{eqn:gtransform-explicit} to find that
{\setlength{\mathindent}{5pt}
\setlength{\abovedisplayskip}{3pt}	
\setlength{\belowdisplayskip}{5pt}\begin{align*}&\qquad
\psi_{\text{GOE}}^{1/2}(T) 
= \frac1{2^{\frac{p}2}\pi^{\frac{p(p+1)}4}}\bigintss_{\S_p(\R)}\hspace{-15pt}
\exp\Big\{-i\Tr (TX) \Big\} f_{\text{GOE}}^{1/2}(X) \,dX
\\ &\qquad\qquad
= \frac1{2^{\frac{p(p+7)}8}\pi^{\frac{3p(p+1)}8}} \bigintss_{\S_p(\R)}\hspace{-15pt}
\exp\Big\{-i\Tr (TX) -\frac18\Tr X^2\Big\} \,dX
\\ &\qquad\qquad
= \frac1{2^{\frac{p(p+7)}8}\pi^{\frac{3p(p+1)}8}} \bigintss_{\R^{p(p+1)/2}}\hspace{-25pt}
\exp\Big\{-2i\sum_{k< l}^p T_{kl}X_{kl} -\frac14\sum_{k< l}^p X^2_{kl}
\\&\hspace{175pt}
-i\sum_{k=1}^p T_{kk}X_{kk} -\frac18\sum_{k=1}^p X^2_{kk}
\Big\} 
\,\prod_{k\leq l}^{p}dX_{kl}
\\ &\qquad\qquad
= \frac{2^{\frac{p(3p+1)}8}}{\pi^{\frac{p(p+1)}8}} 
\prod_{k<l}^{p}\bigintss_{\R}
\exp\Big\{-2iT_{kl}X_{kl}\Big\} \frac{\exp\left\{-\frac14X^2_{kl}\right\}}{\sqrt{4\pi}} 
\,dX_{kl}
\\&\hspace{100pt}
\cdot\prod_{k=1}^{p}\bigintss_{\R}
\exp\Big\{-iT_{kk}X_{kk}\Big\} \frac{\exp\left\{-\frac18X^2_{kk}\right\}}{\sqrt{8\pi}} 
\,dX_{kk}
\\ &\qquad\qquad
= \frac{2^{\frac{p(3p+1)}8}}{\pi^{\frac{p(p+1)}8}} 
\prod_{k<l}^{p} \E{\exp\Big\{-\sqrt{8}iT_{kl}Z\Big\}}
\prod_{k=1}^{p} \E{\exp\Big\{-2iT_{kk}Z\Big\}}
\end{align*}}%
for $Z\sim \text{N}(0, 1)$. The characteristic function of a $\text{N}(0, 1)$ is $\exp(-t^2/2)$, so
{\setlength{\mathindent}{5pt}
\setlength{\abovedisplayskip}{3pt}	
\setlength{\belowdisplayskip}{3pt}\begin{align*}&\qquad\qquad
	= \frac{2^{\frac{p(3p+1)}8}}{\pi^{\frac{p(p+1)}8}} 
\prod_{k < l}^{p} \exp\Big\{-4T^2_{kl}\Big\}
\prod_{k = 1}^{p} \exp\Big\{-2T^2_{kk}\Big\}
\\&\qquad\qquad
= \frac{2^{\frac{p(3p+1)}8}}{\pi^{\frac{p(p+1)}8}} 
\exp\Big\{-2\Tr T^2\Big\}.
\end{align*}}%
Squaring this result yields the desired expression for $\psi_\text{GOE}$.
\end{proof}

In particular, we see that $|\psi_{\text{GOE}}|$ is the density of a $\text{GOE}(p)/4$ distribution, or in other words that $\text{GOE}(p)^*=\text{GOE}(p)/4$. In particular the Gaussian orthogonal ensemble is its own G-conjugate, up to a constant factor.

Let us now compute the G-transform of the normalized Wishart distribution. Unlike the $\text{GOE}(p)$ case, the elements of the matrix are not independent, but the elements of its Cholesky decomposition are. By being careful about complex changes of variables, we can reduce the computation of the G-transform to the computation of characteristic functions of the Cholesky elements.

\begin{proposition}\label{prop:gtransform-nw} Let $n\geq p-2$. Then the G-transform of the normalized Wishart distribution $\sqrt{n}[\text{W}_p(n,I_p/n)-I_p]$ density on $\S_p(\R)$ is given by
{\setlength{\mathindent}{10pt}
\setlength{\abovedisplayskip}{3pt}
\setlength{\belowdisplayskip}{0pt}\begin{align*}&\quad
\psi_{\text{NW}}(T)
= C_{n, p}
\exp\bigg\{2i\sqrt{n}\Tr  T\bigg\}
\bigg|I_p +i\frac{4 T}{\sqrt{n}}\bigg|^{-\frac{n+p+1}2}\!\!
\\[-3pt]&\hspace{-10pt}\text{with}
\\[-7pt]&\quad
C_{n, p}
=\frac{2^{\frac{p(n+2p)}2}}{\pi^{\frac{p(p+1)}2}n^{\frac{p(p+1)}4}}
\frac{\Gamma^2_p\left(\frac{n+p+1}4\right)}{\Gamma_p\left(\frac{n}2\right)}.
\end{align*}}%
\end{proposition}
\begin{proof}
Recall the notation $m=n-p-1$ used throughout this article. The density of a $Y\sim\text{W}_p(n,I_p/n)$ distribution is
{\setlength{\mathindent}{10pt}
	\setlength{\abovedisplayskip}{3pt}	
	\setlength{\belowdisplayskip}{3pt}\begin{align*}&\qquad
f_{\text{W}}(Y)=\frac{n^{\frac{np}2}\1{Y>0}}{2^{\frac{np}2}\Gamma_p\big(\frac{n}2\big)}
\exp\Big\{-\frac{n}2\Tr Y\Big\}|Y|^{\frac{m}2}.
\end{align*}}%
If we do a change of variables $X=\sqrt{n}(Y-I_p)$, so that $Y=I_p+X/\sqrt{n}$ and 
{\setlength{\mathindent}{10pt}
	\setlength{\abovedisplayskip}{3pt}	
	\setlength{\belowdisplayskip}{3pt}\begin{align*}&\qquad
\prod_{i\leq j}^pdY_{ij}=\frac1{n^{\frac{p(p+1)}4}}\prod_{i\leq j}^pdX_{ij},
\end{align*}}%
we see that the normalized Wishart distribution $\sqrt{n}[\text{W}_p(n,I_p/n)-I_p]$ has density
{\setlength{\mathindent}{10pt}
\setlength{\abovedisplayskip}{3pt}	
\setlength{\belowdisplayskip}{3pt}\begin{align}&
f_{\text{NW}}(X)
=\frac{n^{\frac{p(n+m)}4}}{2^{\frac{np}2}\Gamma_p\big(\frac{n}2\big)}
\1{I_p\!+\!\!\frac{X}{\sqrt{n}}\!>\!0}
\notag\\&\hspace{140pt}\cdot
\exp\Big\{-\frac{n}2\Tr\Big[I_p\!+\!\!\frac{X}{\sqrt{n}}\Big]\Big\}
\Big|I_p\!+\!\!\frac{X}{\sqrt{n}}\Big|^{\frac{m}2}.
\label{def:fnw}
\end{align}}%
Notice that $f_{\text{W}}^{1/2}$ is proportional to $\exp\{-\tr{[\frac4nI_p]^{-1}Y}\}|Y|^{\frac{n+p+1}4-\frac{p+1}2}$, so it must be proportional to the density of a matrix gamma distribution $\text{G}_p\left(\frac{n+p+1}4, \frac4nI_p\right)$ when $\frac{n+p+1}4>\frac{p-1}2$, i.e. $n\geq p-2$. In particular, it must be integrable.
As $f_{\text{NW}}$ was obtained by a linear change of variables from $f_{\text{W}}$, $f_{\text{NW}}^{1/2}$ must be integrable too, that is
{\setlength{\mathindent}{20pt}
\setlength{\abovedisplayskip}{3pt}	
\setlength{\belowdisplayskip}{3pt}\begin{align}&
\bigintss_{\S_p(\R)}\hspace{-15pt}
	f_{\text{NW}}^{1/2}(X)dX <\infty.
\label{eqn:fnw12-integrable}
\end{align}}%
 Therefore, we can apply Equation \eqref{eqn:gtransform-explicit} to obtain
{\setlength{\mathindent}{5pt}
\setlength{\abovedisplayskip}{3pt}	
\setlength{\belowdisplayskip}{3pt}\begin{align*}&
\psi^{1/2}_{NW}(T)
= \frac1{2^{\frac{p}2}\pi^{\frac{p(p+1)}4}}
\bigintss_{\S_p(\R)}\hspace{-15pt}
\exp\Big\{\!\!
-i\tr{TX}
\Big\}
f^{1/2}_{\text{NW}}(X) \,dX
\\&
= \frac1{2^{\frac{p}2}\pi^{\frac{p(p+1)}4}}
\E{
\exp\Big\{\!\!
-i\tr{TX}
\Big\}
f^{-1/2}_{\text{NW}}(X)}
\\&
= \frac{2^{\frac{p(n-2)}4}\Gamma^{\frac12}_p\left(\frac{n}2\right)}{\pi^{\frac{p(p+1)}4}n^{\frac{p(n+m)}8}}
\E{
\exp\bigg\{\!\!
-i\tr{TX}
+ \frac{n}4\Tr\Big[I_p\!+\!\!\frac{X}{\sqrt{n}}\Big]\bigg\}
\Big|I_p\!+\!\!\frac{X}{\sqrt{n}}\Big|^{-\frac{m}4}
}.
\end{align*}}%
If we rewrite the expectation in terms of  $Y=I_p+X/\sqrt{n}$, this last expression equals
{\setlength{\mathindent}{5pt}
\setlength{\abovedisplayskip}{3pt}	
\setlength{\belowdisplayskip}{3pt}\begin{align}&
= \frac{2^{\frac{p(n-2)}4}\Gamma^{\frac12}_p\left(\frac{n}2\right)}{\pi^{\frac{p(p+1)}4}n^{\frac{p(n+m)}8}}
\exp\big(i\sqrt{n}\Tr T\big)
\E{
\Etr\!\bigg\{\!\!
\left(\!-i\sqrt{n}T +\!\frac{n}4I_p\right)
\!Y\bigg\}
|Y|^{-\frac{m}4}
\!}.
\label{eq:GNW-sq1}
\end{align}}%
Since $T$ is real symmetric, there must be a spectral decomposition $T=ODO^t$ with $O$ real orthogonal and $D$ real diagonal. As $O^tYO$ has the same distribution as $Y$, namely $\text{W}_p(n,I_p/n)$, we can rewrite Equation \eqref{eq:GNW-sq1} as
{\setlength{\mathindent}{5pt}
\setlength{\abovedisplayskip}{3pt}	
\setlength{\belowdisplayskip}{3pt}\begin{align}&
= \frac{2^{\frac{p(n-2)}4}\Gamma^{\frac12}_p\left(\frac{n}2\right)}{\pi^{\frac{p(p+1)}4}n^{\frac{p(n+m)}8}}
\exp\big(i\sqrt{n}\Tr T\big)
\E{
	\Etr\!\bigg\{\!\!
	\left(\!-i\sqrt{n}D +\!\frac{n}4I_p\right)
	\!Y\bigg\}
	|Y|^{-\frac{m}4}
\!}.
\label{eq:GNW-sq2}
\end{align}}%
Now, since $Y$ is positive-definite it has a Cholesky decomposition $Y=U^tU$ with $U$ upper-triangular. According to Bartlett's theorem \citep[see][Theorem 3.2.14]{muirhead82}, all the elements of $U$ are independent, the diagonal elements have the distribution $U_{kk}^2\sim\chi_{n-k+1}^2/n$ and the upper diagonal elements have $U_{kl}\sim\text{N}(0,1/n)$ for $k<l$. Since
{\setlength{\mathindent}{5pt}
\setlength{\abovedisplayskip}{3pt}	
\setlength{\belowdisplayskip}{3pt}\begin{align*}
\Tr\bigg[\Big(-i\sqrt{n}D+\frac{I_p}4\Big)Y\bigg]
&=\sum_{j,k,l}^p\Big(-i\sqrt{n}D+\frac{n}4I_p\Big)_{jk}U'_{kl}U_{lj}
\\&
=\sum_{l\leq k}^p\Big(-i\sqrt{n}D_{kk}+\frac{n}4\Big)U^2_{lk}
\end{align*}}%
and $|Y|=\prod_{k=1}^p U^2_{kk}$, we have by independence and Equation \eqref{eq:GNW-sq2} that
{\setlength{\mathindent}{5pt}
\setlength{\abovedisplayskip}{3pt}	
\setlength{\belowdisplayskip}{3pt}\begin{align}&
	= \frac{2^{\frac{p(n-2)}4}\Gamma^{\frac12}_p\left(\frac{n}2\right)}{\pi^{\frac{p(p+1)}4}n^{\frac{p(n+m)}8}}
\exp\Bigg\{i\sqrt{n}\Tr T\Bigg\}
\prod_{l< k}^p
\E{ \exp\bigg\{\Big(-i\sqrt{n}D_{kk}+\frac{n}4\Big)U^2_{lk}\bigg\} }
\notag\\&\hspace{90pt}\cdot
\prod_{k=1}^p\E{
	\exp\bigg\{\Big(-i\sqrt{n}D_{kk}+\frac{n}4\Big)U^2_{kk}\bigg\}
	(U_{kk}^2)^{-\frac{m}4}
}.
\label{eq:GNW-sq3}
\end{align}}%
We will now compute these expected values in several steps. For a given $1\leq k\leq p$, let
{\setlength{\mathindent}{5pt}
	\setlength{\abovedisplayskip}{3pt}	
	\setlength{\belowdisplayskip}{3pt}\begin{align}&\qquad
\circled{A}=\E{\exp\bigg\{\Big(-i\sqrt{n}D_{kk}+\frac{n}4\Big)T^2_{kk}\bigg\} T_{kk}^{-\frac{m}2}}.
\label{eq:GNW-defA}
\end{align}}%
Since $T^2_{kk}\sim\chi^2_{n-k+1}/n$ and $m=n-p-1$, we have
{\setlength{\mathindent}{5pt}
\setlength{\abovedisplayskip}{3pt}	
\setlength{\belowdisplayskip}{3pt}\begin{align}&
	\circled{A}
	=\frac{n^{\frac{n-k+1}2}}{2^{\frac{n-k+1}2}\Gamma\Big(\frac{n-k+1}2\Big)}\bigintsss_{\,0}^\infty\hspace{-7pt}
	\exp\bigg\{\Big(-i\sqrt{n}D_{kk}+\frac{n}4\Big)x\bigg\}
	x^{-\frac{m}4}
	\notag\\&\hspace{200pt}
	\cdot
	x^{\frac{n-k+1}2-1}\exp\Big\{-\frac{n}2x\Big\}
	dx
	\notag\\&\hspace{17pt}
	=\frac{n^{\frac{n-k+1}2}}{2^{\frac{n-k+1}2}\Gamma\Big(\frac{n-k+1}2\Big)}\bigintsss_{\,0}^\infty\hspace{-7pt}
	\exp\bigg\{\!\!\!-\!\Big(\frac{n}4+\!\sqrt{n}D_{kk}i\Big)x\bigg\}
	x^{\frac{n-2k+p+3}4-1}
	dx
	\label{eq:GNW-A1}
	\end{align}}%
Consider the truncated integrands
{\setlength{\mathindent}{5pt}
\setlength{\abovedisplayskip}{3pt}	
\setlength{\belowdisplayskip}{3pt}\begin{align*}&\qquad
h_M(x)
=
\exp\bigg\{\!\!\!-\!\Big(\frac{n}4+\sqrt{n}D_{kk}i\Big)x\bigg\}
x^{\frac{n-2k+p+3}4-1}
\1{0<x<M}.
\end{align*}}%
Clearly this sequence is dominated by the integrable positive function $h$,
{\setlength{\mathindent}{5pt}
\setlength{\abovedisplayskip}{3pt}	
\setlength{\belowdisplayskip}{3pt}\begin{align*}&\qquad
|h_M(x)|\leq h(x)=\exp\bigg\{-\frac{n}4x\bigg\}
x^{\frac{n-2k+p+3}4-1},
\\&
\int_0^\infty h(x)dx=\left(\frac{4}{n}\right)^{\frac{n-2k+p+3}4}
\Gamma\left(\frac{n-2k+p+3}4\right)<\infty.
\end{align*}}%
Therefore, by the Dominated Convergence Theorem and Equation \eqref{eq:GNW-A1},
{\setlength{\mathindent}{5pt}
\setlength{\abovedisplayskip}{3pt}	
\setlength{\belowdisplayskip}{3pt}\begin{align*}&
\circled{A}=\frac{n^{\frac{n-k+1}2}}{2^{\frac{n-k+1}2}\Gamma\Big(\frac{n-k+1}2\Big)}
\lim_{M\rightarrow\infty}
\bigintsss_{\,0}^M\hspace{-7pt}
\exp\bigg\{\!\!\!-\!\Big(\frac{n}4+\sqrt{n}D_{kk}i\Big)x\bigg\}
x^{\frac{n-2k+p+3}4-1}
\,dx.
\end{align*}}%
By the change of variables $z=\big(\frac{n}4+\sqrt{n}D_{kk}i\big)x$, this can be rewritten
{\setlength{\mathindent}{5pt}
\setlength{\abovedisplayskip}{3pt}	
\setlength{\belowdisplayskip}{3pt}\begin{align}&\quad
=\frac{n^{\frac{n-k+1}2}\big(\frac{n}4+\sqrt{n}D_{kk}i\big)^{-\frac{n-2k+p+3}4}}{2^{\frac{n-k+1}2}\Gamma\Big(\frac{n-k+1}2\Big)}
\lim_{M\rightarrow\infty}
\bigintss_{\,0}^{\frac{nM}4+\sqrt{n}D_{kk}Mi}
\hspace{-60pt}
e^{-z}
z^{\frac{n-2k+p+3}4-1}
\,dz.
\label{eq:GNW-A2}
\end{align}}%
To compute this integral, we use a contour argument. Consider the closed path $C=C_1+C_2+C_3$ given by $C_1$ a path from $0$ to $\frac{nM}4$, $C_2$ a path from $\frac{nM}4$ to $\frac{nM}4+\sqrt{n}D_{kk}Mi$ and finally $C_3$ a path from $\frac{nM}4+\sqrt{n}D_{kk}Mi$ to 0. A diagram is provided as Figure \ref{fig:GNW-cont1}.

\begin{figure}[t]
\centering
\begin{tikzpicture}

\draw [help lines,->] (-1, 0) -- (4,0);
\node at (4,-0.33){$x$};
\draw [help lines,->] (0, -1) -- (0, 3.5);
\node at (-0.33,3.5) {$y$};

\draw[line width=1pt,   
decoration={markings,
	mark = at position 0.1465 with \arrow{>},
	mark = at position 0.4393 with \arrow{>},
	mark = at position 0.7989 with \arrow{>}
},
postaction={decorate}]
(0,0) -- (3,0) -- (3,3) -- (0,0);

\node at (-0.2,-0.2) {$\scriptstyle 0$};
\node at (1.5,0.4) {$C_1$};
\node at (3,-0.3) {$\scriptstyle \frac{nM}4$};
\node at (3.4,1.5) {$C_2$};
\node at (3,3.3) {$\scriptstyle \frac{nM}4 +\sqrt{n}D_{kk}Mi$};
\node at (1.2172,1.7828) {$C_3$};
\end{tikzpicture}
\caption{Contour $C=C_1+C_2+C_3$ when $D_{kk}\geq0$. The diagram is mirrored around the $x$ axis when $D_{kk}<0$.}
\label{fig:GNW-cont1}
\end{figure}
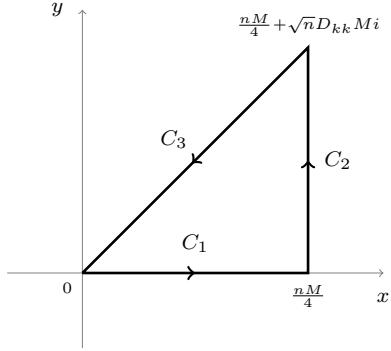

As $k\leq p$, $z\mapsto e^{-z}z^{\frac{n-2k+p+3}4-1}$ is entire and its integral over $C$ must be zero. Therefore
{\setlength{\mathindent}{5pt}
\setlength{\abovedisplayskip}{3pt}	
\setlength{\belowdisplayskip}{3pt}\begin{align*}&
\left|
\lim_{M\rightarrow\infty}\bigints_0^{\frac{nM}4+\sqrt{n}D_{kk}Mi}
\hspace{-60pt}
e^{-z}
z^{\frac{n-2k+p+3}4-1}
dz
-
\Gamma\left(\frac{n-2k+p+3}4\right)
\right|
\\&\qquad
=\lim_{M\rightarrow\infty}\left|
\bigints_{\,0}^{\frac{nM}4+\sqrt{n}D_{kk}Mi}
\hspace{-60pt}
e^{-z}
z^{\frac{n-2k+p+3}4-1}
dz
-
\bigints_{\,0}^{\frac{nM}4}
\hspace{-10pt}
e^{-x}
x^{\frac{n-2k+p+3}4-1}
dx
\right|
\\&\qquad
=\lim_{M\rightarrow\infty}\left|
\bigints_{C_2}
\hspace{-5pt}
e^{-z}
z^{\frac{n-2k+p+3}4-1}
dz
\right|
=\lim_{M\rightarrow\infty}\left|
\bigints_{\,\frac{nM}4}^{\frac{nM}4+\sqrt{n}D_{kk}Mi}
\hspace{-60pt}
e^{-z}
z^{\frac{n-2k+p+3}4-1}
dz
\right|.
\end{align*}}%
Do a change of variables $z=\frac{nM}4+y\sqrt{n}D_{kk}Mi$, so that $y=\frac{z-\frac{nM}4}{\sqrt{n}D_{kk}Mi}$ is real on the path. It yields
{\setlength{\mathindent}{5pt}
\setlength{\abovedisplayskip}{3pt}	
\setlength{\belowdisplayskip}{3pt}\begin{align*}&\quad
=\sqrt{n}|D_{kk}|\lim_{M\rightarrow\infty}
\frac{M}{e^{\frac{nM}4}}
\left|
\bigints_{\;0}^1
\hspace{-7pt}
e^{-y\sqrt{n}D_{kk}Mi}
\bigg[\!\frac{nM}4 +\!y\sqrt{n}D_{kk}Mi\bigg]^{\frac{n-2k+p+3}4-1}
\!\!dy
\right|
\\&\quad
\leq
\sqrt{n}|D_{kk}|
\bigints_{\;0}^1
\hspace{-7pt}
\bigg|\frac{n}4+y\sqrt{n}D_{kk}i\bigg|^{\frac{n-2k+p+3}4-1}
dy
\cdot
\lim_{M\rightarrow\infty}
\frac{M^{\frac{n-2k+p+3}4}}{e^{\frac{nM}4}}.
\end{align*}}%
This last integral is finite, since it is continuous on a bounded interval. Therefore the limit is zero and by Equation \eqref{eq:GNW-A2} and the previous expression,
{\setlength{\mathindent}{5pt}
\setlength{\abovedisplayskip}{3pt}	
\setlength{\belowdisplayskip}{3pt}\begin{align}&\qquad
\circled{A}=\frac{n^{\frac{n-k+1}2}\Gamma\Big(\frac{n-2k+p+3}4\Big)}{2^{\frac{n-k+1}2}\Gamma\Big(\frac{n-k+1}2\Big)}
\Big(\frac{n}4+\sqrt{n}D_{kk}i\Big)^{-\frac{n-2k+p+3}4}.
\label{eq:GNW-A3}
\end{align}}%

Going back to \eqref{eq:GNW-sq3}, let us now consider the expectations in the second products. For fixed $1\leq l<k\leq p$, let 
{\setlength{\mathindent}{5pt}
\setlength{\abovedisplayskip}{3pt}	
\setlength{\belowdisplayskip}{3pt}\begin{align}&\qquad
\circled{B}=\E{\exp\bigg\{\Big(-i\sqrt{n}D_{kk}+\frac{n}4\Big)T^2_{lk}\bigg\}}.
\label{eq:GNW-defB}
\end{align}}%
Since $T_{lk}^2\sim\chi^2_1/n,$
{\setlength{\mathindent}{5pt}
\setlength{\abovedisplayskip}{3pt}	
\setlength{\belowdisplayskip}{3pt}\begin{align}&\qquad
\circled{B}
=\sqrt{\frac{n}{2\pi}}
\bigintss_{\,0}^\infty\hspace{-7pt}
\exp\bigg\{\Big(-i\sqrt{n}D_{kk}+\frac{n}4\Big)x\bigg\}
\frac{e^{-\frac{n}2x}}{\sqrt{x}}dx
\notag\\&\qquad
=\sqrt{\frac{n}{2\pi}}
\bigintss_{\,0}^\infty\hspace{-7pt}
\exp\bigg\{\!\!-\!\Big(\frac{n}4 +i\sqrt{n}D_{kk}\Big)x\bigg\}\frac1{\sqrt{x}}dx.
\label{eq:GNW-B1}
\end{align}}%
Consider the truncated integrands
{\setlength{\mathindent}{5pt}
	\setlength{\abovedisplayskip}{3pt}	
	\setlength{\belowdisplayskip}{3pt}\begin{align}&\qquad
h_M(x)=\exp\bigg\{\!\!-\!\Big(\frac{n}4 +i\sqrt{n}D_{kk}\Big)x\bigg\}
\frac1{\sqrt{x}}\1{\frac1M<x<M}.
\end{align}}%
We see that they are dominated by a positive, integrable function $h(x)$,
{\setlength{\mathindent}{5pt}
	\setlength{\abovedisplayskip}{3pt}	
	\setlength{\belowdisplayskip}{3pt}\begin{align*}&\qquad
|h_M(x)|\leq h(x)=\frac{e^{-\frac{n}4x}}{\sqrt{x}},
\qquad \int_0^\infty h(x)dx=2\sqrt{\frac{\pi}n}<\infty.
\end{align*}}%
Therefore, by the Dominated Convergence Theorem and Equation \eqref{eq:GNW-B1}, we conclude that
{\setlength{\mathindent}{5pt}
\setlength{\abovedisplayskip}{3pt}	
\setlength{\belowdisplayskip}{3pt}\begin{align*}&
\circled{B}=\sqrt{\frac{n}{2\pi}}
\lim_{M\rightarrow\infty}\bigintss_{\,1/M}^M\hspace{-10pt}
\exp\bigg\{\!\!-\!\Big(\frac{n}4+i\sqrt{n}D_{kk}\Big)x\bigg\}\frac1{\sqrt{x}} \,dx.
\end{align*}}%
A complex change of variables $z=\big(\frac{n}4+i\sqrt{n}D_{kk}\big)x$ yields 
{\setlength{\mathindent}{5pt}
\setlength{\abovedisplayskip}{3pt}	
\setlength{\belowdisplayskip}{3pt}\begin{align}&\qquad
=\sqrt{\frac{n}{2\pi}}\Big(\frac{n}4+i\sqrt{n}D_{kk}\Big)^{-\frac12}
\lim_{M\rightarrow\infty}
\bigints_{\,\frac{n}{4M}+\frac{\sqrt{n}D_{kk}}{M}i}^{\frac{nM}4+\sqrt{n}D_{kk}Mi}\hspace{-50pt}
e^{-z}/\sqrt{z}
\,dz.
\label{eq:GNW-B2}
\end{align}}%
Let's compute this integral again using a contour integration argument. Consider the contour $C=C_1+C_2+C_3+C_4$ given by $C_1$ a line from $\frac{n}{4M}$ to $\frac{nM}{4}$, $C_2$ a line from $\frac{nM}{4}$ to $\frac{nM}4+\sqrt{n}D_{kk}Mi$, $C_3$ a line from $\frac{nM}4+\sqrt{n}D_{kk}Mi$ to $\frac{n}{4M}+\frac{\sqrt{n}D_{kk}}{M}i$ and $C_4$ a line from $\frac{n}{4M}+\frac{\sqrt{n}D_{kk}}{M}i$ to $\frac{n}{4M}$. A diagram is provided as Figure \ref{fig:GNW-cont2}.

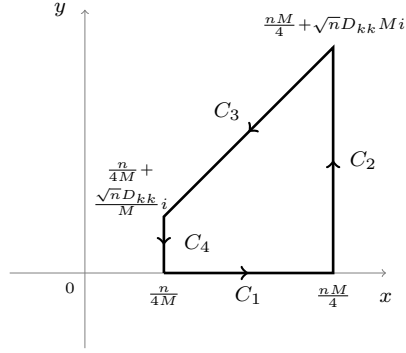
\begin{figure}[t]
\centering
\begin{tikzpicture}

\draw [help lines,->] (-1, 0) -- (4,0);
\node at (4,-0.33){$x$};
\draw [help lines,->] (-0, -1) -- (-0, 3.5);
\node at (-0.33,3.5) {$y$};

\draw[line width=1pt,   
decoration={markings,
	mark = at position 0.1225 with \arrow{>},
	mark = at position 0.4084 with \arrow{>},
	mark = at position 0.7450 with \arrow{>},
	mark = at position 0.9592 with \arrow{>}
},
postaction={decorate}]
(1.05,0) -- (3.3,0) -- (3.3,3) -- (1.05,0.75) -- (1.05,0);

\node at (-0.2,-0.2) 		{$\scriptstyle 0$};
\node at (1.05,-0.3) 		{$\scriptstyle \frac{n}{4M}$};
\node at (2.175,-0.3) 		{$C_1$};
\node at (3.3,-0.3) 			{$\scriptstyle \frac{nM}4$};
\node at (3.7,1.5) 		{$C_2$};
\node at (3.3,3.3) 			{$\scriptstyle \frac{nM}4 +\sqrt{n}D_{kk}Mi$};
\node at (1.8922,2.1578) 	{$C_3$};
\node at (0.6,1.35) 		{$\scriptstyle \frac{n}{4M}+$};
\node at (0.6,0.95) 		{$\scriptstyle \frac{\sqrt{n}D_{kk}}{M}i$};
\node at (1.5,0.4) 	{$C_4$};
\end{tikzpicture}
\caption{Contour $C=C_1+C_2+C_3+C_4$ when $D_{kk}\geq0$. The diagram is mirrored around the $x$ axis when $D_{kk}<0$.}
\label{fig:GNW-cont2}
\end{figure}

Since $z\mapsto e^{-z}/\sqrt{z}$ is holomorphic away from zero,
{\setlength{\mathindent}{5pt}
\setlength{\abovedisplayskip}{3pt}	
\setlength{\belowdisplayskip}{3pt}\begin{align*}&
\left|\lim_{M\rightarrow\infty}
\bigints_{\frac{n}{4M}+\frac{\sqrt{n}D_{kk}}{M}i}^{\frac{nM}4+\sqrt{n}D_{kk}Mi}
\hspace{-60pt}
e^{-z}/\sqrt{z}\,dz
\;-\;
\sqrt{\pi}\;
\right|
=\lim_{M\rightarrow\infty}\left|
\bigints_{\frac{n}{4M}+\frac{\sqrt{n}D_{kk}}{M}i}^{\frac{nM}4+\sqrt{n}D_{kk}Mi}
\hspace{-60pt}
e^{-z}/\sqrt{z}\,dz
-
\bigints_{\frac{n}{4M}}^{\frac{nM}4}
\hspace{-15pt}
e^{-x}/\sqrt{x}\,dx
\right|
\\&\qquad
=\lim_{M\rightarrow\infty}\left|
\bigintss_{C_2}\hspace{-7pt}
e^{-z}/\sqrt{z}\,dz
+\bigintss_{C_4}\hspace{-7pt}
e^{-z}/\sqrt{z}\,dz
\right|
\\&\qquad
\leq\lim_{M\rightarrow\infty}
\left|\bigints_{\,\frac{nM}4}^{\frac{nM}4+\sqrt{n}D_{kk}Mi}
\hspace{-60pt}
e^{-z}/\sqrt{z}\,dz
\quad
\right|
+\left|\bigints_{\frac{n}{4M}+\frac{\sqrt{n}D_{kk}}{M}i}^{\frac{n}{4M}}
\hspace{-40pt}
e^{-z}/\sqrt{z}\,dz
\;\;
\right|.
\end{align*}}%
By changes of variables $z=\frac{nM}4 +y\sqrt{n}D_{kk}Mi$ and $z=\frac{n}{4M}+y\frac{\sqrt{n}D_{kk}}{M}i$ on the two respective integrals, we get
{\setlength{\mathindent}{5pt}
\setlength{\abovedisplayskip}{3pt}	
\setlength{\belowdisplayskip}{3pt}\begin{align*}&\qquad
=\sqrt{n}|D_{kk}|\lim_{M\rightarrow\infty}
\frac{M}{e^{\frac{nM}4}}
\left|\bigintss_{\,0}^1
\frac{e^{y\sqrt{n}D_{kk}Mi}}{\Big[\frac{nM}4 +y\sqrt{n}D_{kk}Mi\Big]^{\frac12}}
dy\;
\right|
\\&\hspace{40pt}
+\sqrt{n}|D_{kk}|\lim_{M\rightarrow\infty}
\frac1{Me^{\frac{n}{4M}}}
\left|\bigintss_{\,0}^1
\frac{e^{y\frac{\sqrt{n}D_{kk}}{M}i}}{\Big[\frac{n}{4M} +y\frac{\sqrt{n}D_{kk}}{M}i\Big]^{\frac12}}
dy\;
\right|
\\&\qquad
\leq\sqrt{n}|D_{kk}|
\bigintss_{\,0}^1\frac1{\Big|\frac{n}4 +y\sqrt{n}D_{kk}i\Big|^{\frac12}}dy\;
\cdot \lim_{M\rightarrow\infty}
\frac{\sqrt{M}}{e^{\frac{nM}4}}
\\&\hspace{40pt}
+\sqrt{n}|D_{kk}|
\bigintss_{\,0}^1\frac1{\Big|\frac{n}{4} +y\sqrt{n}D_{kk}i\Big|^{\frac12}}dy\;\cdot
\lim_{M\rightarrow\infty}
\frac1{\sqrt{M}e^{\frac{n}{4M}}}.
\end{align*}}%
Since $\Big|\frac{n}{4} +y\sqrt{n}D_{kk}i\Big|^{-\frac12}=\Big(\frac{n^2}{16}+y^2nD^2_{kk}\Big)^{-\frac14}$ is continuous on $[0,1]$, a bounded interval, we conclude that the integrals are finite and that the limits are zero. Therefore, by Equation \eqref{eq:GNW-B2},
{\setlength{\mathindent}{5pt}
\setlength{\abovedisplayskip}{3pt}	
\setlength{\belowdisplayskip}{3pt}\begin{align}&\qquad
\circled{B}=\sqrt{\frac{n}{2}}\Big(\frac{n}4 +i\sqrt{n}D_{kk}\Big)^{-\frac12}.
\label{eq:GNW-B3}
\end{align}}%
Recall the definitions of $\circled{A}$ and $\circled{B}$ at Equations \eqref{eq:GNW-defA} and \eqref{eq:GNW-defB}. Combining both Equations \eqref{eq:GNW-A3} and \eqref{eq:GNW-B3} into the expression for $\psi^{1/2}_{\text{NW}}$ at Equation \eqref{eq:GNW-sq3} provides
{\setlength{\mathindent}{5pt}
\setlength{\abovedisplayskip}{3pt}	
\setlength{\belowdisplayskip}{3pt}\begin{align*}&
\psi^{1/2}_\text{NW}(T)
= \frac{2^{\frac{p(n-2)}4}\Gamma^{\frac12}_p\left(\frac{n}2\right)}{\pi^{\frac{p(p+1)}4}n^{\frac{p(n+m)}8}}
\exp\bigg\{i\sqrt{n}\Tr T\bigg\}
\prod_{l< k}^p
\sqrt{\frac{n}{2}}\Big(\frac{n}4 +i\sqrt{n}D_{kk}\Big)^{-\frac12}
\\&\hspace{100pt}\cdot
\prod_{k=1}^p
\frac{n^{\frac{n-k+1}2}\Gamma\Big(\frac{n-2k+p+3}4\Big)}{2^{\frac{n-k+1}2}\Gamma\Big(\frac{n-k+1}2\Big)}
\Big(\frac{n}4 +\sqrt{n}D_{kk}i\Big)^{-\frac{n-2k+p+3}4}
\\&\quad
= \frac{2^{\frac{p(n-2)}4}\Gamma^{\frac12}_p\left(\frac{n}2\right)}{\pi^{\frac{p(p+1)}4}n^{\frac{p(n+m)}8}}
\prod_{k=1}^p
\frac{n^{\frac{n-k+1}2+\frac{k-1}2}}{2^{\frac{n-k+1}2+\frac{k-1}2}}
\prod_{k=1}^p
\frac{\Gamma\left(\frac{n-2k+p+3}4\right)}{\Gamma\left(\frac{n-k+1}2\right)}
\exp\bigg\{i\sqrt{n}\Tr T\bigg\}
\\&\hspace{100pt}\cdot
\prod_{k=1}^p
\Big(\frac{n}4-\sqrt{n}D_{kk}i\Big)^{-\frac{n-2k+p+3}4-\frac{k-1}2}
\\&\quad
= \frac{n^{\frac{p(2n+p+1)}8}}{2^{\frac{p(n+2)}4}\pi^{\frac{p(p+1)}4}}
\prod_{k=1}^p \frac{\Gamma\left(\frac{n-2k+p+3}4\right)}{\Gamma\left(\frac{n-k+1}2\right)}
\exp\bigg\{i\sqrt{n}\Tr T\bigg\}
\bigg|\frac{n}4I_p+i\sqrt{n}T\bigg|^{-\frac{n+p+1}4}\!\!.
\end{align*}}%
But by \citet[Theorem 2.1.12]{muirhead82},
{\setlength{\mathindent}{5pt}
\setlength{\abovedisplayskip}{3pt}	
\setlength{\belowdisplayskip}{3pt}\begin{align*}&
\prod_{k=1}^p\frac{\Gamma\left(\frac{n-2k+p+3}4\right)}{\Gamma\left(\frac{n-k+1}2\right)}
=\frac{\pi^{\frac{p(p-1)}4}\prod\limits_{k=1}^p\Gamma\left(\frac{n+p+1}4+\frac{1-k}2\right)}
{\pi^{\frac{p(p-1)}4}\prod\limits_{k=1}^p\Gamma\left(\frac{n}2+\frac{1-k}2\right)}
=\frac{\Gamma_p\left(\frac{n+p+1}4\right)}{\Gamma_p\left(\frac{n}2\right)},
\end{align*}}%
so by taking a $n/4$ factor out of the determinant, we find that
{\setlength{\mathindent}{5pt}
\setlength{\abovedisplayskip}{3pt}	
\setlength{\belowdisplayskip}{3pt}\begin{align*}&\quad
\psi^{1/2}_\text{NW}(T)=\frac{2^{\frac{p(n+2p)}4}}{\pi^{\frac{p(p+1)}4}n^{\frac{p(p+1)}8}}
\frac{\Gamma_p\left(\frac{n+p+1}4\right)}{\Gamma^{1/2}_p\left(\frac{n}2\right)}
\exp\bigg\{i\sqrt{n}\Tr T\bigg\}
\bigg|I_p +i\frac{4 T}{\sqrt{n}}\bigg|^{-\frac{n+p+1}4}\!\!.
\end{align*}}%
Squaring this result yields the desired expression for $\psi_{\text{NW}}$.
\end{proof}

By Proposition \ref{prop:gtransform-nw}, when $n\geq p-2$ the G-conjugate of a normalized Wishart distribution must have a density on $\S_p(\R)$ given by
{\setlength{\mathindent}{5pt}
\setlength{\abovedisplayskip}{3pt}	
\setlength{\belowdisplayskip}{3pt}\begin{align}\label{eqn:gconjugate-nw} &\quad
|\psi_{\text{NW}}|(T)
=\frac{2^{\frac{p(n+2p)}2}}{\pi^{\frac{p(p+1)}2}n^{\frac{p(p+1)}4}}
\frac{\Gamma^2_p\left(\frac{n+p+1}4\right)}{\Gamma_p\left(\frac{n}2\right)}
\bigg|I_p +\frac{16T^2}{n}\bigg|^{-\frac{n+p+1}4}\!\!.
\end{align}}
As mentioned in the paragraph following Equation \eqref{eqn:gconjugate}, the G-conjugate of a $\chi^2_n/n$ distribution is a scaled $t_{n/2}$. Thus, by analogy, Equation \eqref{eqn:gconjugate-nw} should be represent some kind of generalization of the $t$ distribution to the real symmetric matrices. Matrix-variate generalizations of the $t$ distribution have been investigated in the past, but not for symmetric matrices. Hence it appears the concept is new.

This motivates us to propose in Section \ref{sec:symmetrict} a candidate for a symmetric matrix variate $t$ distribution. Using that definition, the G-conjugate to the normalized Wishart could then be regarded as the $t$ distribution with $n/2$ degrees of freedom and scale matrix $I_p/8$, which we denote $T_{n/2}(I_p/8)$. But regardless of its name, this distribution will play a key role in our results about the middle-scale regime asymptotics of Wishart matrices, and will be investigated in depth in Section \ref{sec:symmetrict}.

\section{The symmetric matrix variate $t$ distribution}\label{sec:symmetrict}

In Section \ref{sec:gtransforms}, Equation \eqref{eqn:gconjugate-nw}, we proved that when $n\geq p-2$, the G-conjugate of the normalized Wishart distribution $\sqrt{n}[\text{W}_p(n,I_p/n)-I_p]$ has density on $\S_p(\R)$ given by
{\setlength{\mathindent}{10pt}
\setlength{\abovedisplayskip}{3pt}	
\setlength{\belowdisplayskip}{5pt}\begin{align}\label{eqn:gconjugate-nw2} &\qquad
|\psi_{\text{NW}}|(T)
=\frac{2^{\frac{p(n+2p)}2}}{\pi^{\frac{p(p+1)}2}n^{\frac{p(p+1)}4}}
\frac{\Gamma^2_p\left(\frac{n+p+1}4\right)}{\Gamma_p\left(\frac{n}2\right)}
\bigg|I_p +\frac{16T^2}{n}\bigg|^{-\frac{n+p+1}4}\!\!.
\end{align}}%
Two remarks are in order. First, we are unaware of any matrix calculus tools that could let us integrate this expression directly. Thus, the mere fact that this expression integrates to unity, a consequence of being the G-conjugate of another distribution, seems remarkable. 

Second, when $p=1$, this is the $t_{n/2}/\sqrt{8}$ distribution. Thus, as we mentioned while discussing Equation \eqref{eqn:gconjugate-nw}, it is natural to interpret this distribution as the parametrization of some generalization of the $t$ distribution to $\S_p(\R)$, the space of real-valued symmetric matrices. The purpose of this section is to propose a candidate definition for such generalization, as well as prove several results concerning the normalized Wishart G-conjugate.

To the best of our knowledge, no extension of the $t$ distribution to symmetric matrices has ever been proposed. However, a non-symmetric matrix variate $t$ distribution has been thoroughly investigated in the literature -- see \citet[Chapter 4]{gupta99} for a thorough summary. Several definitions exist. For our purposes, we say that a $p\times q$ real-valued random matrix $T$ has the matrix variate $t$ distribution with $\nu$ degrees of freedom and $q\times q$ positive-definite scale matrix $\Omega$ if it has density
{\setlength{\mathindent}{10pt}
\setlength{\abovedisplayskip}{3pt}	
\setlength{\belowdisplayskip}{3pt}\begin{align*} &\hspace{45pt}
\frac{\Gamma_p\Big(\frac{\nu+p+q-1}2\Big)}
	{\nu^{\frac{pq}2}\pi^{\frac{pq}2} \Gamma_p\Big(\frac{\nu+p-1}2\Big)}
|\Omega|^{-\frac{p}2}
\bigg| I_p + \frac{T\Omega^{-1}T^t}{\nu}\bigg|^{-\frac{\nu+p+q-1}{2}}.
\end{align*}}%
It is not exactly clear what should be the proper analog of this distribution for symmetric matrices. But it would be elegant if the degrees of freedom of Equation \eqref{eqn:gconjugate-nw2} were to be exactly $n/2$, as in the univariate case. Thus, the following definition seems natural.

\begin{definition}[Symmetric matrix variate $t$ distribution]\label{def:tdistribution} We say a real symmetric $p\times p$ matrix $T$ has the symmetric matrix variate $t$ distribution with $\nu\geq p/2 -1$ degrees of freedom and $p\times p$ positive-definite scale matrix $\Omega$, denoted $T_\nu(\Omega)$, if it has density
{\setlength{\mathindent}{10pt}
\setlength{\abovedisplayskip}{3pt}	
\setlength{\belowdisplayskip}{3pt}\begin{align*} &\hspace{45pt}
f_{T_{n}(\Omega)}(T) \quad\propto\quad
\bigg|I_p +\frac{T\Omega^{-1}T}{\nu}\bigg|^{-\frac{\nu+(p+1)/2}2}\!\!.
\end{align*}}%
\end{definition}
With this definition, the G-conjugate to the normalized Wishart distribution, whose density is given by Equation \eqref{eqn:gconjugate-nw2}, is the $T_{n/2}(I_p/8)$ distribution on $\S_p(\R)$. 

In fact, since Equation \eqref{eqn:gconjugate-nw2} integrates to one, we can deduce the normalization constant of Definition \ref{def:tdistribution}. For an arbitrary degrees of freedom parameter $\nu$, imagine the density $|\psi_{\text{NW}}|$ of the G-conjugate of a normalized Wishart distribution with $n=2\nu\geq p-2$. By virtue of being a G-conjugate, it must integrate to unity. Then from the change of variables $T=\Omega^{-\frac14}S\Omega^{-\frac14}/\sqrt{8}$ which has Jacobian $dT=8^{-\frac{p(p+1)}4}|\Omega|^{-\frac{p+1}4}dS$, we see that
{\setlength{\mathindent}{10pt}
\setlength{\abovedisplayskip}{3pt}	
\setlength{\belowdisplayskip}{5pt}\begin{align*}\qquad
1 &= \bigintss_{\S_p(\R)}\hspace{-20pt}
|\psi_{\text{NW}}|(T)
\,dT
= \frac1{2^{\frac{3p(p+1)}4}|\Omega|^{\frac{p+1}4}}
\bigintss_{\S_p(\R)}\hspace{-20pt}
|\psi_{\text{NW}}| \bigg(\frac{\Omega^{-\frac14}S\Omega^{-\frac14}}{\sqrt{8}}\bigg)
dS
\\&
=
\frac{2^{p(\nu-1)} \Gamma^2_p\left(\frac{\nu+(p+1)/2}2\right)}
{\pi^{\frac{p(p+1)}2} \nu^{\frac{p(p+1)}4} \Gamma_p\left(\nu\right)}
|\Omega|^{-\frac{p+1}4}
\hspace{-5pt}\bigints_{\S_p(\R)}\hspace{-20pt}
\bigg|I_p +\frac{S\Omega^{-1}S}{\nu}\bigg|^{-\frac{\nu+(p+1)/2}2}
\!\!dS.
\end{align*}}%
Thus we must have
{\setlength{\mathindent}{10pt}
\setlength{\abovedisplayskip}{3pt}	
\setlength{\belowdisplayskip}{3pt}\begin{align}&\label{eqn:tdistribution-density}\quad
f_{T_{n}(\Omega)}(T) 
=
\frac{2^{p(\nu-1)} \Gamma^2_p\left(\frac{\nu+(p+1)/2}2\right)}
	{\pi^{\frac{p(p+1)}2} \nu^{\frac{p(p+1)}4} \Gamma_p\left(\nu\right)}
|\Omega|^{-\frac{p+1}4}
\bigg|I_p +\frac{T\Omega^{-1}T}{\nu}\bigg|^{-\frac{\nu+(p+1)/2}2}\!\!\!.
\end{align}}%
It would be interesting to see if this distribution satisfies the properties we would expect of a $t$ distribution, to ensure our guess is the ``correct'' one. However, this would take us too far away from the topic of this article.
Instead, we will focus in the rest of this section on proving results about $T_{n/2}(I_p/8)$, the G-conjugate to the normalized Wishart distribution.

Our first result will concern the asymptotic expansion of its normalization constant. We mention that this constant is the same as the $C_{n,p}$ term appearing in the expression of the G-transform of the normalized Wishart in Proposition \ref{prop:gtransform-nw}.

\begin{lemma}\label{lem:cnp}
The normalization constant of the $T_{n/2}(I_p/8)$ distribution
{\setlength{\mathindent}{10pt}
	\setlength{\abovedisplayskip}{0pt}
	\setlength{\belowdisplayskip}{3pt}\begin{align}&\quad
	C_{n, p}
	=\frac{2^{\frac{p(n+2p)}2}}{\pi^{\frac{p(p+1)}2}n^{\frac{p(p+1)}4}}
	\frac{\Gamma^2_p\left(\frac{n+p+1}4\right)}{\Gamma_p\left(\frac{n}2\right)}
	\label{eq:lem-cnp-defcnp}
	\end{align}}%
has, for every $K\in\N$, the asymptotic expansion
{\setlength{\mathindent}{10pt}
	\setlength{\abovedisplayskip}{3pt}
	\setlength{\belowdisplayskip}{0pt}\begin{align*}&\quad
	C_{n, p}
	=
	\frac{2^{\frac{p(3p+1)}4}}{\pi^{\frac{p(p+1)}4}}
	\exp\bigg\{\!\!
	-\!\frac12\!\sum_{k=1}^{K+1}\!\frac{\keven}{k(k\!+\!1)(k\!+\!2)}\frac{p^{k+2}}{n^k}
	\notag\\&\hspace{150pt}
	-\!\frac14\!\sum_{k=1}^{K+1}\!\frac{1\!+\!2\keven}{k(k\!+\!1)}\frac{p^{k+1}}{n^k}
	+ o\Big(\frac{p^{K+3}}{n^{K+1}}\Big)
	\bigg\}.
	\end{align*}}%
as $n\rightarrow\infty$ with $p/n\rightarrow0$.
\end{lemma}
\begin{proof}
	By Stirling's approximation applied to $\log\Gamma$, as well as \citet[Theorem 2.1.12]{muirhead82}, we find that
	{\setlength{\mathindent}{10pt}
		\setlength{\abovedisplayskip}{0pt}
		\setlength{\belowdisplayskip}{0pt}\begin{align*}\quad
		\log\Gamma_p(x) 
		&= \frac{p(p-1)}4\log\pi + \sum_{i=1}^p\log\Gamma\!\left(\!x-\frac{i\!-\!1}2\right)
		\\&
		= \frac{p(p-1)}4\log\pi 
		+ \sum_{i=1}^p\left[
		\left(\!x-\frac{i\!-\!1}2 - \frac12\right)\log\left(\!x-\frac{i\!-\!1}2\right)
		\right.\\&\hspace{110pt}\left.
		-\left(\!x-\frac{i\!-\!1}2\right)
		+\frac12\log{2\pi}
		+O\left(\!\frac1{x}\right)
		\right]
		\\&
		= \frac{p(p+1)}4\log\pi 
		+ \frac{p}2\log 2
		- px
		+ \frac{p(p-1)}4
		\\&\hspace{40pt}
		+ \sum_{i=1}^p \left(\!x-\frac{i}2\right)\log\left(\!x-\frac{i\!-\!1}2\right)
		+O\left(\frac{p}{x}\right)
		\end{align*}}%
	as $x\rightarrow\infty$. Thus
	{\setlength{\mathindent}{10pt}
		\setlength{\abovedisplayskip}{3pt}
		\setlength{\belowdisplayskip}{3pt}\begin{align*}&\quad
		2\log\Gamma_p\left(\!\frac{n\!+\!p\!+\!1}4\right)
		-\log\Gamma_p\left(\frac{n}2\right)
		= \frac{p(p+1)}4\log\pi 
		+ \frac{p}2\log 2
		\\&\hspace{50pt}
		- \frac{p(p+3)}4
		+ \sum_{i=1}^p \left(\!\frac{n\!+\!p\!+\!1}2-i\right)\log\left(\!\frac{n\!+\!p\!+\!1}4-\frac{i\!-\!1}2\right)
		\\&\hspace{140pt}
		- \sum_{i=1}^p \left(\frac{n}2-\frac{i}2\right)\log\left(\frac{n}2-\frac{i\!-\!1}2\right)
		+ o(1)
		\\&\hspace{20pt}
		= \frac{p(p+1)}4\log\pi 
		- \frac{p(2n+p-1)}4\log 2
		+\frac{p(p+1)}4\log n
		- \frac{p(p+3)}4
		\\&\hspace{90pt}
		+ \frac12\sum_{i=1}^p \Big( n-[2i\!-\!p\!-\!1]\Big)\log\left(1-\frac{2i\!-\!p\!-\!3}{n}\right)
		\\&\hspace{140pt}
		- \frac12\sum_{i=1}^p \left( n-i \right)\log\left(1-\frac{i\!-\!1}{n}\right)
		+ o(1),
		\end{align*}}%
	as $n\rightarrow\infty$ with $p/n\rightarrow0$, and so by Equation \eqref{eq:lem-cnp-defcnp},
	{\setlength{\mathindent}{5pt}
		\setlength{\abovedisplayskip}{3pt}
		\setlength{\belowdisplayskip}{0pt}\begin{align}&
		\log C_{n, p} = \frac{p(3p+1)}4\log 2 
		- \frac{p(p+1)}4\log \pi 
		-\frac{p(p+3)}4
		\notag\\&\hspace{110pt}
		+ \frac12\sum_{i=1}^p\!\Big(n-[2i\!-\!p\!-\!1]\Big) \log\!\left(\!1-\frac{2i\!-\!p\!-\!3}{n}\right)
		\notag\\&\hspace{110pt}
		- \frac12\sum_{i=1}^p\!\left(n- i\right) \log\!\left(\!1-\frac{i\!-\!1}{n}\right)
		+ o(1).
		\label{eq:lem-cnp-logcnp}
		\end{align}}%
	Let us now focus on the two sums in this expression. Recall that for any $k\geq1$,
	{\setlength{\mathindent}{10pt}
		\setlength{\abovedisplayskip}{0pt}
		\setlength{\belowdisplayskip}{3pt}\begin{align}&\quad
		-\log(1-x) = x + \frac{x^2}2+\frac{x^3}3+\cdots+\frac{x^k}{k}+O(x^{k+1})
		\hspace{10pt}
		\text{ as }x\rightarrow0,
		\label{eq:lem-cnp-taylorlog}
		\end{align}}%
	even for negative $x$. Therefore,
	{\setlength{\mathindent}{5pt}
		\setlength{\abovedisplayskip}{3pt}
		\setlength{\belowdisplayskip}{3pt}\begin{align*}&
		- \frac12\sum_{i=1}^p \left(n - i\right)\log\left(\!1-\frac{i\!-\!1}{n}\right)
		\\&\hspace{10pt}
		=\; \frac12\sum_{i=1}^p n \sum_{k=1}^{K+2}\frac{(i\!-\!1)^k}{k\,n^k} 
		- \frac12\sum_{i=1}^p i \sum_{k=1}^{K+1}\frac{(i\!-\!1)^k}{k\,n^k}
		+ O\left(\frac{p^{K+4}}{n^{K+2}}\right)
		\\&\hspace{10pt}
		=\; \frac12\sum_{i=1}^p(i\!-\!1)
		+ \frac12\sum_{i=1}^p \sum_{k=1}^{K+1}\left(\frac{i\!-\!1}{k\!-\!1}-\frac{i}{k}\right)\frac{(i\!-\!1)^k}{n^k}
		+ O\left(\frac{p^{K+3}}{n^{K+1}}\frac{p}{n}\right)
		\\&\hspace{10pt}
		=\; \frac{p(p-1)}4
		- \frac12\sum_{i=1}^p \sum_{k=1}^{K+1}\left(\frac{i\!-\!1}{k(k\!+\!1)}+\frac{1}{k}\right)\frac{(i\!-\!1)^k}{n^k}
		+ o\left(\frac{p^{K+3}}{n^{K+1}}\right)
		\\&\hspace{10pt}
		=\; \frac{p(p-1)}4
		- \frac12 \sum_{k=1}^{K+1}\!\frac1{k(k\!+\!1)n^k} \!\sum_{i=1}^{p-1} i^{k+1}
		- \frac12 \sum_{k=1}^{K+1}\!\frac1{k\,n^k} \!\sum_{i=1}^{p-1} i^k
		+ o\left(\frac{p^{K+3}}{n^{K+1}}\right)\!.
		\end{align*}}%
	Now let $B_k$ denote the Bernoulli numbers, with the convention $B_1=\frac12$. Faulhaber's formula provides
	{\setlength{\mathindent}{5pt}
		\setlength{\abovedisplayskip}{3pt}
		\setlength{\belowdisplayskip}{3pt}\begin{align*}&\hspace{10pt}
		= \frac{p(p-1)}4
		- \frac12 \sum_{k=1}^{K+1}\!\frac1{k(k\!+\!1)n^k} \cdot\frac1{k\!+\!2}\sum_{l=1}^{k+2} B_{k+2-l}\,(p\!-\!1)^l
		\\&\hspace{100pt}
		- \frac12 \sum_{k=1}^{K+1}\!\frac1{k\,n^k} \cdot\frac1{k\!+\!1}\sum_{l=1}^{k+1}B_{k+1-l}\,(p\!-\!1)^l
		+ o\left(\frac{p^{K+3}}{n^{K+1}}\right)\!.
		\end{align*}}%
	But by the binomial theorem, $\frac{(p-1)^{k+2}}{n^k} = \frac{p^{k+2}}{n^k}-(k+2)\frac{p^{k+1}}{n^k}+o(1)$, $\frac{(p-1)^{k+1}}{n^k}=\frac{p^{k+1}}{n} + o(1)$ and $\frac{(p-1)^l}{n^k} = o(1)$ for any $1\leq l\leq k$. Thus
	{\setlength{\mathindent}{5pt}
		\setlength{\abovedisplayskip}{3pt}
		\setlength{\belowdisplayskip}{3pt}\begin{align*}&\hspace{10pt}
		= \frac{p(p-1)}4
		- \frac12 \sum_{k=1}^{K+1}\bigg[
		\binom{k\!+\!2}{k\!+\!2}\frac{B_0}{k(k\!+\!1)(k\!+\!2)}\frac{p^{k+2}-(k\!+\!2)p^{k+1}}{n^k}
		\\&\hspace{140pt}
		+ \binom{k\!+\!2}{k\!+\!1}\frac{B_1}{k(k\!+\!1)(k\!+\!2)}\frac{p^{k+1}}{n^k}
		+ o(1)
		\bigg]
		\\&\hspace{20pt}
		- \frac12 \sum_{k=1}^{K+1} \bigg[
		\binom{k\!+\!1}{k\!+\!1}\frac{B_0}{k(k\!+\!1)}\frac{p^{k+1}}{n^k}
		+ o(1)
		\bigg]
		+ o\left(\frac{p^{K+3}}{n^{K+1}}\right)\!.
		\end{align*}}%
	Using that $B_0=1$ and $B_1=\frac12$, we obtain
	{\setlength{\mathindent}{5pt}
		\setlength{\abovedisplayskip}{3pt}
		\setlength{\belowdisplayskip}{3pt}\begin{align}&\hspace{10pt}
		= \frac{p(p-1)}4
		- \frac12 \sum_{k=1}^{K+1} \frac{1}{k(k\!+\!1)(k\!+\!2)}\frac{p^{k+2}}{n^k}
		\notag\\&\hspace{140pt}
		- \frac14 \sum_{k=1}^{K+1} \frac{1}{k(k\!+\!1)}\frac{p^{k+1}}{n^k}
		+ o\left(\frac{p^{K+3}}{n^{K+1}}\right)\!.
		\label{eq:lem-cnp-firstsum}
		\end{align}}%
	The analysis of the other sum of Equation \eqref{eq:lem-cnp-logcnp} is similar but more involved, as we must distinguish the cases where $p$ is even and where $p$ is odd. We find, from Equation \eqref{eq:lem-cnp-taylorlog} again, that
	{\setlength{\mathindent}{5pt}
		\setlength{\abovedisplayskip}{3pt}
		\setlength{\belowdisplayskip}{3pt}\begin{align}&
		\frac12\sum_{i=1}^p\!\Big(n-[2i\!-\!p\!-\!1]\Big) \log\!\left(\!1-\frac{2i\!-\!p\!-\!3}{n}\right)
		\notag\\&\hspace{10pt}
		= -\frac12\sum_{i=1}^p (2i\!-\!p\!-\!1)\log\!\left(\!1-\frac{2i\!-\!p\!-\!3}{n}\right)
		+\frac12\sum_{i=1}^p n\log\!\left(\!1-\frac{2i\!-\!p\!-\!3}{n}\right)
		\notag\\&\hspace{10pt}
		= \frac12\sum_{i=1}^p (2i\!-\!p\!-\!1)\!\!\sum_{k=1}^{K+1}\!\frac{(2i\!-\!p\!-\!3)^k}{k\,n^k}
		- \frac12\sum_{i=1}^p n\!\!\sum_{k=1}^{K+2}\!\frac{(2i\!-\!p\!-\!3)^k}{k\,n^k}
		+ O\left(\frac{p^{K+4}}{n^{K+2}}\right)
		\notag\\&\hspace{10pt}
		= \frac12\sum_{i=1}^p (2i\!-\!p\!-\!3)
		+ \frac12\sum_{k=1}^p\sum_{k=1}^{K+1}
		\!\left(\!\frac{2i\!-\!p\!-\!1}{k} - \frac{2i\!-\!p\!-\!3}{k\!+\!1}\right)\!\frac{(2i\!-\!p\!-\!3)^k}{n^k}
		\notag\\&\hspace{270pt}
		+ O\left(\frac{p^{K+3}}{n^{K+1}}\frac{p}{n}\right)
		\notag\\&\hspace{10pt}
		= p + \frac12\sum_{i=1}^p\sum_{k=1}^{K+1}
		\!\left(\!\frac{2i\!-\!p\!-\!3}{k(k\!+\!1)} + \frac{2}{k}\right)\!\frac{(2i\!-\!p\!-\!3)^k}{n^k}
		+ o\left(\frac{p^{K+3}}{n^{K+1}}\right)
		\notag\\&\hspace{10pt}
		= p + \frac12\sum_{k=1}^{K+1} \frac{1}{k(k\!+\!1)n^k} \sum_{i=1}^p (2i\!-\!p\!-\!3)^{k+1}
		\notag\\&\hspace{150pt}
		+ \sum_{k=1}^{K+1} \frac{1}{kn^k}\sum_{i=1}^p (2i\!-\!p\!-\!3)^k
		+ o\left(\frac{p^{K+3}}{n^{K+1}}\right)
		\notag\\&\hspace{10pt}
		= p + \frac12\sum_{k=1}^{K+1} \!\frac{1}{k(k\!+\!1)n^k} \bigg[\!(-p\!-\!1)^{k+1} + (-p\!+\!1)^{k+1} +\!\sum_{i=3}^p (2i\!-\!p\!-\!3)^{k+1}\bigg]
		\notag\\&\hspace{20pt}
		+ \sum_{k=1}^{K+1} \!\frac{1}{kn^k}\bigg[\!(-p\!-\!1)^{k} + (-p\!+\!1)^{k} +\!\sum_{i=3}^p (2i\!-\!p\!-\!3)^{k}\bigg]
		+ o\left(\frac{p^{K+3}}{n^{K+1}}\right)\!\!.
		\label{eq:lem-cnp-halfway-secondsum}
		\end{align}}%
	At this point, it is simpler to analyze the cases where $p$ is even and odd separately. If $p$ is odd, define $q=(p-3)/2$ and observe that by Faulhaber's formula,
	{\setlength{\mathindent}{10pt}
		\setlength{\abovedisplayskip}{3pt}
		\setlength{\belowdisplayskip}{3pt}\begin{align*}\quad
		\sum_{i=3}^p(2i\!-\!p\!-\!3)^l 
		&= \Big[1\!+\!(-1)^l\Big]\sum_{i=1}^q (2i)^l 
		\\&
		= \frac{\1{l\,\mathrm{even}}}{l\!+\!1} \sum_{s=1}^{l+1}\binom{l\!+\!1}{s} B_{l+1-s}\,2^{l+1-s}\,(2q)^s.
		\end{align*}}%
	By the binomial theorem, $\frac{(2q)^{k+2}}{n^k} = \frac{p^{k+2}}{n^k}-3(k\!+\!2)\frac{p^{k+1}}{n^k}+o(1)$,  $\frac{(2q)^{k+1}}{n^k} = \frac{p^{k+1}}{n^k} + o(1)$ and $\frac{(2q)^l}{n^k} = o(1)$ for $1\leq l\leq k$. Moreover,
	{\setlength{\mathindent}{5pt}
		\setlength{\abovedisplayskip}{3pt}
		\setlength{\belowdisplayskip}{3pt}\begin{align}&\quad
		\frac{(-p\!-\!1)^{k+1}}{n^k} = \frac{(-p\!+\!1)^{k+1}}{n^k} = \frac{(-p)^{k+1}}{n^k} + o(1)
		\label{eq:lem-cnp-power1}
		\end{align}}%
	and
	{\setlength{\mathindent}{5pt}
		\setlength{\abovedisplayskip}{3pt}
		\setlength{\belowdisplayskip}{3pt}\begin{align}&\quad
		\frac{(-p\!-\!1)^{k}}{n^k} = \frac{(-p\!+\!1)^{k}}{n^k} = o(1).
		\label{eq:lem-cnp-power2}
		\end{align}}%
	Thus, for odd $p$, Equation \ref{eq:lem-cnp-halfway-secondsum} equals
	{\setlength{\mathindent}{5pt}
		\setlength{\abovedisplayskip}{3pt}
		\setlength{\belowdisplayskip}{3pt}\begin{align*}&\hspace{10pt}
		= p + \frac12\sum_{k=1}^{K+1} \!\frac{1}{k(k\!+\!1)} 
		\bigg[
		+ \kodd\binom{k\!+\!2}{k\!+\!2}\frac{B_0}{k\!+\!2} \frac{p^{k+2}-3(k\!+\!2)p^{k+1}}{n^k}
		\\&\hspace{100pt}
		+ \kodd\binom{k\!+\!2}{k\!+\!1}\frac{2B_1}{k\!+\!2} \frac{p^{k+1}}{n^k}
		+ 2\frac{(-p)^{k+1}}{n^k}
		+ o(1)
		\bigg]
		\\&\hspace{55pt}
		+ \frac12\sum_{k=1}^{K+1} \!\frac{1}{k}\bigg[\keven\binom{k\!+\!1}{k\!+\!1}\frac{B_0}{k\!+\!1} \frac{p^{k+1}}{n^k} + o(1)\bigg]
		+ o\left(\frac{p^{K+3}}{n^{K+1}}\right)\!\!.
		\end{align*}}%
	Moreover,
	{\setlength{\mathindent}{20pt}
		\setlength{\abovedisplayskip}{3pt}
		\setlength{\belowdisplayskip}{3pt}\begin{align*}&
		2(-p)^{k+1} - 3\kodd p^{k+1} + \kodd p^{k+1} + \keven p^{k+1} 
		\\&\hspace{240pt}
		= -\keven p^{k+1}.
		\end{align*}}%
	Thus,
	{\setlength{\mathindent}{5pt}
		\setlength{\abovedisplayskip}{3pt}
		\setlength{\belowdisplayskip}{3pt}\begin{align}&\hspace{10pt}
		= p +\!\frac12\!\sum_{k=1}^{K+1} \!\!\frac{\kodd}{k(k\!+\!1)(k\!+\!2)} \frac{p^{k+2}}{n^k}
		-\!\frac12\!\sum_{k=1}^{K+1} \!\frac{\keven}{k(k\!+\!1)} \frac{p^{k+1}}{n^k}
		+\!o\!\left(\!\frac{p^{K+3}}{n^{K+1}}\right)\!\!.
		\label{eq:lem-cnp-secondsum-odd}
		\end{align}}%
	When $p$ is even, let $q=(p-2)/2$ and observe that by Faulhaber's formula,
	{\setlength{\mathindent}{10pt}
		\setlength{\abovedisplayskip}{3pt}
		\setlength{\belowdisplayskip}{3pt}\begin{align*}\quad
		\sum_{i=3}^p(2i\!-\!p\!-\!3)^l 
		&= \Big[1\!+\!(-1)^l\Big]\sum_{i=1}^q (2i-1)^l 
		= 2\1{l\,\mathrm{even}}\bigg(\sum_{i=1}^{2q}i^l - \sum_{i=1}^{q}(2i)^l\bigg)
		\\&
		= \frac{\1{l\,\mathrm{even}}}{l\!+\!1} \sum_{s=1}^{l+1}\binom{l\!+\!1}{s} B_{l+1-s}\,(2-2^{l+1-s})\,(2q)^s.
		\end{align*}}%
	But by the binomial theorem, $\frac{(2q)^{k+2}}{n^k} = \frac{p^{k+2}}{n^k}-2(k\!+\!2)\frac{p^{k+1}}{n^k}+o(1)$,  $\frac{(2q)^{k+1}}{n^k} = \frac{p^{k+1}}{n^k} + o(1)$ and $\frac{(2q)^l}{n^k} = o(1)$ for $1\leq l\leq k$. If we apply Equations \eqref{eq:lem-cnp-power1}--\eqref{eq:lem-cnp-power2}, then Equation \eqref{eq:lem-cnp-halfway-secondsum} becomes
	{\setlength{\mathindent}{5pt}
		\setlength{\abovedisplayskip}{3pt}
		\setlength{\belowdisplayskip}{3pt}\begin{align*}&\hspace{10pt}
		= p + \frac12\sum_{k=1}^{K+1} \frac{1}{k(k\!+\!1)}\bigg[
		\kodd\binom{k\!+\!2}{k\!+\!2}\frac{(2-1)B_0}{k\!+\!2}\frac{p^{k+2}-2(k+2)p^{k+1}}{n^k}
		\\&\hspace{140pt}
		+ \kodd\binom{k\!+\!2}{k\!+\!1}\frac{(2-2)B_1}{k\!+\!2}\frac{p^{k+1}}{n^k} 
		+ o(1)
		\bigg]
		\\&\hspace{40pt}
		+\frac12\sum_{k=1}^K\frac1{k}\bigg[
		\keven\binom{k\!+\!1}{k\!+\!1}\frac{(2-1)B_0}{k\!+\!1}\frac{p^{k+1}}{n^k}
		+ o(1)
		\bigg]
		+ \!o\!\left(\!\frac{p^{K+3}}{n^{K+1}}\right)\!\!.
		\end{align*}}%
	Moreover,
	{\setlength{\mathindent}{20pt}
		\setlength{\abovedisplayskip}{3pt}
		\setlength{\belowdisplayskip}{3pt}\begin{align*}&
		2(-p)^{k+1} - 2\kodd p^{k+1} + \keven p^{k+1} = -\keven p^{k+1}.
		\end{align*}}%
	Thus again,
	{\setlength{\mathindent}{5pt}
		\setlength{\abovedisplayskip}{3pt}
		\setlength{\belowdisplayskip}{3pt}\begin{align}&\hspace{10pt}
		= p +\!\frac12\!\sum_{k=1}^{K+1} \!\!\frac{\kodd}{k(k\!+\!1)(k\!+\!2)} \frac{p^{k+2}}{n^k}
		-\!\frac12\!\sum_{k=1}^{K+1} \!\frac{\keven}{k(k\!+\!1)} \frac{p^{k+1}}{n^k}
		+\!o\!\left(\!\frac{p^{K+3}}{n^{K+1}}\right)\!\!.
		\label{eq:lem-cnp-secondsum-even}
		\end{align}}%
	which is the exact same result as in the odd $p$ case (see Equation \ref{eq:lem-cnp-secondsum-odd}). Plugging Equations \eqref{eq:lem-cnp-firstsum} and \eqref{eq:lem-cnp-secondsum-odd}--\eqref{eq:lem-cnp-secondsum-even} in Equation \eqref{eq:lem-cnp-logcnp}, we obtain
	{\setlength{\mathindent}{10pt}
		\setlength{\abovedisplayskip}{3pt}
		\setlength{\belowdisplayskip}{3pt}\begin{align*}&
		\log C_{n,p} = \frac{p(3p+1)}4\log 2 - \frac{p(p+1)}2\log\pi
		- \frac12\sum_{k=1}^{K+1}\frac{\keven}{k(k\!+\!1)(k\!+\!2)}\frac{p^{k+2}}{n^k}
		\\&\hspace{140pt}
		- \frac14\sum_{k=1}^{K+1}\frac{1+2\keven}{k(k\!+\!1)}\frac{p^{k+1}}{n^k}
		+ \!o\!\left(\!\frac{p^{K+3}}{n^{K+1}}\right)\!\!,
		\end{align*}}%
	as desired.
\end{proof}

Thus the constant $C_{n,p}$ is closely related to the normalization constant of the $\text{GOE}(p)$ distribution $2^{p(3p+1)/4}/\pi^{p(p+1)/2}$.

We now turn our attention to the study of the asymptotic moments of a $T_{n/2}(I_p/8)$ distribution. We first remind the reader of some classic results. For a Gaussian Orthogonal Ensemble matrix $Z\sim\text{GOE}(p)$, a moment-based approach to Wigner's theorem states that for any $k\in\N$, its $k^{\text{th}}$ moment satisfy
{\setlength{\mathindent}{10pt}
\setlength{\abovedisplayskip}{3pt}
\setlength{\belowdisplayskip}{3pt}\begin{align*}&\hspace{50pt}
\lim_{p\rightarrow\infty} \text{E}\bigg[\frac1p\Tr\Big(\frac{Z}{\sqrt{p}}\Big)^{k}\bigg] = C_{k/2}\keven,
\end{align*}}%
where $C_k=\frac1{k+1}\binom{2k}{k}$ is the $k^\text{th}$ Catalan number. In fact, \citet[section 2.1.4 on p.17]{anderson10} show that the variance of the $k^{\text{th}}$ moment satisfies  $\lim\limits_{p\rightarrow\infty}\text{Var}\big[\frac1p\Tr (Z/\sqrt{p})^{k}\big]=0$, so we really have
{\setlength{\mathindent}{10pt}
\setlength{\abovedisplayskip}{3pt}
\setlength{\belowdisplayskip}{3pt}\begin{align*}&\hspace{50pt}
\frac1p\Tr\Big(\frac{Z}{\sqrt{p}}\Big)^{k} \overset{\text{L}^2}{\longrightarrow} C_{k/2}\keven 
\end{align*}}%
as $p\rightarrow\infty$.

Now, what do we know about the moments of $T_{n/2}(I_p/8)$? By symmetry, $\E{\Tr T^k}=0$ for odd $k$, but it is much less clear what happens for even $k$. It turns out that in many ways, if $T\sim T_{n/2}(I_p/8)$ then $4T\sim T_{n/2}(2I_p)$ mimics the Gaussian Orthogonal Ensemble results outlined above, especially when $p/n\rightarrow0$ as $n\rightarrow\infty$. We have the following result.

\begin{theorem}\label{thm:moments} Let $k\in\N$ and $T\sim T_{n/2}(I_p/8)$. If $p/n\rightarrow c\in[0,1)$, the moments of $T$ satisfy the asymptotic bounds $\E{\Tr T^{2k}} = O(p^{k+1})$ and $\E{\Tr^2 T^{k}} = O(p^{k+2})$ as $n\rightarrow\infty$. In fact, for any $k\in\N$, 
{\setlength{\mathindent}{10pt}
\setlength{\abovedisplayskip}{5pt}
\setlength{\belowdisplayskip}{5pt}\begin{align*}&\hspace{50pt}
\frac1p\Tr\Big(\frac{4T}{\sqrt{p}}\Big)^{k} \overset{\text{L}^2}{\longrightarrow} C_{k/2}\keven
\end{align*}}%
as $n, p\rightarrow\infty$ with $p/n\rightarrow0$, where $C_k=\frac1{k+1}\binom{2k}{k}$ is the $k^{\text{th}}$ Catalan number.
\end{theorem}

Although our proof will rely on the close relationship between the Wishart and the $t$ distribution, it is worthwhile to step back and think why a $T_{n/2}(2I_p)$ should behave like a $\text{GOE}(p)$ when $p/n\rightarrow0$. One good reason might be the classic result that as $n\rightarrow\infty$, the density of a $t$ distribution converges pointwise to a standard normal density. Thus, we might think that as long as $p$ does not grow too fast, in some aspects the symmetric $t$ distribution should behave like a $\text{GOE}(p)$.

In the context of the proof, it will prove useful to use the notion of power sum symmetric polynomials. For any integer partition $\kappa=(\kappa_1,\dots,\kappa_q)$ in decreasing order $\kappa_1\geq\dots\geq\kappa_q>0$, define its associated power sum polynomial to be
{\setlength{\mathindent}{10pt}
\setlength{\abovedisplayskip}{3pt}
\setlength{\belowdisplayskip}{3pt}\begin{align}&\hspace{50pt}
	r_\kappa(Z) = \prod_{i=1}^q\Tr Z^{\kappa_i}.
	\label{def:powersum-polynomial}
\end{align}}%
The norm of the partition $\kappa$ is $|\kappa|=\kappa_1+\dots+\kappa_q>0$, which should not be confused with its length $q(\kappa)=q$ (number of elements).

By convention, we will assume there also exists an empty partition $\varnothing=()$ with length $q(\varnothing)=0$, norm $|\varnothing|=0$ and power sum polynomial $r_\varnothing(Z)=1$.

Let's now turn to the proof of the theorem. The odd moments of the $T_{n/2}(I_p/8)$ moments are zero by symmetry, so it makes sense to focus on the even moments $\E{\Tr T^{2k}}$ and the square moments $\E{\Tr^2 T^{k}}$. Our first step in the proof is to express these in terms of expectations of power sum polynomials of an inverse Wishart $Y^{-1}\sim\text{W}^{-1}_p(n, I_p/n)$, where by power sum polynomials we mean expressions like at Equation \eqref{def:powersum-polynomial}. Recall the useful shorthand $m=n-p-1$. 

\begin{lemma}\label{lem:moments-as-expectation} Let $T\sim\text{T}_{n/2}(I_p/8)$. Then for any $k\in\N$, whenever $n$ is large enough so that $n\geq p+16k+6$, we can compute the $2k^\text{th}$ moment of $T$ by
{\setlength{\mathindent}{5pt}
\setlength{\abovedisplayskip}{0pt}
\setlength{\belowdisplayskip}{3pt}\begin{align}
\E{\Tr T^{2k}} &= 
	\frac{(-1)^k }{n^k}
	\bigint_{\substack{Y>0}} \hspace{-13pt}
		\frac{n^{\frac{np}2}}{2^{\frac{np}2} \Gamma_p\big(\frac{n}2\big)}
		\exp\Big\{\!\!-\!\frac{n}4\Tr Y \Big\}
		\big|Y\big|^{\frac{m}4}
\notag\\[-5pt]&\hspace{-35pt}
	\cdot\raisebox{2pt}{$\mathlarger{\mathlarger\sum_{i_1,\dots,i_{2k}}^p}$}
	\frac{\partial_\text{s}}{\partial_\text{s} X_{i_1 i_{2k}}}
	\dots \frac{\partial_\text{s}}{\partial_\text{s} X_{i_3i_2}}
	\frac{\partial_\text{s}}{\partial_\text{s} X_{i_2i_1}}
	\exp\Big\{\!\!-\!\frac{n}4\Tr(Y\!+\!X) \Big\}
	\big|Y\!+\!X\big|^{\frac{m}4}
	\bigg\vert_{X=0}
	dY
\notag\\&
= \frac{(-1)^k}{n^k}\!\!\sum_{|\kappa|\leq 2k} \hspace{-3pt}
	b^{(1)}_\kappa(n,m,p)
	\E{r_\kappa(Y^{-1})}
\label{eqn:moments-trT2k-as-polynomial}
\end{align}}%
and its squared $k^\text{th}$ moment by
{\setlength{\mathindent}{5pt}
\setlength{\abovedisplayskip}{3pt}
\setlength{\belowdisplayskip}{3pt}\begin{align}
\E{\Tr^2 T^k} &= 
	\frac{(-1)^k}{n^{k}}
	\bigint_{\substack{Y>0}} \hspace{-13pt}
		\frac{n^{\frac{np}2}}{2^{\frac{np}2} \Gamma_p\big(\frac{n}2\big)}
		\exp\Big\{\!\!-\!\frac{n}4\Tr Y \Big\}
		\big|Y\big|^{\frac{m}4}
\notag\\[-2pt]&\hspace{20pt}
	\cdot	\raisebox{2pt}{$\mathlarger{\mathlarger
	\sum_{\substack{i_1,\dots,i_{k}\\j_1,\dots,j_{k}}}^p}$}
	\frac{\partial_\text{s}}{\partial_\text{s} X_{j_1 j_{k}}}
	\dots \frac{\partial_\text{s}}{\partial_\text{s} X_{j_3j_2}}
	\frac{\partial_\text{s}}{\partial_\text{s} X_{j_2j_1}}
	\frac{\partial_\text{s}}{\partial_\text{s} X_{i_1 i_{k}}}
	\dots \frac{\partial_\text{s}}{\partial_\text{s} X_{i_3i_2}}
	\frac{\partial_\text{s}}{\partial_\text{s} X_{i_2i_1}}
\notag\\[-5pt]&\pushright{
	\exp\Big\{\!\!-\!\frac{n}4\Tr(Y\!+\!X) \Big\}
	\big|Y\!+\!X\big|^{\frac{m}4}
	\bigg\vert_{X=0}
	dY
\hspace{45pt}}\notag\\[-5pt]&
= \frac{(-1)^k}{n^{k}}\!\!\sum_{|\kappa|\leq 2k+1} \hspace{-8pt}
	b^{(2)}_\kappa(n,m,p)
	\E{r_\kappa(Y^{-1})},
\label{eqn:moments-tr2Tk-as-polynomial}
\end{align}}%
for $Y^{-1}\sim\text{W}^{-1}_p(n,I_p/n)$ and some $b^{(1)}_\kappa$, $b^{(2)}_\kappa$. These $b^{(1)}_\kappa$, $b^{(2)}_\kappa$ are polynomials in $n,m,p$, indexed by integer partitions $\kappa$, whose degrees satisfy $\mathrm{deg}\,b^{(1)}_\kappa\leq 2k+1-q(\kappa)$ and $\mathrm{deg}\,b^{(2)}_\kappa\leq 2k+2-q(\kappa)$. The sums are taken over all partitions of the integers $\kappa$ satisfying $|\kappa|\leq 2k$ and $|\kappa|\leq 2k+1$ respectively, including the empty partition. 
\end{lemma}
\begin{proof}
Let $f_{\text{NW}}$ and $\psi_{\text{NW}}$ stand for the density and the G-transform of a normalized Wishart matrix $\sqrt{n}[\text{W}_p(n,I_p/n)-I_p]$. In the proof of Proposition \ref{prop:gtransform-nw}, we concluded at Equation \eqref{eqn:fnw12-integrable} that $f^{1/2}_{\text{NW}}$ had to be integrable when $n\geq p-2$, as its integral was proportional to a multivariate gamma function. Let $R(X)=-X$ be the flip operator. Since $f^{1/2}_{\text{NW}}$ is integrable, $f^{1/2}_{\text{NW}}\circ R$ must be integrable as well, and so their convolution $f^{1/2}_{\text{NW}}\star \big(f^{1/2}_{\text{NW}}\circ R\big)$ is well-defined and integrable.

At Equation \eqref{def:fourier-transform} at the start of Section \ref{sec:gtransforms}, we defined our notion of Fourier transform for integrable functions on $\S_p(\R)$. Define the map $\iota:S_p(\R)\rightarrow \R^{p(p+1)/2}$ that maps a symmetric matrix to its vectorized upper triangle, and let $\tau:\S_p(\R)\rightarrow\S_p(\R)$ be the map
\begin{align*}
\tau(X)_{ij} = \begin{cases}
2X_{ij} & \text{ if }i\neq j \\
X_{jj} & \text{ if }i=1.
\end{cases}
\end{align*}
Then in terms of the usual Fourier transform on $\R^{p(p+1)/2}$,
{\setlength{\mathindent}{10pt}
\setlength{\abovedisplayskip}{3pt}
\setlength{\belowdisplayskip}{3pt}
\begin{align*}
\mathcal{F}\{f\}(T) = 2^{\frac{p(p-1)}4}\mathcal{F}\big\{f\circ\iota^{-1}\big\}\big(\iota\circ\tau(T)\big).
\end{align*}}%
This close relationship transfer properties to our Fourier transform on $\S_p(\R)$. We will need three.
\begin{enumerate}
\item For any integrable function $f$, we have $\mathcal{F}\big\{\overline{f\circ R}\big\}=\overline{\mathcal{F}\{f\}}$.
\item (\textit{Convolution}) For any two integrable functions $f_1$ and $f_2$, we have $\mathcal{F}\{f_1\star f_2\}=2^{\frac{p}2}\pi^{\frac{p(p+1)}4}\mathcal{F}\{f_1\}\mathcal{F}\{f_2\}$.
\item (\textit{Fourier inversion}) For any continuous integrable $f$ with integrable Fourier transform $\phi$, we have
{\setlength{\mathindent}{10pt}
\setlength{\abovedisplayskip}{3pt}
\setlength{\belowdisplayskip}{3pt}\begin{align*}
f(X) = \frac1{2^{\frac{p}2}\pi^{\frac{p(p+1)}4}}\bigintss_{\S_p(\R)}\hspace{-15pt}
e^{i\Tr (TX)}\phi(T)\,dT= \mathcal{F}\{\phi\}(-X),
\end{align*}}%
for all $X\in\S_p(\R)$.
\end{enumerate}
These properties are important for the following. Since $f_{\text{NW}}$ is real-valued, properties 1 and 2 provide
{\setlength{\mathindent}{10pt}
\setlength{\abovedisplayskip}{3pt}
\setlength{\belowdisplayskip}{3pt}\begin{align*}&
2^{-\frac{p}2}\pi^{-\frac{p(p+1)}4}\mathcal{F}\Big\{
	f^{1/2}_{\text{NW}}\star \big(f^{1/2}_{\text{NW}}\circ R\big) \Big\}
	= \mathcal{F}\big\{ f^{1/2}_{\text{NW}}\Big\} \mathcal{F}\big\{ f^{1/2}_{\text{NW}}\circ R \big\}
\\&\pushright{
	= \mathcal{F}\big\{ f^{1/2}_{\text{NW}}\big\} \overline{\mathcal{F}\big\{ f^{1/2}_{\text{NW}}\big\}}
	= \psi_{\text{NW}}^{1/2} \overline{\psi_{\text{NW}}^{1/2}}
	= \big|\psi_{\text{NW}}^{1/2}\big|.
}\end{align*}}%
But then, since $|\psi_{\text{NW}}|$ is integrable (in fact, to unity), the Fourier inversion formula yields that
{\setlength{\mathindent}{20pt}
\setlength{\abovedisplayskip}{0pt}
\setlength{\belowdisplayskip}{3pt}\begin{align}
f^{1/2}_{\text{NW}}\star \big(f^{1/2}_{\text{NW}}\circ R\big)(X)
 = \bigintss_{\S_p(\R)}\hspace{-15pt}
 e^{i\Tr (TX)}\big|\psi_{\text{NW}}(T)\big|\,dT.
\label{def:characteristic-function-t}
\end{align}}%
Thus we might say the characteristic function of the $T_{n/2}(I_p/8)$ distribution is given by $f^{1/2}_{\text{NW}}\star \big(f^{1/2}_{\text{NW}}\circ R\big)$. It is well known that the derivatives of the characteristic function of a distribution evaluated at zero provide its moments, up to a constant. This suggests we should try to repeatedly differentiate $f^{1/2}_{\text{NW}}\star \big(f^{1/2}_{\text{NW}}\circ R\big)$ at zero to compute $\E{\Tr T^{2k}}$ and $\E{\Tr^2 T^{k}}$, our ultimate goal.

Unfortunately, the convolution is given by an integral whose domain makes it difficult to directly interchange the differentiation and integration symbols. Because the integrand is orthogonally invariant, we found it easier to compute the derivatives at zero by taking a limit over a sequence of decreasing positive-definite matrices at both sides instead. In this spirit, define on the open set $\{0<X<I_p\}\subset\S_p(\R)$ the real-valued functions
{\setlength{\mathindent}{10pt}
\setlength{\abovedisplayskip}{0pt}
\setlength{\belowdisplayskip}{3pt}\begin{align*}&
H_1(X) = 
	\frac{(-1)^k}{n^k}
	\hspace{-5pt}{\mathlarger{\mathlarger\sum_{i_1,\dots,i_{2k}}^p}}\hspace{-5pt}
	\frac{\partial_\text{s}}{\partial_\text{s} X_{i_1 i_{2k}}}
	\dots \frac{\partial_\text{s}}{\partial_\text{s} X_{i_3i_2}}
	\frac{\partial_\text{s}}{\partial_\text{s} X_{i_2i_1}}
	f^{1/2}_{\text{NW}}\!\star\!\big(f^{1/2}_{\text{NW}}\!\circ\!R\big)(\sqrt{n}X)
\end{align*}}%
and
{\setlength{\mathindent}{10pt}
\setlength{\abovedisplayskip}{0pt}
\setlength{\belowdisplayskip}{3pt}\begin{align*}&
H_2(X) = 
	\frac{(-1)^k}{n^{k}}
	\hspace{-5pt}{\mathlarger{\mathlarger\sum_{\substack{i_1,\dots,i_{k}\\j_1,\dots,j_{k}}}^p}}\hspace{-2pt}
	\frac{\partial_\text{s}}{\partial_\text{s} X_{j_1 j_{k}}}
	\dots \frac{\partial_\text{s}}{\partial_\text{s} X_{j_3j_2}}
	\frac{\partial_\text{s}}{\partial_\text{s} X_{j_2j_1}}
	\frac{\partial_\text{s}}{\partial_\text{s} X_{i_1 i_{k}}}
	\dots \frac{\partial_\text{s}}{\partial_\text{s} X_{i_3i_2}}
	\frac{\partial_\text{s}}{\partial_\text{s} X_{i_2i_1}}
\\[-5pt]&\pushright{
	f^{1/2}_{\text{NW}}\!\star\!\big(f^{1/2}_{\text{NW}}\!\circ\!R\big)(\sqrt{n}X)
}\end{align*}}%
for fixed $k$, $p$ and $n$. Here $\frac{\partial_\text{s}}{\partial_\text{s}X_{ij}}$ stands for the symmetric differentiation operator $\frac{\partial_\text{s}}{\partial_\text{s}X_{ij}} = \frac{1+\delta_{ij}}2\frac{\partial}{\partial X_{ij}}$, as defined in Section \ref{sec:notation}. The $\sqrt{n}$ scaling in the argument helps link the convolution to an expectation with respect to an inverse Wishart distribution.

Let's first relate these functions to the moments of the $T_{n/2}(I_p/8)$ distribution. The symmetric differentation operator has the pleasant property that  $\frac{\partial_\text{s}}{\partial_\text{s}X_{ij}}\Tr (XT) = T_{ij}$ for any two symmetric matrices $X$, $T$. Thus, for any $1\leq l\leq 2k$ and indices $1\leq i_1, \dots, i_{2l}\leq p$, we find that
{\setlength{\mathindent}{20pt}
\setlength{\abovedisplayskip}{3pt}
\setlength{\belowdisplayskip}{3pt}\begin{align}&
\bigg|\frac{\partial_\text{s}}{\partial_\text{s} X_{i_{2l} i_{2l-1}}}
	\dots \frac{\partial_\text{s}}{\partial_\text{s} X_{i_4i_3}}\frac{\partial_\text{s}}{\partial_\text{s} X_{i_2i_1}}
	e^{i\sqrt{n}\Tr (TX)}\big|\psi_\text{NW}\big|(T)\bigg|
\notag\\&\pushrightn{
=  n^l \Big|T_{i_{2l}i_{2l-1}}\cdots T_{i_4i_3}T_{i_2i_1}\Big|\big|\psi_\text{NW}\big|(T)
\label{eqn:moments-integrability1}
}\end{align}}%
for all $X\in\S_p(\R)$. 

We now show that the right hand side (\ref{def:characteristic-function-t}) is integrable. This is not a mere formality: when $p=1$, asking if this expression is integrable is the same as asking if the $t$ distribution with $n/2$ degrees of freedom has an $l^\text{th}$ moment, and it is well-known that the $t$ distribution only possesses moments of order smaller than its degrees of freedom. So the answer is most likely to be positive, but only for $n$ large enough.

Let us see why. For any symmetric matrix $T$,
{\setlength{\mathindent}{20pt}
\setlength{\abovedisplayskip}{3pt}
\setlength{\belowdisplayskip}{3pt}\begin{align*}
|T_{ij}|
\leq \sqrt{\lambda_1(T^2)}
\leq \sqrt{\prod\limits_{i=1}^p\Big(1+\lambda_i(T^2)\Big) }
= \sqrt{\big|I_p+T^2\big|},
\end{align*}}%
where $\lambda_1(T^2)\geq \dots\geq \lambda_p(T^2)\geq0$ are the ordered eigenvalues of the positive-definite matrix $T^2$. Thus
{\setlength{\mathindent}{10pt}
\setlength{\abovedisplayskip}{3pt}
\setlength{\belowdisplayskip}{3pt}\begin{align}&
\bigintss_{\S_p(\R)}\hspace{-15pt}
n^l \Big|T_{i_{2l}i_{2l-1}}\cdots T_{i_2i_1}\Big| \big|\psi_\text{NW}\big|(T)\,dT
\notag\\[-5pt]&\quad
= \frac{n^{\frac{3l}2}}{4^l}C_{n,p} \bigintss_{\S_p(\R)}\hspace{-15pt}
	\Big|\frac{4T_{i_{2l}i_{2l-1}}}{\sqrt{n}}
		\cdots \frac{4T_{i_2i_1}}{\sqrt{n}}\Big|
	\bigg|I_p +\frac{16T^2}{n}\bigg|^{-\frac{n+p+1}4}\,dT
\notag\\[-3pt]&\quad
\leq \frac{n^{\frac{3l}2}}{4^l}C_{n,p}\bigintss_{\S_p(\R)}\hspace{-15pt}
	\bigg|I_p
		 +\frac{16T^2}{n}
	 \bigg|^{-\frac{(n-2l)+p+1}4}\,dT
\notag\\[-5pt]&\quad
	 \leq \frac{n^{\frac{3l}2}C_{n,p}}{4^lC_{n-2l, p} }
	 {\left(\frac{n}{n-2l}\right)\!\!}^{\frac{p(p+1)}4}
	 \hspace{-5pt}\bigintss_{\S_p(\R)}\hspace{-15pt}
	 	C_{n-2l, p}\bigg|I_p
	 		 +\frac{16T^2}{n-2l}
	 	 \bigg|^{-\frac{(n-2l)+p+1}4}dT.
\label{eqn:moments-integrability2}
\end{align}}%
When $n-2l\geq p-2$, the last integrand is the density of a $T_{n/2-l}(I_p/8)$ distribution, so integrates to unity. Thus, when $n\geq p+4k-2$, the right hand side of Equation \eqref{eqn:moments-integrability1} is an integrable function for all $1\leq l\leq 2k$ and $1\leq i_1, \dots, i_{2l}\leq p$. By Equation \eqref{def:characteristic-function-t}, and repeated differentiation under the integral sign justified by the integrability bounds given by Equations \eqref{eqn:moments-integrability1} and \eqref{eqn:moments-integrability2}, we find that
{\setlength{\mathindent}{10pt}
\setlength{\abovedisplayskip}{0pt}
\setlength{\belowdisplayskip}{3pt}\begin{align}
H_1(X) &= \frac{(-1)^k }{n^k}
	\hspace{-5pt}{\mathlarger{\mathlarger\sum_{i_1,\dots,i_{2k}}^p}}\hspace{-5pt}
	\frac{\partial_\text{s}}{\partial_\text{s} X_{i_1 i_{2k}}}
	\dots \frac{\partial_\text{s}}{\partial_\text{s} X_{i_3i_2}}
	\frac{\partial_\text{s}}{\partial_\text{s} X_{i_2i_1}}
	\bigintss_{\S_p(\R)}\hspace{-17pt}
	 e^{i\sqrt{n}\Tr (TX)}\big|\psi_{\text{NW}}(T)\big|dT
\notag\\&
= \bigintss_{\S_p(\R)}\hspace{-15pt}
	 \Tr T^{2k} e^{i\sqrt{n}\Tr (TX)}\big|\psi_{\text{NW}}(T)\big|dT
\label{eqn:moments-H1-prelimit-cf}
\end{align}}%
and
{\setlength{\mathindent}{10pt}
\setlength{\abovedisplayskip}{0pt}
\setlength{\belowdisplayskip}{3pt}\begin{align}
H_2(X) &= \frac{(-1)^k}{n^{k}}
	\hspace{-5pt}{\mathlarger{\mathlarger\sum_{\substack{i_1,\dots,i_{k}\\j_1,\dots,j_{k}}}^p}}\hspace{-2pt}
	\frac{\partial_\text{s}}{\partial_\text{s} X_{j_1 j_{k}}}
	\dots \frac{\partial_\text{s}}{\partial_\text{s} X_{j_2j_1}}
\notag\\[-15pt]&\hspace{62pt}
	\frac{\partial_\text{s}}{\partial_\text{s} X_{i_1 i_{k}}}
	\dots \frac{\partial_\text{s}}{\partial_\text{s} X_{i_3i_2}}
	\frac{\partial_\text{s}}{\partial_\text{s} X_{i_2i_1}}
	\bigintss_{\S_p(\R)}\hspace{-15pt}
	 e^{i\sqrt{n}\Tr (TX)}\big|\psi_{\text{NW}}(T)\big|dT
\notag\\[-5pt]&
= \bigintss_{\S_p(\R)}\hspace{-15pt}
	 \Tr^2T^{k} e^{i\sqrt{n}\Tr (TX)}\big|\psi_{\text{NW}}(T)\big|dT.
\label{eqn:moments-H2-prelimit-cf}
\end{align}}%
for any $X\in\S_p(\R)$ and any $n\geq p+4k-2$.

Now let's relate $H_1$ and $H_2$ to the definition of $f^{1/2}_{\text{NW}}\!\star\!\big(f^{1/2}_{\text{NW}}\!\circ\!R\big)$ as a convolution. This is where restricting $H_1$ and $H_2$ to small positive-definite matrices becomes useful. By Equation \eqref{def:fnw}, the expression is
{\setlength{\mathindent}{10pt}
\setlength{\abovedisplayskip}{0pt}
\setlength{\belowdisplayskip}{3pt}\begin{align*}&
f^{1/2}_{\text{NW}}\!\star\!\big(f^{1/2}_{\text{NW}}\!\circ\!R\big)(\sqrt{n} X)
= \bigintss_{\S_p(\R)}\hspace{-15pt}
f_\text{NW}^{1/2}(Z)f_\text{NW}^{1/2}(Z-\sqrt{n}X)dZ
\\[-5pt]&\pushright{
= \frac{n^{\frac{np}2}}{2^{\frac{np}2} \Gamma_p\big(\frac{n}2\big)}
	\hspace{-3pt}\bigints_{\substack{Y+X>0,\,Y>0}}\hspace{-50pt}
	\exp\Big\{\!\!-\!\frac{n}4\Tr(Y\!+\!X) \Big\}
	\big|Y\!+\!X\big|^{\frac{m}4}
	\exp\Big\{\!\!-\!\frac{n}4\Tr Y \Big\}
	\big|Y\big|^{\frac{m}4}dY
}\end{align*}}%
using the change of variables $Y=I_p+Z/\sqrt{n}-X$ with $dZ=n^{\frac{p(p+1)}4}dY$. For $X>0$, we have $\1{Y+X>0, Y>0}=\1{Y>0}$, and thus $H_1$, $H_2$ satisfy
{\setlength{\mathindent}{5pt}
\setlength{\abovedisplayskip}{0pt}
\setlength{\belowdisplayskip}{3pt}\begin{align}&
H_1(X) = 
	\frac{(-1)^k }{n^k}
	\hspace{-5pt}{\mathlarger{\mathlarger\sum_{i_1,\dots,i_{2k}}^p}}\hspace{-5pt}
	\frac{\partial_\text{s}}{\partial_\text{s} X_{i_1 i_{2k}}}
	\dots \frac{\partial_\text{s}}{\partial_\text{s} X_{i_3i_2}}
	\frac{\partial_\text{s}}{\partial_\text{s} X_{i_2i_1}}
	\hspace{-3pt}\bigintss_{\substack{Y>0}}\hspace{-15pt}
	\exp\Big\{\!\!-\!\frac{n}4\Tr(Y\!+\!X) \Big\}
\notag\\[-5pt]&\pushrightn{\cdot
	\big|Y\!+\!X\big|^{\frac{m}4}
	\frac{n^{\frac{np}2}}{2^{\frac{np}2} \Gamma_p\big(\frac{n}2\big)}
	\exp\Big\{\!\!-\!\frac{n}4\Tr Y \Big\}
	\big|Y\big|^{\frac{m}4}dY
\label{eqn:moments-H1-pre-convolution}
}\end{align}}%
and
{\setlength{\mathindent}{5pt}
\setlength{\abovedisplayskip}{0pt}
\setlength{\belowdisplayskip}{3pt}\begin{align}&
H_2(X) = 
	\frac{(-1)^k}{n^{k}}
	\hspace{-5pt}{\mathlarger{\mathlarger\sum_{\substack{i_1,\dots,i_{k}\\j_1,\dots,j_{k}}}^p}}\hspace{-2pt}
	\frac{\partial_\text{s}}{\partial_\text{s} X_{j_1 j_{k}}}
	\dots \frac{\partial_\text{s}}{\partial_\text{s} X_{j_3j_2}}
	\frac{\partial_\text{s}}{\partial_\text{s} X_{j_2j_1}}
\notag\\[-7pt]&\hspace{60pt}
	\frac{\partial_\text{s}}{\partial_\text{s} X_{i_1 i_{k}}}
	\dots \frac{\partial_\text{s}}{\partial_\text{s} X_{i_3i_2}}
	\frac{\partial_\text{s}}{\partial_\text{s} X_{i_2i_1}}
	\hspace{-3pt}\bigintss_{\substack{Y>0}}\hspace{-15pt}
	\exp\Big\{\!\!-\!\frac{n}4\Tr(Y\!+\!X) \Big\}
	\big|Y\!+\!X\big|^{\frac{m}4}
\notag\\&\pushrightn{\cdot
	\frac{n^{\frac{np}2}}{2^{\frac{np}2} \Gamma_p\big(\frac{n}2\big)}
	\exp\Big\{\!\!-\!\frac{n}4\Tr Y \Big\}
	\big|Y\big|^{\frac{m}4}dY.
\label{eqn:moments-H2-pre-convolution}
}\end{align}}%
We would now like to interchange the integral and differentiation signs. To do so, we must understand what the repeated derivatives of $\exp\{-\frac{n}4\Tr Y \Big\} |Y|^{\frac{m}4}$ look like.
Differentiating once, we see that:
{\setlength{\mathindent}{10pt}
\setlength{\abovedisplayskip}{3pt}
\setlength{\belowdisplayskip}{3pt}\begin{align*}&
\frac{\partial_\text{s}}{\partial_\text{s} X_{i_2i_1}}
	\exp\Big\{\!\!-\!\frac{n}4\Tr (Y\!\!+\!\!X)\Big\}\big|Y\!\!+\!\!X\big|^{\frac{m}4}
\\&\hspace{20pt}
= \bigg[\frac{m}4(Y\!\!+\!\!X)^{-1}_{i_2i_1}-\frac{n}4(I_p)_{i_2i_1}\bigg]
	\exp\Big\{\!\!-\!\frac{n}4\Tr (Y\!\!+\!\!X)\Big\}\big|Y\!\!+\!\!X\big|^{\frac{m}4}.
\end{align*}}%
Differentiating twice, we see that:
{\setlength{\mathindent}{5pt}
\setlength{\abovedisplayskip}{3pt}
\setlength{\belowdisplayskip}{3pt}\begin{align*}&
\frac{\partial_\text{s}}{\partial_\text{s} X_{i_4i_3}}\frac{\partial_\text{s}}{\partial_\text{s} X_{i_2i_1}}
	\exp\Big\{\!\!-\!\frac{n}4\Tr (Y\!\!+\!\!X)\Big\}\big|Y\!\!+\!\!X\big|^{\frac{m}4}
\\&\hspace{0pt}
= \bigg[
		- \frac{m}8(Y\!\!+\!\!X)^{-1}_{i_2i_4}(Y\!\!+\!\!X)^{-1}_{i_3i_1}
		- \frac{m}8(Y\!\!+\!\!X)^{-1}_{i_2i_3}(Y\!\!+\!\!X)^{-1}_{i_4i_1}
\\&\hspace{10pt}
		+ \frac{m^2}{16}(Y\!\!+\!\!X)^{-1}_{i_4i_3}(Y\!\!+\!\!X)^{-1}_{i_2i_1}
		- \frac{mn}{16}(Y\!\!+\!\!X)^{-1}_{i_4i_3}(I_p)^{-1}_{i_2i_1}
\\&\hspace{10pt}
		- \frac{mn}{16}(I_p)^{-1}_{i_4i_3}(Y\!\!+\!\!X)^{-1}_{i_2i_1}
		+ \frac{n^2}{16}(I_p)^{-1}_{i_4i_3}(I_p)^{-1}_{i_2i_1}
	\bigg]
	\exp\Big\{\!\!-\!\frac{n}4\Tr (Y\!\!+\!\!X)\Big\}\big|Y\!\!+\!\!X\big|^{\frac{m}4}\!\!.
\end{align*}}%
So in general, it is clear that the repeated derivatives are given by some polynomial in entries of $(Y+X)^{-1}$, times $\exp\{-\frac{n}4\Tr (Y+X) \Big\} |Y+X|^{\frac{m}4}$. We won't investigate further the nature of these polynomials beyond remarking that for any indices $1\leq l \leq 2k$ and $1\leq i_1,\dots, i_{2l}\leq p$, and any symmetric matrices $X, Y\in\S_p(\R)$, we must have some crude bound
{\setlength{\mathindent}{10pt}
\setlength{\abovedisplayskip}{3pt}
\setlength{\belowdisplayskip}{3pt}\begin{align*}&
\bigg|\frac{\partial_\text{s}}{\partial_\text{s} X_{i_{2l} i_{2l-1}}}
	\dots \frac{\partial_\text{s}}{\partial_\text{s} X_{i_4i_3}}\frac{\partial_\text{s}}{\partial_\text{s} X_{i_2i_1}}
	\exp\bigg\{\!\!-\!\frac{n}4\Tr (Y\!+\!X)\bigg\}\big|Y\!+\!X\big|^{\frac{m}4}
	\bigg|
\\&\hspace{30pt}
\leq\; \sum_{s=0}^l\sum_{\substack{J\in\\\{1,\dots,p\}^{2l}}} 
	\Big|a_{J,s}(n, m)\Big| 
	\prod_{t=s+1}^l \big|(I_p)_{j_{2t}j_{2t-1}}\big|
	\prod_{t=1}^s\big|(Y\!+\!X)^{-1}_{j_{2t}j_{2t-1}}\big|
\\[-10pt]&\pushright{\cdot
	\exp\bigg\{\!\!-\!\frac{n}4\Tr (Y\!+\!X)\bigg\}\big|Y\!+\!X\big|^{\frac{m}4}
}\end{align*}}%
for some polynomials $a_{J, s}$ that do not depend on $X$ or $Y$. We relegate a proof of this result as Lemma \ref{lem:first-derivatives-lemma} in Section \ref{sec:auxiliary}. This can be uniformly bounded for all $0\leq X\leq I_p$ by
{\setlength{\mathindent}{10pt}
\setlength{\abovedisplayskip}{3pt}
\setlength{\belowdisplayskip}{3pt}\begin{align}&\hspace{30pt}
\leq\; C_1(n, m, p) \sum_{s=0}^l \Tr^s (Y^{-1})\exp\Big\{-\frac{n}4\Tr Y\Big\} \big[1+\Tr Y\big]^{\frac{mp}4}
\label{eqn:moments-integrability3}
\end{align}}%
for some constant $C_1(n,m,p)$ that does not depend on $X$ or $Y$. But for any $n\geq p-2$ and $l\geq 0$,
{\setlength{\mathindent}{20pt}
\setlength{\abovedisplayskip}{3pt}
\setlength{\belowdisplayskip}{3pt}\begin{align*}&
\bigintss_{Y>0}\hspace{-10pt}
	C_1(n, m, p) \sum_{s=0}^l \Tr^s (Y^{-1})
	\exp\Big\{-\frac{n}4\Tr Y\Big\} \big[1+\Tr Y\big]^{\frac{mp}4}
\\[-5pt]&\pushright{\cdot
	\frac{n^{\frac{np}2}}{2^{\frac{np}2} \Gamma_p\big(\frac{n}2\big)}
	\exp\Big\{\!\!-\!\frac{n}4\Tr Y \Big\}
	\big|Y\big|^{\frac{m}4}dY
\hspace{20pt}}\\[-5pt]&\hspace{20pt}
=  \frac{n^{\frac{np}2} C_1(n,m,p)}{2^{\frac{np}2} \Gamma_p\big(\frac{n}2\big)}
	\sum_{s=0}^l \bigint_{Y>0}\hspace{-18pt}
	\big[1\!+\!\Tr Y\big]^{\frac{mp}4} \!\Tr^s (Y^{-1})
\\[-12pt]&\pushright{
	\exp\Big\{\!\!-\!\frac{n}2\Tr Y \Big\}  \big|Y\big|^{\frac{n+p+1}4-\frac{p+1}2}\!dY
\hspace{20pt}}\\[-3pt]&\hspace{20pt}
= \left(\frac{n}2\right)^{\frac{mp}4}
	\frac{\Gamma_p\left(\frac{n+p+1}4\right)}{\Gamma_p\left(\frac{n}2\right)}
	\E{\vphantom{\Big\vert}(1+\Tr Y)^{\frac{mp}4}\Tr^s (Y^{-1})}
\end{align*}}%
for a $Y$ with a matrix gamma distribution $\text{G}_p\left(\frac{n+p+1}4, \frac{n}2I_p\right)$. The Cauchy-Schwarz inequality then entails the bound
{\setlength{\mathindent}{40pt}
\setlength{\abovedisplayskip}{3pt}
\setlength{\belowdisplayskip}{3pt}\begin{align}&
\leq \left(\frac{n}2\right)^{\frac{mp}4}
	\frac{\Gamma_p\left(\frac{n+p+1}4\right)}{\Gamma_p\left(\frac{n}2\right)}
	\E{\vphantom{\Big\vert} (1+\Tr Y)^{\frac{mp}2}}^{\frac12} 
	\E{\vphantom{\Big\vert} \Tr^{2s} (Y^{-1})}^{\frac12}.
\label{eqn:moments-integrability4}
\end{align}}%
The first expectation is always finite when $n\geq p-2$. Since $\Tr^{2s}(Y^{-1})$ can be written as a sum of zonal polynomials indexed by partitions of the integer $2s$,  the results of \citet[Theorem 7.2.13]{muirhead82} imply that the second expectation is finite whenever $\frac{n+p+1}4>2s+\frac{p-1}2 \Leftrightarrow n\geq p+8s-2$. Thus, in Equation \eqref{eqn:moments-H1-pre-convolution} with $l\leq k$ and \eqref{eqn:moments-H2-pre-convolution} with $l\leq 2k$, whenever $n\geq p+16k-2$ we are justified in repeatedly differentiating under the integral sign by the integrability bounds given by Equations \eqref{eqn:moments-integrability3} and \eqref{eqn:moments-integrability4}, and obtain in that case
{\setlength{\mathindent}{5pt}
\setlength{\abovedisplayskip}{0pt}
\setlength{\belowdisplayskip}{3pt}\begin{align}&
H_1(X) = 
	\frac{(-1)^k }{n^k}
	\bigints_{\substack{Y>0}}
	{\mathlarger{\mathlarger\sum_{i_1,\dots,i_{2k}}^p}}
	\frac{\partial_\text{s}}{\partial_\text{s} X_{i_1 i_{2k}}}
	\dots \frac{\partial_\text{s}}{\partial_\text{s} X_{i_3i_2}}
	\frac{\partial_\text{s}}{\partial_\text{s} X_{i_2i_1}}
\\[-5pt]&\pushright{
	\exp\Big\{\!\!-\!\frac{n}4\Tr(Y\!+\!X) \Big\}
	\big|Y\!+\!X\big|^{\frac{m}4}
	\frac{n^{\frac{np}2}}{2^{\frac{np}2} \Gamma_p\big(\frac{n}2\big)}
	\exp\Big\{\!\!-\!\frac{n}4\Tr Y \Big\}
	\big|Y\big|^{\frac{m}4}dY
}\label{eqn:moments-H1-prelimit-wishart}\end{align}}%
and
{\setlength{\mathindent}{5pt}
\setlength{\abovedisplayskip}{0pt}
\setlength{\belowdisplayskip}{3pt}\begin{align}&
H_2(X) = 
	\frac{(-1)^k}{n^{k}}
	\bigints_{\substack{Y>0}}
	{\mathlarger{\mathlarger\sum_{\substack{i_1,\dots,i_{k}\\j_1,\dots,j_{k}}}^p}}
	\frac{\partial_\text{s}}{\partial_\text{s} X_{j_1 j_{k}}}
	\dots \frac{\partial_\text{s}}{\partial_\text{s} X_{j_3j_2}}
	\frac{\partial_\text{s}}{\partial_\text{s} X_{j_2j_1}}
\notag\\[-3pt]&\pushright{
	\frac{\partial_\text{s}}{\partial_\text{s} X_{i_1 i_{k}}}
	\dots \frac{\partial_\text{s}}{\partial_\text{s} X_{i_3i_2}}
	\frac{\partial_\text{s}}{\partial_\text{s} X_{i_2i_1}}
	\exp\Big\{\!\!-\!\frac{n}4\Tr(Y\!+\!X) \Big\}
	\big|Y\!+\!X\big|^{\frac{m}4}
}\notag\\&\pushright{\cdot
	\frac{n^{\frac{np}2}}{2^{\frac{np}2} \Gamma_p\big(\frac{n}2\big)}
	\exp\Big\{\!\!-\!\frac{n}4\Tr Y \Big\}
	\big|Y\big|^{\frac{m}4}dY.
} \label{eqn:moments-H2-prelimit-wishart}
\end{align}}%

Let us now look at how $H_1(X)$ and $H_2(X)$ behave as $X\rightarrow0$. On one hand, for any symmetric matrix $T$ we have $|\Tr T^k|\leq \sqrt{p\Tr T^{2k}}\leq \sqrt{p}|I_p+T^2|^{k/2}$, so we must have the bounds
{\setlength{\mathindent}{10pt}
\setlength{\abovedisplayskip}{3pt}
\setlength{\belowdisplayskip}{3pt}\begin{align*}&
\bigg|\Tr T^{2k}e^{i\sqrt{n}\Tr(TX)}\big|\psi_\text{NW}(T)\big|\bigg|
\leq \frac{n^{k}C_{n, p}}{16^kC_{n-4k, p}} C_{n-4k, p}\bigg|I_p +\frac{16T^2}{n}\bigg|^{-\frac{(n-4k)+p+1}4}
\end{align*}}%
and
{\setlength{\mathindent}{10pt}
\setlength{\abovedisplayskip}{3pt}
\setlength{\belowdisplayskip}{3pt}\begin{align*}&
\bigg|\Tr^2 T^{k}e^{i\sqrt{n}\Tr(TX)}\big|\psi_\text{NW}(T)\big|\bigg|
\leq \frac{pn^kC_{n, p}}{16^kC_{n-4k, p}} C_{n-4k, p}\bigg|I_p +\frac{16T^2}{n}\bigg|^{-\frac{(n-4k)+p+1}4}.
\end{align*}}%
holding uniformly in $X$. When $n-4k\geq p-2 \Leftrightarrow n\geq p+4k-2$, the right hand sides are proportional to the density of the G-conjugates of the normalized Wishart distributions with $n-4k$ degrees of freedom, so are integrable. Thus, by the dominated convergence theorem and Equations \eqref{eqn:moments-H1-prelimit-cf} and \eqref{eqn:moments-H2-prelimit-cf},
{\setlength{\mathindent}{10pt}
\setlength{\abovedisplayskip}{0pt}
\setlength{\belowdisplayskip}{3pt}\begin{align}
\hspace{-5pt}\text{\raisebox{5pt}{$\lim_{\substack{X\rightarrow 0\\0<X<I_p}}$}}\hspace{-2pt} 
H_1(X) 
&= \!\!\bigintss_{\S_p(\R)}\hspace{-18pt}
	 \Tr T^{2k} 
	 \hspace{-5pt}\text{\raisebox{5pt}{$\lim_{\substack{X\rightarrow 0\\0<X<I_p}}$}}\hspace{-2pt} 
	 e^{i\sqrt{n}\Tr (TX)}\big|\psi_{\text{NW}}(T)\big|dT
= \E{\Tr T^{2k}}
\label{eqn:moments-H1-postlimit-cf}
\end{align}}%
and
{\setlength{\mathindent}{10pt}
\setlength{\abovedisplayskip}{0pt}
\setlength{\belowdisplayskip}{3pt}\begin{align}
\hspace{-5pt}\text{\raisebox{5pt}{$\lim_{\substack{X\rightarrow 0\\0<X<I_p}}$}}\hspace{-2pt} 
H_2(X) 
&= \!\!\bigintss_{\S_p(\R)}\hspace{-18pt}
	 \Tr^2 T^{k} 
	 \hspace{-5pt}\text{\raisebox{5pt}{$\lim_{\substack{X\rightarrow 0\\0<X<I_p}}$}}\hspace{-2pt} 
	 e^{i\sqrt{n}\Tr (TX)}\big|\psi_{\text{NW}}(T)\big|dT
= \E{\Tr^2 T^{k}}
\label{eqn:moments-H2-postlimit-cf}
\end{align}}%
for $T\sim T_{n/2}(I_/8)$.

On the other hand, the integrands at Equations \eqref{eqn:moments-H1-prelimit-wishart} and \eqref{eqn:moments-H2-prelimit-wishart} take a particularly simple form. Lemma \ref{lem:second-derivatives-lemma} establishes by induction that there must be polynomials $b_\kappa^{(1)}$ and $b_\kappa^{(2)}$ in $n$, $m$ and $p$ with degrees $\mathrm{deg}\,b_\kappa^{(1)}\leq 2k+1-q(\kappa)$ and $\mathrm{deg}\, b_\kappa^{(2)}\leq 2k+2-q(\kappa)$ such that
{\setlength{\mathindent}{10pt}
\setlength{\abovedisplayskip}{3pt}
\setlength{\belowdisplayskip}{3pt}\begin{align}&
H_1(X) = \frac{(-1)^k }{n^k}
	\bigint_{Y>0}\hspace{-15pt} 
	\raisebox{2pt}{$\mathlarger\sum\limits_{|\kappa|\leq 2k}$}
	b^{(1)}_\kappa(n,m,p)
	r_\kappa([Y+X]^{-1})
	\exp\bigg\{\!\!-\!\frac{n}4\Tr(Y+X)\bigg\}
\notag\\[-5pt]&\pushrightn{\cdot
	 \big|Y+X\big|^{\frac{m}4}
	\frac{n^{\frac{np}2}}{2^{\frac{np}2} \Gamma_p\big(\frac{n}2\big)}
	\exp\Big\{\!\!-\!\frac{n}4\Tr Y \Big\}
	\big|Y\big|^{\frac{m}4}dY
}\label{eqn:moments-H1-prelimit-cv}
\end{align}}%
and
{\setlength{\mathindent}{10pt}
\setlength{\abovedisplayskip}{3pt}
\setlength{\belowdisplayskip}{3pt}\begin{align}&
H_2(X) = \frac{(-1)^k}{n^{k}}
	\bigint_{Y>0}\hspace{-15pt} 
	\hspace{-2pt}\raisebox{2pt}{$\mathlarger\sum\limits_{|\kappa|\leq 2k+1}$} \hspace{-7pt}
	b^{(2)}_\kappa(n,m,p)
	r_\kappa([Y+X]^{-1})
	\exp\bigg\{\!\!-\!\frac{n}4\Tr(Y+X)\bigg\}
\notag\\[-5pt]&\pushrightn{\cdot
	 \big|Y+X\big|^{\frac{m}4}
	\frac{n^{\frac{np}2}}{2^{\frac{np}2} \Gamma_p\big(\frac{n}2\big)}
	\exp\Big\{\!\!-\!\frac{n}4\Tr Y \Big\}
	\big|Y\big|^{\frac{m}4}dY
}\label{eqn:moments-H2-prelimit-cv}
\end{align}}%
for any $0<X<I_p$ and $n\geq p+16k-2$. The sums are taken over all partitions of the integers $\kappa$ satisfying $|\kappa|\leq 2k$ and $|\kappa|\leq 2k+1$ respectively, including the empty partition. But for any integer partition $\kappa$, the bound
{\setlength{\mathindent}{10pt}
\setlength{\abovedisplayskip}{5pt}
\setlength{\belowdisplayskip}{5pt}\begin{align*}
r_\kappa([Y+X]^{-1})e^{-\frac{n}4\Tr (Y+X)}\big|Y+X\big|^{\frac{m}4}
\;\leq\;
\Tr^{|\kappa|}(Y^{-1})e^{-\frac{n}4\Tr Y}\big[ 1+\Tr Y \big]^{\frac{mp}4}
\end{align*}}%
holds uniformly in $0\leq X\leq I_p$. Thus for $|\kappa| \leq 2k+1$, the right hand side is integrable for $n\geq p+16k+6$, by the same argument as for Equation \eqref{eqn:moments-integrability4}. Thus for such $n$, by the dominated convergence theorem and Equations \eqref{eqn:moments-H1-prelimit-cv} and \eqref{eqn:moments-H2-prelimit-cv}, we obtain that
{\setlength{\mathindent}{10pt}
\setlength{\abovedisplayskip}{0pt}
\setlength{\belowdisplayskip}{0pt}\begin{align}&
\hspace{-5pt}\text{\raisebox{7pt}{$\lim_{\substack{X\rightarrow 0\\0<X<I_p}}$}}
	H_1(X) 
= \frac{(-1)^k}{n^k}
	\bigints_{Y>0}\hspace{-7pt} 
	\hspace{-5pt}\text{\raisebox{7pt}{$\lim_{\substack{X\rightarrow 0\\0<X<I_p}}$}} \;
	\mathlarger\sum_{|\kappa|\leq 2k} \hspace{-5pt}
	b^{(1)}_\kappa(n,m,p)r_\kappa([Y+X]^{-1})
\notag\\&\pushright{\cdot
	\exp\bigg\{\!\!-\!\frac{n}4\Tr(Y+X)\bigg\}
	 \big|Y+X\big|^{\frac{m}4}
	\frac{n^{\frac{np}2}}{2^{\frac{np}2} \Gamma_p\big(\frac{n}2\big)}
	\exp\Big\{\!\!-\!\frac{n}4\Tr Y \Big\}
	\big|Y\big|^{\frac{m}4}dY
}\notag\\&\hphantom{\hspace{-5pt}\text{\raisebox{5pt}{$\lim_{\substack{X\rightarrow 0\\0<X<I_p}}$}}\hspace{-2pt} H_1(X) }
= \frac{(-1)^k}{n^k}\!\!\sum_{|\kappa|\leq 2k} \hspace{-3pt}
	b^{(1)}_\kappa(n,m,p)
	\E{r_\kappa(Y^{-1})}
\label{eqn:moments-H1-postlimit-cv}
\end{align}}%
and
{\setlength{\mathindent}{10pt}
\setlength{\abovedisplayskip}{0pt}
\setlength{\belowdisplayskip}{3pt}\begin{align}&
\hspace{-5pt}\text{\raisebox{7pt}{$\lim_{\substack{X\rightarrow 0\\0<X<I_p}}$}}
	H_2(X) 
= \frac{(-1)^k}{n^{k}}
	\bigints_{Y>0}\hspace{-7pt} 
	\hspace{-5pt}\text{\raisebox{7pt}{$\lim_{\substack{X\rightarrow 0\\0<X<I_p}}$}} \;
	\mathlarger\sum_{|\kappa|\leq 2k+1} \hspace{-5pt}
	b^{(2)}_\kappa(n,m,p)
	r_\kappa([Y+X]^{-1})
\notag\\&\pushright{\cdot
	\exp\bigg\{\!\!-\!\frac{n}4\Tr(Y+X)\bigg\}
	 \big|Y+X\big|^{\frac{m}4}
	\frac{n^{\frac{np}2}}{2^{\frac{np}2} \Gamma_p\big(\frac{n}2\big)}
	\exp\Big\{\!\!-\!\frac{n}4\Tr Y \Big\}
	\big|Y\big|^{\frac{m}4}dY
}\notag\\&\hphantom{\hspace{-5pt}\text{\raisebox{5pt}{$\lim_{\substack{X\rightarrow 0\\0<X<I_p}}$}}\hspace{-2pt} H_2(X) }
= \frac{(-1)^k}{n^{k}}\!\!\sum_{|\kappa|\leq 2k+1} \hspace{-8pt}
	b^{(2)}_\kappa(n,m,p)
	\E{r_\kappa(Y^{-1})},
\label{eqn:moments-H2-postlimit-cv}
\end{align}}%
where $Y$ follows a $\text{W}_p(n,I_p/n)$ distribution. Combining Equations \eqref{eqn:moments-H1-postlimit-cf}--\eqref{eqn:moments-H2-postlimit-cf} with Equations \eqref{eqn:moments-H1-postlimit-cv}--\eqref{eqn:moments-H2-postlimit-cv} and Lemma \ref{lem:second-derivatives-lemma} concludes the proof.
\end{proof}


Something remarkable about Lemma \ref{lem:moments-as-expectation} is that it provides us with an algorithm to compute the moments of a symmetric $t$ distribution in terms of the moments of an inverse Wishart matrix. For example, when $k=1$, repeated differentiation yields that
{\setlength{\mathindent}{20pt}
\setlength{\abovedisplayskip}{5pt}
\setlength{\belowdisplayskip}{5pt}\begin{align}&
\sum_{i_1, i_2}^p 
	\frac{\partial_\text{s}}{\partial_\text{s} X_{i_1i_2}}
	\frac{\partial_\text{s}}{\partial_\text{s} X_{i_2i_1}}
	\exp\Big\{\!\!-\!\frac{n}4\Tr(Y\!+\!X) \Big\}
	\big|Y\!+\!X\big|^{\frac{m}4}
\notag\\&\hspace{30pt}
=\bigg[\frac{m(m-2)}{16}\Tr(Y+X)^{-2}-\frac{mn}{8}\Tr(Y+X)^{-1}+\frac{n^2p}{16}
\notag\\&\pushrightn{
	-\frac{m}8\Tr^2(Y+X)^{-1}\bigg] \exp\Big\{\!\!-\!\frac{n}4\Tr(Y\!+\!X) \Big\}
	\big|Y\!+\!X\big|^{\frac{m}4}.
\label{eqn:moments-exampleA1}
}\end{align}}%
We can recognize $\Tr(Y+X)^{-2}$ and $\Tr^2(Y+X)$ as the power sum polynomials $r_{(2)}([Y+X]^{-1})$ and $r_{(1,1)}([Y+X]^{-1})$ in the sense of Equation \eqref{def:powersum-polynomial}, so we must have $b_{(2)}^{(1)}=m(m-2)/16$, $b_{(1,1)}^{(1)}=-m/8$, $b_{(1)}^{(1)}=mn/8$ and $b_{\varnothing}^{(1)}=n^2p/16$ in the result of Lemma \ref{lem:second-derivatives-lemma}. Hence Lemma \ref{lem:moments-as-expectation} really tells us that whenever $n \geq p+22$, $\E{\Tr T^2}$ for $T\sim T_{n/2}(I_p/8)$ can be expressed as
{\setlength{\mathindent}{12pt}
\setlength{\abovedisplayskip}{5pt}
\setlength{\belowdisplayskip}{5pt}\begin{align}
\E{\Tr T^{2}} &= 
		-\frac{m(m-2)}{16n}\E{\Tr Y^{-2}}
		+\frac{m}{8n}\E{\Tr^2 Y^{-1}}
\notag\\&\hspace{121pt}
		+\frac{m}{8}\E{\Tr Y^{-1}}
		-\frac{np}{16}
\label{eqn:moments-exampleA2}
\end{align}}%
where $Y\sim\text{W}_p(n, I_p/n)$. 

Of course, this also works with square moments and higher $k$. For example, the same strategy for, say, square moments with $k=2$ yields that whenever $n \geq p+38$, $\E{\Tr^2 T^2}$ for $T\sim T_{n/2}(I_p/8)$ can be expressed as
{\setlength{\mathindent}{0pt}
\setlength{\abovedisplayskip}{5pt}
\setlength{\belowdisplayskip}{5pt}\begin{align}&
\E{\Tr^2T^{2}} = 
-\frac{m \left(m^{2} - 5 m + 10\right)}{16 n^{2}}\E{\Tr Y^{-4}} 
 +\frac{m \left(m - 2\right)}{4 n^{2}}\E{\Tr Y^{-3}\Tr Y^{-1}} 
\notag\\&\pushright{
 +\frac{m \left(m^{3} - 4 m^{2} + 20 m - 32\right)}{256 n^{2}}\E{\Tr^{2} Y^{-2}} 
   +\frac{m^{2}}{64 n^{2}}\E{\Tr^{4} Y^{-1}} 
}\notag\\&\pushright{
 -\frac{m \left(m^{2} - 2 m + 16\right)}{64 n^{2}}\E{\Tr Y^{-2}\Tr^{2} Y^{-1}}
  +\frac{m \left(m - 2\right)}{8 n}\E{\Tr Y^{-3}} 
}\notag\\&\pushright{
 -\frac{m \left(m^{2} - 2 m + 16\right)}{64 n}\E{\Tr Y^{-2}\Tr Y^{-1}} 
 +\frac{m^{2}}{32 n}\E{\Tr^{3} Y^{-1}} 
 }\notag\\&\pushright{
 +\frac{m \left(m p - 2 p - 8\right)}{128}\E{\Tr Y^{-2}} 
 -\frac{m \left(- m + p\right)}{64}\E{\Tr^{2} Y^{-1}} 
 }\notag\\&\pushrightn{
 -\frac{m n p}{64}\E{\Tr Y^{-1}} 
 +\frac{n^{2} p^{2}}{256}
}
\label{eqn:moments-exampleB2}
\end{align}}%
again where $Y\sim\text{W}_p(n, I_p/n)$. 

Unfortunately, as we consider larger orders, the repeated differentiation of $\exp\{-\frac{n}4\Tr Z\}|Z|^{m/4}$ quickly becomes too cumbersome to perform by hand. But at least in theory, we can compute expressions like \eqref{eqn:moments-exampleA2} and \eqref{eqn:moments-exampleB2} for any $\E{\Tr T^{2k}}$ and $\E{\Tr^2T^{k}}$, and Lemma \ref{lem:moments-as-expectation} summarize that fact. That is, using the Fourier inversion theorem we have reduced the problem of computing moments of the $t$-distribution $T_{n/2}(I_p/8)$ to that of computing expected power sum polynomials of the inverse Wishart distribution $\text{W}^{-1}_p(n, I_p/n)$, for large enough $n$.

How can we compute expected power sum polynomials of an inverse Wishart? There are two approaches in the literature. \citet{letac04} found an expression in terms of a different basis, the zonal polynomials, which behave particularly nicely with respect to the inverse Wishart distribution, and whose expectations have a simple closed form. From this, they provided an algorithm for computing expected power sum polynomials to arbitrary order. \citet{matsumoto12} found expressions of coordinate-wise moments in terms of modified Wiengarten orthogonal functions, from which expectations of power sum polynomials can be computed. We follow the idea of \citet{letac04} in our asymptotic analysis.

For any integer partition $\kappa$, there exist coefficients $c_{\kappa, \lambda}$ (which depend solely on $\kappa$ and $\lambda$) such that
{\setlength{\mathindent}{30pt}
\setlength{\abovedisplayskip}{3pt}
\setlength{\belowdisplayskip}{3pt}\begin{align*}
r_\kappa(Y^{-1}) = \sum_{|\lambda|=|\kappa|} c_\lambda C_\lambda(Y^{-1})
\end{align*}}%
for $C_\lambda$ the so-called zonal polynomials. For an overview of the topic with a focus on random matrix theory, see \citet[Chapter 7]{muirhead82}. The coefficients $c_{\kappa, \lambda}$ are explicitly computable. If we follow the normalization of zonal polynomials of \cite{muirhead82}, for example, we find that
{\setlength{\mathindent}{10pt}
\setlength{\abovedisplayskip}{7pt}
\setlength{\belowdisplayskip}{7pt}\begin{align}
\begin{array}{ccc}
\begin{bmatrix}r_{\varnothing}\end{bmatrix} = \begin{bmatrix}C_{\varnothing}\end{bmatrix}\!,
\vspace{5pt}
&
\multirow{2}{*}{$\quad\text{and}\quad$\vspace{-5pt}} &
\multirow{2}{*}{$\begin{bmatrix}
r_{(2)} \\
r_{(1,1)}
\end{bmatrix}
= \begin{bmatrix}
1 & -\frac12 \\
1 & 1
\end{bmatrix}
\begin{bmatrix}
C_{(2)} \\
C_{(1,1)}
\end{bmatrix}. $\vspace{-5pt}}\\
\begin{bmatrix}r_{(1)}\end{bmatrix} = \begin{bmatrix}C_{(1)}\end{bmatrix}\!,
& &
\end{array}
\label{eqn:moments-example-RC}
\end{align}}%
As mentioned, expectations of zonal polynomials with respect to a Wishart or inverse Wishart distribution take a particularly simple form. From \citet[Theorem 7.2.13 and Equation (18) on p.237]{muirhead82}, the expected zonal polynomials for $Y^{-1}\sim\text{W}^{-1}_p(n, I_p/n)$ are
{\setlength{\mathindent}{10pt}
\setlength{\abovedisplayskip}{0pt}
\setlength{\belowdisplayskip}{7pt}\begin{align}&
\E{C_\lambda(Y^{-1})}
= \frac{n^{|\lambda|}}{2^{|\lambda|}\prod_{i=1}^{q(\lambda)}\frac{m-i+1}{2}}
	C_\lambda(I_p)
\notag\\&\hspace{10pt}
= \frac{2^{|\lambda|}|\lambda|!\prod_{i<j}^{q(\lambda)} (2\lambda_i-2\lambda_j-i+j)}{\prod_{i=1}^{q(\lambda)}(2\lambda_i+q(\lambda)-i)!}
n^{|\lambda|}\prod_{i=1}^{q(\lambda)}\prod_{l=0}^{\lambda_i-1}\frac{p+(1-i+2l)}{m-(1-i+2l)}
\label{eqn:moments-expected-zonal}
\end{align}}%
for $\lambda\neq \varnothing$, and $\E{C_{\varnothing}(Y^{-1})}=1$. For example, the first few expected zonal polynomials are
{\setlength{\mathindent}{10pt}
\setlength{\abovedisplayskip}{5pt}
\setlength{\belowdisplayskip}{7pt}\begin{align*}
\E{C_{\varnothing}(Y^{-1})} &= 1,
&
\E{C_{(1)}(Y^{-1})} &= \frac{np}{m}, \\
\E{C_{(1,1)}(Y^{-1})} &= \frac{2 n^{2} p \left(p - 1\right)}{3 m \left(m + 1\right)},
&
\E{C_{(2)}(Y^{-1})} &= \frac{n^{2} p \left(p + 2\right)}{3 m \left(m - 2\right)}.
\end{align*}}%
From this, we can exactly compute $\E{r_\kappa(Y^{-1})}$ and thus $\E{\Tr T^{2k}}$ and $\E{\Tr ^2T^{k}}$, as a function of $p$ and $n$ (or $m$). For example, by Equation \eqref{eqn:moments-example-RC} we find that
{\setlength{\mathindent}{0pt}
\setlength{\abovedisplayskip}{5pt}
\setlength{\belowdisplayskip}{7pt}\begin{align*}
\E{r_{\varnothing}(Y^{-1})} &= 1,
&
\E{r_{(1)}(Y^{-1})} &= \frac{np}{m}, \\
\E{r_{(1,1)}(Y^{-1})} &= \frac{n^{2} p \left(m p - p + 2\right)}{m \left(m - 2\right) \left(m + 1\right)},
&
\E{r_{(2)}(Y^{-1})} &= \frac{n^{2} p \left(m + p\right)}{m \left(m - 2\right) \left(m + 1\right)}.
\end{align*}}%
and thus, by Equation \eqref{eqn:moments-exampleA2}, whenever $n\geq p+38$
{\setlength{\mathindent}{10pt}
\setlength{\abovedisplayskip}{5pt}
\setlength{\belowdisplayskip}{7pt}\begin{align}
\E{\Tr T^2} = \frac{np \left(m p + m + 2\right)}{16 \left(m - 2\right) \left(m + 1\right)}.
\label{eqn:moments-exampleA3}
\end{align}}%
In a similar way, we can compute the expected zonal polynomials and hence, the expected power sum polynomials of $Y^{-1}$ for $|\kappa|=3,4$. So from Equation \eqref{eqn:moments-exampleB2}, we obtain for $n\geq p+38$ that
{\setlength{\mathindent}{10pt}
\setlength{\abovedisplayskip}{5pt}
\setlength{\belowdisplayskip}{5pt}\begin{align}
\E{\Tr^2 T^2} =& \frac{n^{2} p}
{256 \left(m - 6\right) \left(m - 2\right) \left(m - 1\right) \left(m + 1\right) \left(m + 3\right)} 
\Big(m^{3} p^{3} + 2 m^{3} p^{2}
\notag\\&
+ 5 m^{3} p + 4 m^{3} - 3 m^{2} p^{3} + 6 m^{2} p^{2} + 9 m^{2} p + 24 m^{2} - 12 m p^{3} 
\notag\\&
- 36 m p^{2} + 36 m p - 36 p\Big).
\label{eqn:moments-exampleB3}
\end{align}}%
Of course, this reasoning also works for other $k$'s. In particular, we essentially derived a (potentially inefficient) algorithm to compute the moments of a $T_{n/2}(I_p/8)$ distribution to arbitrary order on our path to proving this theorem.

At this point, it is worthwhile to realize that Equations \eqref{eqn:moments-exampleA3} and \eqref{eqn:moments-exampleB3} are already enough to prove the theorem for small moments. For example, when $n\rightarrow\infty$ such that $p/n\rightarrow c\in[0,1)$, then $m\sim (1-c)n$ and
{\setlength{\mathindent}{17pt}
\setlength{\abovedisplayskip}{5pt}
\setlength{\belowdisplayskip}{5pt}\begin{align*}&
\E{\Tr T^2}
= \frac{m^2}{\left(m - 2\right) \left(m + 1\right)}\left(
\frac{np^2}{16m}
+ \frac{np}{16m}
+ \frac{n}{8m^2}\right)
\sim \frac1{16(1-c)}p^2,
\end{align*}}%
which proves that $\E{\Tr T^2}=O(p^2)$. In fact,
{\setlength{\mathindent}{10pt}
\setlength{\abovedisplayskip}{5pt}
\setlength{\belowdisplayskip}{5pt}\begin{align}
\E{\frac1{p}\Tr \Big(\frac{4T}{\sqrt{p}}\Big)^2} - 1
&= \frac{m^2}{\left(m\!-\!2\right) \left(m\!+\!1\right)}
\bigg(\frac{1}{p} + \frac{p}{m} + \frac{2}{m} 
+ \frac{3}{m p} + \frac{2}{m^{2}} + \frac{2}{m^{2} p}\bigg)
\notag\\&
\sim \frac1{p}+\frac{p}{m}\rightarrow 0.
\label{eqn:moments-example-bias}
\end{align}}%
Moreover, when $n, p\rightarrow\infty$ such that $p/n\rightarrow0$, then $m\sim n$ and
{\setlength{\mathindent}{10pt}
\setlength{\abovedisplayskip}{5pt}
\setlength{\belowdisplayskip}{5pt}\begin{align}&
\E{\bigg(\frac1{p}\Tr \Big(\frac{4T}{\sqrt{p}}\Big)^2 - 1\bigg)^{\raisebox{-3pt}{$\scriptstyle 2$}}}
=\frac{256}{p^4}\E{\Tr^2 T^2}
-\frac{32}{p^2}\E{\Tr T^2}^2+1
\notag\\&\hspace{30pt}
= \frac{m^5}{\left(m\!-\!6\right) \left(m\!-\!2\right) \left(m\!-\!1\right) \left(m\!+\!1\right) \left(m\!+\!3\right)}
\bigg(
\frac{5}{p^{2}} 
+ \frac{4}{p^{3}} 
+ \frac{2}{m} 
+ \frac{22}{m p} 
\notag\\&\hspace{50pt}
+ \frac{27}{m p^{2}} 
+ \frac{32}{m p^{3}} 
+ \frac{p^{2}}{m^{2}} 
+ \frac{6 p}{m^{2}}
+ \frac{33}{m^{2}} 
+ \frac{60}{m^{2} p} 
+ \frac{115}{m^{2} p^{2}} 
+ \frac{52}{m^{2} p^{3}} 
\notag\\&\hspace{50pt}
- \frac{3 p^{2}}{m^{3}} 
+ \frac{6 p}{m^{3}} 
+ \frac{3}{m^{3}} 
+ \frac{118}{m^{3} p} 
+ \frac{93}{m^{3} p^{2}} 
+ \frac{24}{m^{3} p^{3}} 
- \frac{12 p^{2}}{m^{4}} 
- \frac{96 p}{m^{4}} 
\notag\\&\hspace{50pt}
- \frac{48}{m^{4}} 
- \frac{84}{m^{4} p} 
- \frac{36}{m^{4} p^{2}} 
- \frac{144}{m^{5}} 
- \frac{144}{m^{5} p} 
- \frac{36}{m^{5} p^{2}}
\bigg)
\notag\\&\hspace{30pt}
\sim \frac{5}{p^2} + \frac{2}{m} + \frac{p^2}{m^2} \rightarrow 0.
\label{eqn:moments-example-L2distance}
\end{align}}%
Thus
\begin{ceqn}
\begin{align*}
\frac1p\Tr\Big(\frac{4T}{\sqrt{p}}\Big)^{2} \overset{\text{L}^2}{\longrightarrow} 1=C_1
\end{align*}
\end{ceqn}
and the theorem is proven for the second moment. 

In theory, we could proceed in the same way for any moment of interest, but naturally we could never conclude that the theorem holds for all moments that way. Nonetheless, the calculations give us some hints about how to argue in the general case.

The idea is to express the moments of the symmetric $t$ distribution as polynomials of $p$ and $p/m$. There are two regimes where random matrix theory is well understood: the classical regime where $p$ is held fixed as $n\rightarrow\infty$, and the linear, high-dimensional regime where $p$ grows linearly with $n$. From this, we can therefore conclude a few facts regarding the behavior of symmetric $t$ moments in these regimes. But these moments are polynomials, and a polynomial is a very rigid object: results from the two extreme cases where $p$ is fixed and $p$ grows linearly will be enough to prove results for every regime in between, yielding the first part of the theorem. Proving the second part will then be the simple matter of applying the GOE approximation of \cite{jiang15} and \cite{bubeck16a} to the specific shape found for the symmetric $t$ moments while proving the first part, namely Equations \eqref{eqn:moments-trT2k-limit} and \eqref{eqn:moments-tr2Tk-limit}.

\noindent \begin{proof}[Proof of Theorem \ref{thm:moments}]
Recall the expected zonal polynomial of an inverse Wishart $\text{W}^{-1}_p(n, I_p/n)$ is given by Equation \eqref{eqn:moments-expected-zonal}. Based on the previous calculations, it is tempting to define
{\setlength{\mathindent}{0pt}
\setlength{\abovedisplayskip}{3pt}
\setlength{\belowdisplayskip}{3pt}\begin{align}&
c_\lambda' = \frac{2^{|\lambda|}|\lambda|!\prod_{i<j}^{q(\lambda)} (2\lambda_i-2\lambda_j-i+j)}{\prod_{i=1}^{q(\lambda)}(2\lambda_i+q(\lambda)-i)!},\quad
R_\lambda(m) = \prod_{i=1}^{q(\lambda)}\prod_{l=0}^{\lambda_i-1}\frac{m}{m-(1-i+2l)}
\notag\\&\text{and}\hspace{80pt}
P_\lambda(m, p) = \prod_{i=1}^{q(\lambda)}
\prod_{l=0}^{\lambda_i-1}\left(\frac{p}{m}+\frac{1-i+2l}{m}\right)
\label{def:moments-R}
\end{align}}%
so that
\begin{ceqn}
{\setlength{\mathindent}{10pt}
\setlength{\abovedisplayskip}{3pt}
\setlength{\belowdisplayskip}{3pt}\begin{align*}&
\E{C_\lambda(Y^{-1})} = c'_\lambda n^{|\lambda|} R_\lambda(m) P_\lambda(m, p).
\end{align*}}%
\end{ceqn}
With these expressions the expected power sum polynomials can be written as
{\setlength{\mathindent}{10pt}
\setlength{\abovedisplayskip}{3pt}
\setlength{\belowdisplayskip}{3pt}\begin{align*}&
\E{r_\kappa(Y^{-1})} = \sum_{|\lambda|=|\kappa|}^{p}
	c_{\kappa,\lambda}c'_\lambda n^{|\kappa|}
	\Bigg(
		\prod_{|\mu|=|\kappa|} R_\mu(m)
		\prod_{\substack{|\mu|=|\kappa|\\\mu\not=\lambda}} R_\mu^{-1}(m)
	\Bigg)
	P_\lambda(m,p)
\\&\hspace{40pt}
= \frac{n^{|\kappa|}}{m^{|\kappa|}}\cdot 
	\prod_{|\mu|=|\kappa|}R_\mu(m)
	\cdot m^{|\kappa|}
	\hspace{-5pt}\sum_{|\lambda|=|\kappa|}\hspace{-5pt} c_{\kappa,\lambda}c'_\lambda
	\prod_{\substack{|\mu|=|\kappa|\\\mu\not=\lambda}} R_\mu^{-1}(m) P_{\lambda}(m, p).
\end{align*}}%
In other words, if we define
{\setlength{\mathindent}{40pt}
\setlength{\abovedisplayskip}{3pt}
\setlength{\belowdisplayskip}{0pt}\begin{align}
R'_{|\mu|} &= \prod_{|\mu|=|\kappa|}R_\mu(m),
\notag\\
P'_{|\lambda|}(m, p) &= m^{|\kappa|}
	\hspace{-5pt}\sum_{|\lambda|=|\kappa|}\hspace{-5pt} c_{\kappa,\lambda}c'_\lambda
	\prod_{\substack{|\mu|=|\kappa|\\\mu\not=\lambda}} R_\mu^{-1}(m) P_{\lambda}(m, p)
\label{def:moments-Rprime}
\end{align}}%
then
{\setlength{\mathindent}{10pt}
\setlength{\abovedisplayskip}{3pt}
\setlength{\belowdisplayskip}{3pt}\begin{align}&\hspace{40pt}
\E{r_\kappa(Y^{-1})} = \frac{n^{|\kappa|}}{m^{|\kappa|}}
	R'_{|\kappa|}(m)
	P'_{|\kappa|}(m,p).
\label{eqn:moments-invariant-using-RPprime}
\end{align}}%
But $R^{-1}_\mu(m)=\prod_{i=1}^{q(\lambda)}\prod_{l=0}^{\lambda_i-1}\big(1-\frac{1-i+2l}{m}\big)$ is a polynomial in $1/m$, while $P_\lambda(m,p) = \prod_{i=1}^{q(\lambda)} \prod_{i=1}^{\lambda_i-1}\left(\frac{p}{m}+\frac{1-i+2l}{m}\right)$ is a polynomial in $p/m$ and $1/m$, both of degree at most $|\mu|=|\lambda|=|\kappa|$. Thus
{\setlength{\mathindent}{20pt}
\setlength{\abovedisplayskip}{7pt}
\setlength{\belowdisplayskip}{3pt}\begin{align}
P'_{|\kappa|}(m, p) &\equiv m^{|\kappa|}\sum_{|\lambda|=|\kappa|}c_{\kappa,\lambda}c'_\lambda
	\prod_{\substack{|\mu|=|\kappa|\\\mu\not=\lambda}} R_\mu^{-1}(m) P_{\lambda}(m, p)
\notag\\&
= m^{|\kappa|}\sum_{i=0}^{|\kappa|}\sum_{j=0}^{|\kappa|}b_{ij}\Big(\frac{p}{m}\Big)^i\frac1{m^j}
\label{eqn:moments-Pprime-as-poly}
\end{align}}%
for some coefficients $b_{ij}$ that don't depend on $m$, $p$ (or $n$). Define the polynomials $f_j(\alpha)=\sum_{i=0}^{|\kappa|}b_{ij}\alpha^i$, so that
{\setlength{\mathindent}{20pt}
\setlength{\abovedisplayskip}{3pt}
\setlength{\belowdisplayskip}{3pt}\begin{align}
P'_{|\kappa|}(m, p)
= m^{|\kappa|}\sum_{j=0}^{|\kappa|} f\Big(\frac{p}{m}\Big) m^{-j}.
\label{eqn:moments-Pprime-with-f}
\end{align}}%
Let us show that for all $0\leq j < |\kappa|-q(\kappa)$, the polynomial $f_j$ must be identically zero over the interval $\alpha\in\big(0, 1/\max(|\kappa|-2,0)\big)$. Indeed, say this was not the case, and let $0\leq j_0 < |\kappa|-q(\kappa)$ be the smallest $j$ with the property that $f_{j_0}(\alpha_0)\neq0$ for some $\alpha_0\in \big(0, \frac1{\max(|\kappa|-2, 0)}\big)$. As $f_{j_0}$ is a polynomial, by continuity it must be non-zero in a neighborhood of $\alpha_0$, so we may as well assume $\alpha_0$ is rational without loss of generality. Now look at what happens to $\E{r_\kappa(Y^{-1})}$ as $p$ grows to infinity at the very specific linear rate $p=\lfloor \frac{\alpha_0}{1+\alpha_0}(n-1)\rfloor$. Since $\alpha_0$ is rational, there must be a subsequence $n_l$ such that $p_l$ is exactly an integer (for example, if $\alpha_0=a/b$ with $a$, $b$ integers, we can take $n_l=(a+b)l+1$). Then for $p_l=\frac{\alpha_0}{1+\alpha_0}(n_l-1)$, we have exactly $p_l=\alpha_0m_l$.

Since $\alpha_0<\frac1{\max(|\kappa|-2, 0)}$, then $|\kappa| < 1+\big(\frac{\alpha_0}{1+\alpha_0}\big)^{-1}$.  Thus by H\"older's inequality and Lemma \ref{lem:hd-wishart},
{\setlength{\mathindent}{20pt}
\setlength{\abovedisplayskip}{5pt}
\setlength{\belowdisplayskip}{5pt}\begin{align*}
\lim_{l\rightarrow\infty}\frac1{m_l^{|\kappa|-q(\kappa)-j_0}}
\cdot\frac1{p_l^{q(\kappa)}}\E{r_\kappa(Y^{-1})}
\leq 0\cdot\lim_{l\rightarrow\infty}\frac1{p_l}\E{\Tr Y^{-|\kappa|}} = 0.
\end{align*}}%
On the other hand, by Equations \eqref{eqn:moments-invariant-using-RPprime} and \eqref{eqn:moments-Pprime-with-f}, the definition of $j_0$ and the fact that $R_{|\kappa|}(m)\rightarrow1$ as $m\rightarrow\infty$,
{\setlength{\mathindent}{20pt}
\setlength{\abovedisplayskip}{5pt}
\setlength{\belowdisplayskip}{5pt}\begin{align*}&
\lim_{l\rightarrow\infty}
	\frac1{m_l^{|\kappa|-q(\kappa)-j_0}}
\cdot\frac1{p_l^{q(\kappa)}}\E{r_\kappa(Y^{-1})}
\\&\qquad
= \lim_{l\rightarrow\infty}
	\left(\frac{n_l}{m_l}\right)^{|\kappa|}
	R'_{|\kappa|}(m_l)
	\left(\frac{m_l}{p_l}\right)^{q(\kappa)}
	\sum_{j=j_0}^{|\kappa|} f_j(\alpha_0)m_l^{j_0-j}
\\&\qquad
= (1+\alpha_0)^{|\kappa|}\alpha_0^{-q(\kappa)} f_{j_0}(\alpha_0).
\end{align*}}%
As $\alpha_0>0$, $f_{j_0}(\alpha_0)$ must therefore equal zero, a contradiction. Hence, as claimed, the polynomials $f_j(\alpha)$ for $0\leq j < |\kappa|-q(\kappa)$ all vanish over the interval $\big(0, \frac1{\max(|\kappa|-2, 0)}\big)$. 

But a polynomial can have an infinite number of zeros only if all its coefficients are zero, so we conclude that
\begin{ceqn}
{\setlength{\abovedisplayskip}{5pt}
\setlength{\belowdisplayskip}{5pt}\begin{align*}&
b_{ij} = 0
\qquad\text{for}\quad 0\leq j<|\kappa|-q(\kappa).
\end{align*}}%
\end{ceqn}
Thus, from Equations \eqref{eqn:moments-invariant-using-RPprime} and \eqref{eqn:moments-Pprime-as-poly} we have
{\setlength{\mathindent}{20pt}
\setlength{\abovedisplayskip}{3pt}
\setlength{\belowdisplayskip}{3pt}\begin{align*}
\E{r_\kappa(Y^{-1})} = \left(\frac{n}{m}\right)^{|\kappa|} 
	m^{q(\kappa)} 
	R'_{|\kappa|}(m)
	P''_{\kappa}(m, p)
\end{align*}}%
where
{\setlength{\mathindent}{20pt}
\setlength{\abovedisplayskip}{3pt}
\setlength{\belowdisplayskip}{3pt}\begin{align*}
P''_{\kappa}(m, p) =
	\sum_{i=0}^{|\kappa|}\sum_{j=|\kappa|-q(\kappa)}^{|\kappa|}\hspace{-5pt}
	b_{ij} \left(\frac{p}{m}\right)^i \frac1{m^{j-|\kappa|+q(\kappa)}}.
\end{align*}}%

Going back to equations \eqref{eqn:moments-trT2k-as-polynomial} and \eqref{eqn:moments-tr2Tk-as-polynomial} and plugging in the above yields that as long as $n\geq p+16k+6$,
{\setlength{\mathindent}{20pt}
\setlength{\abovedisplayskip}{3pt}
\setlength{\belowdisplayskip}{3pt}\begin{align}
\E{\Tr T^{2k}} &= \frac{m^{2k+1}}{n^{k}} R'_{2k}(m) Q^{(1)}_\kappa(m,p),
\label{eqn:moments-trT2k-as-Q}
\\
\E{\Tr^2 T^{k}} &= \frac{m^{2k+2}}{n^{k}} R'_{2k+1}(m) Q^{(2)}_\kappa(m,p)
\label{eqn:moments-tr2Tk-as-Q}
\end{align}}%
where
{\setlength{\mathindent}{20pt}
\setlength{\abovedisplayskip}{3pt}
\setlength{\belowdisplayskip}{3pt}\begin{align*}
Q^{(1)}_\kappa(m,p) &=
(-1)^k\hspace{-6pt}\sum_{|\kappa|\leq 2k}\hspace{-2pt} \Big(1\!+\!\frac{p}{m}\!+\!\frac{1}{m}\Big)^{|\kappa|}
\frac{R'_{|\kappa|}(m)}{R'_{2k}(m)} \frac{b^{(1)}_{\kappa}(n, m, p)}{m^{2k+1-q(\kappa)}} P''_\kappa(m,p),
\\
Q^{(2)}_\kappa(m,p) &=
(-1)^k\hspace{-9pt}\sum_{|\kappa|\leq 2k+1}\hspace{-5pt} \Big(1\!+\!\frac{p}{m}\!+\!\frac{1}{m}\Big)^{|\kappa|}
\frac{R'_{|\kappa|}(m)}{R'_{2k+1}(m)} \frac{b^{(2)}_{\kappa}(n, m, p)}{m^{2k+2-q(\kappa)}} P''_\kappa(m,p).
\end{align*}}%
Now, for any $a\leq b$, we can associate a partition $\mu$ of norm $|\mu|=a$ with the partition $\mu^*=(\mu_1+b-a, \mu_+2,\dots,\mu_{q(\mu)})$ of norm $|\mu^*|=b$, which satisfies
{\setlength{\mathindent}{10pt}
\setlength{\abovedisplayskip}{3pt}
\setlength{\belowdisplayskip}{3pt}\begin{align*}
\prod_{i=1}^{q(\mu^*)}\prod_{j=0}^{\mu^*_i-1} \bigg(1-\frac{1\!-\!i\!+\!2j}{m}\bigg)
= \prod_{i=1}^{q(\mu)}\prod_{j=0}^{\mu_i-1} \bigg(1-\frac{1\!-\!i\!+\!2j}{m}\bigg)
\prod_{j=\mu_1}^{\substack{\mu_1+b\\-a-1}} \bigg(1-\frac{2j}{m}\bigg).
\end{align*}}%
By definition for the $R_\mu(m)$'s at Equation \eqref{def:moments-R}, this means that every factor that appears in $R^{-1}_\mu(m)$ appears in $R^{-1}_{\mu^*}(m)$, so by definition of the $R_{|\mu|}(m)$'s at Equation \eqref{def:moments-Rprime}, $R_a(m)R^{-1}(m)$ is a polynomial in $\frac1{m}$. Moreover, as $b^{(1)}_\kappa$ and $b_\kappa^{(2)}$ are polynomials of degrees $d_1(\kappa)\equiv 2k+1-q(\kappa)$ and $d_2(\kappa)\equiv 2k+2-q(\kappa)$ respectively, there exists coefficients $c^{(1)}_{ijl}$ and $c^{(2)}_{ijl}$ such that
{\setlength{\mathindent}{20pt}
\setlength{\abovedisplayskip}{0pt}
\setlength{\belowdisplayskip}{3pt}\begin{align*}
\frac{b_\kappa^{(1)}(n,m,p)}{m^{2k+1-q(\kappa)}} &= \frac{1}{m^{d_1(\kappa)}}
\sum_{i=0}^{d_1(\kappa)}
\sum_{j=0}^{\substack{d_1(\kappa)\\-i}}
\sum_{l=0}^{\substack{d_1(\kappa)\\-i-j}}
c_{ijl}^{(1)}m^in^jp^l
\\&
= \sum_{i=0}^{d_1(\kappa)}
\sum_{j=0}^{\substack{d_1(\kappa)\\-i}}
\sum_{l=0}^{\substack{d_1(\kappa)\\-i-j}}
c^{(1)}_{ijl}\frac{1}{m^{d_1(\kappa)-i-j-l}}
\Big(1\!+\!\frac{p}{m}\!+\!\frac{1}{m}\Big)^j
\Big(\frac{p}{m}\Big)^l
\end{align*}}%
and
{\setlength{\mathindent}{20pt}
\setlength{\abovedisplayskip}{0pt}
\setlength{\belowdisplayskip}{3pt}\begin{align*}
\frac{b_\kappa^{(2)}(n,m,p)}{m^{2k+2-q(\kappa)}} &= \frac{1}{m^{d_2(\kappa)}}
\sum_{i=0}^{d_2(\kappa)}
\sum_{j=0}^{\substack{d_2(\kappa)\\-i}}
\sum_{l=0}^{\substack{d_2(\kappa)\\-i-j}}
c_{ijl}^{(2)}m^in^jp^l
\\&
= \sum_{i=0}^{d_2(\kappa)}
\sum_{j=0}^{\substack{d_2(\kappa)\\-i}}
\sum_{l=0}^{\substack{d_2(\kappa)\\-i-j}}
c^{(2)}_{ijl}\frac{1}{m^{d_2(\kappa)-i-j-l}}
\Big(1\!+\!\frac{p}{m}\!+\!\frac{1}{m}\Big)^j
\Big(\frac{p}{m}\Big)^l.
\end{align*}}%
As $d_1(-i-j-l\geq0), j, l\geq0$ in the first case and $d_2(\kappa)-i-j-l,j,l\geq0$ in the second case, we conclude that these two expressions are polynomials in $\frac{p}{m}$ and $\frac{1}{m}$. Therefore, looking back at \eqref{eqn:moments-trT2k-as-Q} and \eqref{eqn:moments-tr2Tk-as-Q}, we conclude that the $Q^{(1)}_\kappa(m,p)$'s and  $Q^{(2)}_\kappa(m,p)$'s are polynomials in $\frac{p}{m}$ and $\frac{1}{m}$. Therefore, if $n\geq p+16k+6$ there must be coefficients $a^{(1)}_{ij}$, $a^{(2)}_{ij}$ and large enough integers $D_1$, $D_2$ such that
{\setlength{\mathindent}{0pt}
\setlength{\abovedisplayskip}{3pt}
\setlength{\belowdisplayskip}{3pt}\begin{align}&\hspace{20pt}
\E{\Tr T^{2k}} = \frac{m^{2k+1}}{n^k}R_{2k}(m)
\sum_{i=0}^{D_1}\sum_{j=0}^i a^{(1)}_{ij}\frac{p^j}{m^i}
\notag\\[-5pt]&\pushrightn{
= \Big(\frac{m}{n}\Big)^{k}R_{2k}(m)\sum_{i=0}^{D_1}g_i^{(1)}(p)m^{k+1-i}
} \label{eqn:moments-trT2k-as-g}
\\[-5pt]&\hspace{20pt}
\E{\Tr^2 T^{k}} = \frac{m^{2k+2}}{n^{k}}R_{2k+1}(m)
\sum_{i=0}^{D_2}\sum_{j=0}^i a^{(2)}_{ij}\frac{p^j}{m^i}
\notag\\[-5pt]&\pushrightn{
= \Big(\frac{m}{n}\Big)^{k}R_{2k+1}(m)\sum_{i=0}^{D_2}g_i^{(2)}(p)m^{k+2-i}
}\label{eqn:moments-tr2Tk-as-g}
\end{align}}%
for polynomials $g_i^{(1)}(p)=\sum\limits_{j=0}^{i}a^{(1)}_{ij}p^j$ and $g_i^{(2)}(p)=\sum\limits_{j=0}^ia^{(2)}_{ij}p^j$.

We will now proceed to show that $g_i^{(1)}$ and $g_i^{(2)}$ must vanish on $\N$ for $0\leq i_0<k+1$ and $0\leq i_0 < k+2$ respectively. Our argument relies on the analysis of the asymptotic behavior of the moments of $T$ in the classical regime where $p$ is held fixed while $n$ grows to infinity. 

Observe first that $\E{\Tr T^{2k}}$ and $\E{\Tr^2 T^{k}}$ must have a finite limit as $n\rightarrow\infty$ with $p$ held fixed. Indeed, since $16T^2/n$ is positive definite, $|I_p+16T^2/n|$ is greater than one and so we have the bound
{\setlength{\mathindent}{10pt}
\setlength{\abovedisplayskip}{3pt}
\setlength{\belowdisplayskip}{3pt}\begin{align*}&
\E{\Tr T^{2k}}
= C_{n,p} \bigints_{\S_p(\R)}\hspace{-15pt}
\Tr T^{2k} \left|I_p+\frac{16T^2}{n}\right|^{-\frac{n+p+1}4}dT
\\[-7pt]&\pushright{
\leq\;
C_{n,p} \bigints_{\S_p(\R)}\hspace{-15pt}
\Tr T^{2k} \left|I_p+\frac{16T^2}{n}\right|^{-\frac{n}4}dT.
}\end{align*}}%
When $p$ is held fixed, $\lim\limits_{n\rightarrow\infty}C_{n,p} = 2^{\frac{p(3p+1)}4}/\pi^{\frac{p(p+1)}4}$ by Lemma \ref{lem:cnp}. Moreover,
{\setlength{\mathindent}{40pt}
\setlength{\abovedisplayskip}{0pt}
\setlength{\belowdisplayskip}{3pt}\begin{align*}
\left|I_p+\frac{16T^2}{n}\right|^{-\frac{n}4}
= \prod_{i=1}^p\left(1+\frac{\lambda_i(4T^2)}{n/4}\right)^{-\frac{n}4}
\end{align*}}%
for $\lambda_1(4T^2)\geq \dots \geq \lambda_p(4T^2)\geq0$ the eigenvalues of $4T^2$, and  $(1+x/n)^{-n}$ is monotone decreasing towards $\exp(x)$. Therefore, for a fixed dimension $p$ we can apply the monotone convergence theorem to obtain that
{\setlength{\mathindent}{5pt}
\setlength{\abovedisplayskip}{3pt}
\setlength{\belowdisplayskip}{3pt}\begin{align}
\lim_{n\rightarrow\infty}\E{\Tr T^{2k}}
\;\leq\;
\frac{2^{\frac{p(3p+1)}4}}{\pi^{\frac{p(p+1)}4}} \bigints_{\S_p(\R)}\hspace{-20pt}
\Tr T^{2k} e^{-4\Tr T^2}dT
= \E{\Tr Z^{2k}}<\infty
\label{eqn:moments-trT2-finite-classical}
\end{align}}%
for $Z\sim\text{GOE}(p)/4$. Repeating the argument with $\Tr^2T^{k}$ yields similarly
{\setlength{\mathindent}{5pt}
\setlength{\abovedisplayskip}{3pt}
\setlength{\belowdisplayskip}{3pt}\begin{align}
\lim_{n\rightarrow\infty}\E{\Tr^2 T^{k}}
\;\leq\;
\frac{2^{\frac{p(3p+1)}4}}{\pi^{\frac{p(p+1)}4}} \bigints_{\S_p(\R)}\hspace{-20pt}
\Tr^2 T^{k} e^{-4\Tr T^2}dT
= \E{\Tr^2 Z^{k}}<\infty.
\label{eqn:moments-tr2T2-finite-classical}
\end{align}}%
Thus indeed $\E{\Tr T^{2k}}$ and $\E{\Tr^2 T^{k}}$ have finite limits when $p$ is held fixed.

We can use this to show that $g_i^{(1)}$ and $g_i^{(2)}$ must vanish on $\N$ for $0\leq i_0<k+1$ and $0\leq i_0 < k+2$ as follows. Say the first statement wasn't true, and let $0\leq i_0 <k+1$ be the smallest $i$ such that for some $p_0\in\N$, $g_{i_0}^{(1)}(p_0)\neq0$. Then by Equation \eqref{eqn:moments-trT2k-as-g} and the definition of $i_0$, the limit of $\E{\Tr T^{2k}}$ as $n\rightarrow\infty$ with $p$ fixed at $p_0$ satisfies
{\setlength{\mathindent}{10pt}
\setlength{\abovedisplayskip}{3pt}
\setlength{\belowdisplayskip}{3pt}\begin{align*}
\lim_{n\rightarrow\infty} \frac{\E{\Tr T^{2k}}}{m^{k+1-i_0}}
= 1^{k} \cdot 1\cdot \lim_{n\rightarrow\infty}\sum_{i=i_0}^{D_1} g_i^{(1)}(p_0) m^{i_0-i}
= g^{(1)}_{i_0}(p_0).
\end{align*}}%
But $m=n-p-1$ tends to infinity as $n$ tends to infinity, and since $k+1-i_0>0$, Equation \eqref{eqn:moments-trT2-finite-classical} means that $\E{\Tr T^{2k}}/m^{k+1-i_0}$ must tend to zero. Thus $g^{(1)}_{i_0}(p_0)$ has to equal zero, which contradicts our assumption. Thus for every $0\leq i< k+1$, the polynomial $g^{(1)}_i$ must vanish on $\N$.

Similarly, for $0\leq i< k+2$, the polynomial $g^{(2)}_i$ must vanish on $\N$, because if it wasn't the case, we could take $0\leq i_0<k+2$ as the smallest $i$ with the property that for some $p_0\in\N$, $g^{(2)}_{i_0}(p_0)\neq 0$, and then by Equation \eqref{eqn:moments-tr2Tk-as-g} with $p$ fixed at $p_0$ as $n\rightarrow\infty$ we would get
{\setlength{\mathindent}{10pt}
\setlength{\abovedisplayskip}{3pt}
\setlength{\belowdisplayskip}{3pt}\begin{align*}
\lim_{n\rightarrow\infty} \frac{\E{\Tr^2 T^{k}}}{m^{k+2-i_0}}
= 1^{k} \cdot 1\cdot \lim_{n\rightarrow\infty}\sum_{i=i_0}^{D_2} g_i^{(2)}(p_0) m^{i_0-i}
= g^{(2)}_{i_0}(p_0).
\end{align*}}%
But then by Equation \eqref{eqn:moments-tr2T2-finite-classical}, as $m$ tends to infinity and $k+2-i_0\geq1$ the ratio $\E{\Tr^2 T^{k}}/m^{k+2-i_0}$ must tend to zero. Thus we must have $g^{(2)}_{i_0}(p_0)=0$, a contradiction. Hence indeed for $0\leq i< k+2$, the polynomial $g^{(2)}_i$ must vanish on $\N$.

But of course, a polynomial can only have an infinite number of zeroes if its coefficients are all zero, so we must have
\begin{ceqn}
{\setlength{\abovedisplayskip}{5pt}
\setlength{\belowdisplayskip}{5pt}\begin{align}
\begin{array}{ll}
a^{(1)}_{ij} = 0 \qquad&\text{for}\quad 0\leq i< k+1,
\\[4pt]
a^{(2)}_{ij} = 0 \qquad&\text{for}\quad 0\leq i< k+2.
\end{array}
\label{eqn:moments-aijs-are-zeros}
\end{align}}%
\end{ceqn}

Now say that $p$ varies with $n$ in such a way that $\lim_{n\rightarrow\infty}p/n=\alpha<1$. Then for large enough $n$, $n\geq p + 16k +6$ so by Equations \eqref{eqn:moments-trT2k-as-g} and \eqref{eqn:moments-aijs-are-zeros},
{\setlength{\mathindent}{0pt}
\setlength{\abovedisplayskip}{3pt}
\setlength{\belowdisplayskip}{3pt}\begin{align}&
\lim_{n\rightarrow\infty} \frac1{p^{k+1}}\E{\Tr T^{2k}} 
= \lim_{n\rightarrow\infty} \left(\frac{m}{n}\right)^k R_{2k}(m)
\sum_{i=k+1}^{D_1}\sum_{j=0}^{i} a^{(1)}_{ij}\frac{p^{j-(k+1)}}{m^{i-(k+1)}}
\notag\\[-3pt]&\hspace{10pt}
= (1-\alpha)^k\cdot 1\cdot \lim_{n\rightarrow\infty}\Bigg[
\sum_{i=k+1}^{D_1}\sum_{j=0}^{k} 
\frac{a^{(1)}_{ij}}{m^{i-(k+1)}p^{(k+1)-j}}
\notag\\[-5pt]&\pushright{
+\sum_{i=k+1}^{D_1}\sum_{j=k+1}^{i}
a^{(1)}_{ij}\left(\frac{p}{m}\right)^{j-(k+1)}\frac{1}{m^{i-j}}
\Bigg]
}\notag\\[-5pt]&\hspace{10pt}
= (1-\alpha)^k\Bigg[
\sum_{j=0}^{k} \frac{a_{k+1}^{(1)}}{\big(
\raisebox{2pt}{$\lim\limits_{p\rightarrow\infty}p$}
\big)^{(k+1)-j}}
+ \sum_{i=k+1}^{D_1} \hspace{-5pt}
a^{(1)}_{ii}\left(\frac{\alpha}{1-\alpha}\right)^{i-(k+1)}
\Bigg]
\label{eqn:moments-trT2k-limit}
\end{align}}%
and by Equations \eqref{eqn:moments-tr2Tk-as-g} and \eqref{eqn:moments-aijs-are-zeros},
{\setlength{\mathindent}{0pt}
\setlength{\abovedisplayskip}{3pt}
\setlength{\belowdisplayskip}{3pt}\begin{align}&
\lim_{n\rightarrow\infty} \frac1{p^{k+2}}\E{\Tr^2 T^{k}} 
= \lim_{n\rightarrow\infty} \left(\frac{m}{n}\right)^{k} R_{2k+1}(m)
\sum_{i=k+2}^{D_2}\sum_{j=0}^{i} a^{(2)}_{ij}\frac{p^{j-(k+2)}}{m^{i-(k+2)}}
\notag\\[-3pt]&\hspace{10pt}
= (1-\alpha)^{k}\cdot 1\cdot \lim_{n\rightarrow\infty}\Bigg[
\sum_{i=k+2}^{D_2}\sum_{j=0}^{k+1} 
\frac{a^{(2)}_{ij}}{m^{i-(k+2)}p^{(k+2)-j}}
\notag\\[-5pt]&\pushright{
+\sum_{i=k+2}^{D_1}\sum_{j=k+2}^{i}
a^{(2)}_{ij}\left(\frac{p}{m}\right)^{j-(k+2)}\frac{1}{m^{i-j}}
\Bigg]
}\notag\\[-5pt]&\hspace{10pt}
= (1-\alpha)^{k}\Bigg[
\sum_{j=0}^{k+1} \frac{a_{k+2}^{(2)}}{\big(
\raisebox{2pt}{$\lim\limits_{p\rightarrow\infty}p$}
\big)^{(k+2)-j}}
+ \sum_{i=k+2}^{D_2} \hspace{-5pt}
a^{(2)}_{ii}\left(\frac{\alpha}{1-\alpha}\right)^{i-(k+2)}
\Bigg].
\label{eqn:moments-tr2Tk-limit}
\end{align}}%
Although we might not know what $a_{ij}$ coefficients are, this shows at least that the limits are finite. In particular, from Equations \eqref{eqn:moments-trT2k-limit} and \eqref{eqn:moments-tr2Tk-limit} we can conclude that $\E{\Tr T^{2k}}= O(p^{k+1})$ and $\E{\Tr^2 T^{k}}= O(p^{k+2})$, which shows the first claim of the theorem.

For the second claim, let $n,p\rightarrow\infty$ with $p/n\rightarrow\alpha=0$. Then Equations \eqref{eqn:moments-trT2k-limit} and \eqref{eqn:moments-tr2Tk-limit} specialize to
\begin{ceqn}
{\setlength{\abovedisplayskip}{5pt}
\setlength{\belowdisplayskip}{5pt}\begin{align}&
\begin{array}{c}
\lim\limits_{n\rightarrow\infty} \dfrac1{p^{k+1}}\E{\Tr T^{2k}} = a^{(1)}_{(k+1)(k+1)},
\\[5pt]
\lim\limits_{n\rightarrow\infty} \dfrac1{p^{k+2}}\E{\Tr^2 T^{k}} = a^{(2)}_{(k+2)(k+2)}.
\end{array}
\label{eqn:moments-middle-scale-limit-trT2k-tr2Tk}
\end{align}}%
\end{ceqn}
What is interesting about this result is that these limits must be the same \textit{regardless of the way $p$ grows!} As long as $p\rightarrow\infty$ with $p/n\rightarrow0$, the limits are $a^{(1)}_{(k+1)(k+1)}$ and $a^{(2)}_{(k+2)(k+2)}$, regardless of whether $p\sim\log n$ or $p\sim \sqrt{n}$ or some other growth.

Now, \citet[Theorem 7]{bubeck16a} and \citet[Theorem 1]{jiang15} have shown that when $p\rightarrow\infty$ with $p^3/n\rightarrow0$, the total variation distance between a normalized Wishart $\sqrt{n}(\text{W}_p(n, I_p/n)-I_p)$ matrix and a Gaussian Orthogonal Ensemble $\text{GOE}(p)$ matrix tends to zero as $n\rightarrow\infty$. Therefore, the Hellinger distance satisfies also $H^2(\psi_{\text{NW}}, \psi_{\text{GOE}})=H^2(f_{\text{NW}}, f_{\text{GOE}})\rightarrow0$ as $n\rightarrow\infty$.

But convergence in Hellinger distance has strong implications for real-valued statistics. Indeed, for $T_1\sim\text{T}_{n/2}(I_p/8)=|\psi_\text{NW}|$, $T_2\sim\text{GOE}(p)/4=|\psi_\text{GOE}|$ and any function $g:\S_p(\R)\rightarrow\R$ such that $g(T_1)$, $g(T_2)$ are square-integrable,
{\setlength{\mathindent}{10pt}
\setlength{\abovedisplayskip}{0pt}
\setlength{\belowdisplayskip}{3pt}\begin{align}&
\Big|\E{g(T_1)} - \E{g(T_2)}\Big|
= \bigg|\bigintss_{\S_p(\R)} \hspace{-15pt}
g(T)\Big[|\psi_\text{NW}|(T)-|\psi_\text{GOE}|(T)\Big]dT\bigg|
\notag\\[-5pt]&\hspace{30pt}
\leq\;\; \bigg|\bigintss_{\S_p(\R)} \hspace{-15pt}
g(T)\overline{\psi^{1/2}_\text{NW}(T)}
\Big[\psi^{1/2}_\text{NW}(T)-\psi^{1/2}_\text{GOE}(T)\Big]
dT\bigg|
\notag\\[-5pt]&\hspace{70pt}
+ \bigg|\bigintss_{\S_p(\R)} \hspace{-15pt}
g(T)\psi^{1/2}_\text{GOE}(T)
\overline{\Big[\psi^{1/2}_\text{NW}(T)-\psi^{1/2}_\text{GOE}(T)\Big]}
dT\bigg|
\notag\\[-5pt]&\hspace{30pt}
\leq \;\; \bigg[
\bigintss_{\S_p(\R)} \hspace{-15pt}
g(T)^2|\psi_\text{NW}|(T)dT^{\frac12}
+
\bigintss_{\S_p(\R)} \hspace{-15pt}
g(T)^2|\psi_\text{GOE}|(T)dT^{\frac12}
\bigg]
\notag\\[-5pt]&\hspace{150pt}\cdot
\bigintss_{\S_p(\R)} \hspace{-15pt}
\Big|\psi^{1/2}_\text{NW}(T)-\psi^{1/2}_\text{GOE}(T)\Big|^2
dT^{\frac12}
\notag\\[-2pt]&\hspace{30pt}
= \bigg[
\E{g(T_1)^2}^{\frac12}+\E{g(T_2)^2}^{\frac12}
\bigg]
H(\psi_\text{NW}, \psi_\text{GOE})
\label{eqn:moments-statistics-follow-from-hellinger}
\end{align}}%
by the Cauchy-Schwarz inequality. 

Let's consider applying this result to $g(T)=\Tr T^{2k}/p^{k+1}$ and $g(T)=\Tr^2 T^{k}/p^{k+2}$. What do we know about these statistics? In the case where $T_2\sim\text{GOE}(p)/4$, results of \citet[Lemma 2.1.6 and the equation above Equation (2.1.21)]{anderson10} provide that
{\setlength{\mathindent}{30pt}
\setlength{\abovedisplayskip}{3pt}
\setlength{\belowdisplayskip}{3pt}\begin{align*}
\lim_{n\rightarrow\infty} \E{\frac{\Tr T_2^{2k}}{p^{k+1}}} 
\;\;&=\;\; \frac{C_{k}}{4^{2k}}
\\
\lim_{n\rightarrow\infty} \E{\frac{\Tr^2 T_2^{k}}{p^{k+2}}} 
\;\;&=\;\; \lim_{n\rightarrow\infty} \Big(\E{\frac{\Tr T_2^{k}}{p^{k/2+1}}}\Big)^2 
= \bigg(\frac{C_{k/2}}{4^{k}}\keven \bigg)^2
\\
\lim_{n\rightarrow\infty} \E{\Big(\frac{\Tr T_2^{2k}}{p^{k+1}}\Big)^2} 
\;\;&=\;\; \lim_{n\rightarrow\infty} \Big(\E{\frac{\Tr T_2^{2k}}{p^{k+1}}}\Big)^2 
= \bigg(\frac{C_{k}}{4^{2k}}\bigg)^2 < \infty
\\
\lim_{n\rightarrow\infty} \E{\Big(\frac{\Tr^2 T_2^{k}}{p^{k+2}}\Big)^2} 
\;\;&\leq\;\; \lim_{n\rightarrow\infty} \E{\frac{\Tr T_2^{4k}}{p^{2k+1}}} 
= \frac{C_{2k}}{4^{4k}} < \infty
\end{align*}}%
because these expressions only depend on $p$, and since $p\rightarrow\infty$ as $n\rightarrow\infty$, taking a limit as $n\rightarrow\infty$ is the same as taking a limit as $p\rightarrow\infty$. Moreover, in the $T_1\sim\text{T}_{n/2}(I_p/8)$ case, using Jensen's inequality and Equations \eqref{eqn:moments-trT2k-limit}--\eqref{eqn:moments-tr2Tk-limit} we can at least see that
{\setlength{\mathindent}{30pt}
\setlength{\abovedisplayskip}{3pt}
\setlength{\belowdisplayskip}{3pt}\begin{align*}
\lim_{n\rightarrow\infty} \E{\Big(\frac{\Tr T_1^{2k}}{p^{k+1}}\Big)^2} \;\;&\leq\;\;
\lim_{n\rightarrow\infty} \E{\frac{\Tr T_1^{4k}}{p^{2k+1}}} < \infty
\\
\lim_{n\rightarrow\infty} \E{\Big(\frac{\Tr^2 T_1^{k}}{p^{k+2}}\Big)^2} \;\;&\leq\;\;
\lim_{n\rightarrow\infty} \E{\frac{\Tr^2 T_1^{2k}}{p^{k+2}}} < \infty.
\end{align*}}%
Therefore, using Equation \eqref{eqn:moments-statistics-follow-from-hellinger} with  $g(T)=\Tr T^{2k}/p^{k+1}$ and $g(T)=\Tr^2 T^{k}/p^{k+2}$ we find that when $n,p\rightarrow\infty$ with $p^3/n\rightarrow0$,
{\setlength{\mathindent}{10pt}
\setlength{\abovedisplayskip}{0pt}
\setlength{\belowdisplayskip}{3pt}\begin{align*}&\hspace{20pt}
\bigg|\lim_{n\rightarrow\infty}\E{\frac{\Tr T_1^{2k}}{p^{k+1}}} 
- \frac{C_k}{4^{2k}}\bigg|
\;\leq\; \bigg(\lim_{n\rightarrow\infty}\E{\Big(\frac{\Tr T_1^{2k}}{p^{k+1}}\Big)^2}^{\frac12} 
+ \frac{C_{2k}}{4^{2k}}\bigg)\cdot 0 
= 0,
\\&\hspace{20pt}
\bigg|\lim_{n\rightarrow\infty}\E{\frac{\Tr^2 T_1^{k}}{p^{k+2}}} 
- \Big(\frac{C_{k/2}}{4^{k}}\keven \Big)^2\bigg|
\\&\hspace{127pt}
\;\leq\; \bigg(\lim_{n\rightarrow\infty}\E{\Big(\frac{\Tr^2 T_1^{k}}{p^{k+2}}\Big)^2}^{\frac12} 
+ \frac{C^{1/2}_{2k}}{4^{2k}}\bigg)\cdot 0 
= 0.
\end{align*}}%
Since $p^3/n\rightarrow0$ implies $p/n\rightarrow0$, we conclude from Equation \eqref{eqn:moments-middle-scale-limit-trT2k-tr2Tk} that
\begin{ceqn}
{\setlength{\abovedisplayskip}{0pt}
\setlength{\belowdisplayskip}{5pt}\begin{align*}&
a^{(1)}_{(k+1)(k+1)} = \frac{C_k}{4^{2k}},\hspace{40pt}
a^{(2)}_{(k+2)(k+2)} = \frac{C^2_{k/2}}{4^{2k}}\keven^2 .
\end{align*}}%
\end{ceqn}
But then, by that same equation, we conclude that when $n,p\rightarrow\infty$, not only when $p^3/n\rightarrow0$ but for all $p$ such that $p/n\rightarrow0$, we have
\begin{ceqn}
{\setlength{\abovedisplayskip}{5pt}
\setlength{\belowdisplayskip}{5pt}
\begin{align}&
\begin{array}{l}
\lim\limits_{n\rightarrow\infty} \dfrac1{p^{k+1}}\E{\Tr T^{2k}} = \dfrac{C_k}{4^{2k}},
\\[5pt]
\lim\limits_{n\rightarrow\infty} \dfrac1{p^{k+2}}\E{\Tr^2 T^{k}} = \dfrac{C^2_{k/2}}{4^{2k}}\keven^2
\end{array}
\label{eqn:moments-moment-limits}
\end{align}}%
\end{ceqn}
for $T\sim\text{T}_{n/2}(I_p/8)$. To finish the proof, use Equation \eqref{eqn:moments-moment-limits} with the fact that $\E{\Tr T^k}=0$ for odd $k$ to find that
{\setlength{\mathindent}{10pt}
\setlength{\abovedisplayskip}{3pt}
\setlength{\belowdisplayskip}{3pt}\begin{align*}\hspace{20pt}
\lim_{n\rightarrow\infty}\Var{\frac{\Tr T^{k}}{p^{k/2+1}}} 
&= \lim_{n\rightarrow\infty}\E{\frac{\Tr^2 T^{k}}{p^{k+2}}} 
-\E{\frac{\Tr T^{k}}{p^{k/2+1}}} ^2
\\&
= \frac{C^2_{k/2}}{4^{2k}}\keven^2 - \Big(\frac{C_{k/2}}{4^{k}}\keven \Big)^2 = 0.
\end{align*}}%
Thus $\Tr T^{k}/p^{k+1}\overset{L^2}{\longrightarrow} C_k/4^{2k}$, as desired. This proves the second claim and concludes the proof.
\end{proof}

A pleasant consequence of this result is that when $n,p\rightarrow\infty$ with $p/n\rightarrow0$, we can conclude a version of the semicircle law holds for the $\text{T}_{n/2}(2I_p)$ distribution. This is interesting because the $\text{T}_{n/2}(2I_p)$ distribution has dependent entries with heavy tails, whose distribution varies with $n, p$.

Let $4T/\sqrt{p}\sim 4\text{T}_{n/2}(I_p/8)/\sqrt{p}$, with eigenvalues $\lambda_1(4T/\sqrt{p})\geq\dots\geq\lambda_p(4T/\sqrt{p})$. Then define its empirical spectral measure to be
\begin{ceqn}
{\setlength{\abovedisplayskip}{0pt}
\setlength{\belowdisplayskip}{3pt}\begin{align*}
L_{4T/\sqrt{p}}(A) = \frac1{p}\sum_{i=1}^p\1{\lambda_i(4T/\sqrt{p}) \in A}.
\end{align*}}%
\end{ceqn}
Since $L_{4T/\sqrt{p}}$ depends on the random matrix $T$, it is a random measure on $\R$.

\begin{corollary}[Semicircle law for the $t$ distribution]\label{cor:semicircle} The empirical spectral measure $L_{4T/\sqrt{p}}$ of a $4\text{T}_{n/2}(I_p/8)/\sqrt{p}$ random matrix converges weakly, in square mean, to the semicircle distribution
\begin{ceqn}
{\setlength{\abovedisplayskip}{5pt}
\setlength{\belowdisplayskip}{5pt}\begin{align*}
L(A) = \int_A\frac{\sqrt{4-x^2}}{2\pi}\1{|x|\leq 2}dx.
\end{align*}}%
\end{ceqn}
\end{corollary}
\begin{proof}
Let $f$ be any continuous function $\R\rightarrow\R$ that vanishes at infinity. By the Stone-Weierstrass theorem, there exists a sequence $f_1, f_2, \dots$ of polynomials such that for any $\epsilon>0$, $\sup_{x\in\R}|f(x)-f_l(x)|<\epsilon$. To fix some notation, write
{\setlength{\abovedisplayskip}{5pt}
\setlength{\belowdisplayskip}{5pt}\begin{align*}
f_l(x) = \sum_{k=1}^{\text{deg}f_l}a_{lk}x^{k}.
\end{align*}}%
Then since $L_{4T/\sqrt{p}}$ and $L$ are both probability measures,
{\setlength{\mathindent}{10pt}
\setlength{\abovedisplayskip}{5pt}
\setlength{\belowdisplayskip}{5pt}\begin{align*}&
\E{\bigg(
\int_\R f(x)dL_{4T/\sqrt{p}}(x)
- \int_\R f(x)dL(x)
\bigg)^2}^{\frac12}
\\&\hspace{30pt}
\leq\;\; \E{\bigg(
\int_\R \big[f(x)-f_l(x)\big]dL_{4T/\sqrt{p}}(x)
\bigg)^2}^{\frac12}
\\&\hspace{50pt}
+ \E{\bigg(
\int_\R f_l(x)dL_{4T/\sqrt{p}}(x)
- \int_\R f_l(x)dL(x)
\bigg)^2}^{\frac12}
\\&\hspace{50pt}
+ \E{\bigg(
\int_\R \big[f_l(x)- f(x)\big]dL(x)
\bigg)^2}^{\frac12}
\\&\hspace{30pt}
\leq\;\; \epsilon + \sum_{k=1}^{\text{deg}f_l}|a_{lk}|\E{\bigg(
\frac1{p}\Tr \Big(\frac{4T}{\sqrt{p}}\Big)^k
- C_{k/2}\keven 
\bigg)^2}^{\frac12}
+ \epsilon.
\end{align*}}%
By Theorem \ref{thm:moments}, the expectation tends to zero as $n,p\rightarrow\infty$ with $p/n\rightarrow0$. Thus
{\setlength{\mathindent}{30pt}
\setlength{\abovedisplayskip}{0pt}
\setlength{\belowdisplayskip}{5pt}\begin{align*}&
\lim_{n\rightarrow\infty}\E{\bigg(
\int_\R f(x)dL_{4T/\sqrt{p}}(x)
- \int_\R f(x)dL(x)
\bigg)^2}^{\frac12}
\leq\;\; 2\epsilon.
\end{align*}}%
But this is true for every $\epsilon>0$, so the limit must be zero. Hence for every continuous $f$ that vanishes at infinity, the integral $\int fdL_{4T/\sqrt{p}}$ converges in square mean to $\int fdL$. By \citet[Theorem 4.4.1 and 4.4.2]{chung01}, this implies that for every \textit{bounded} continuous $f$, the integral $\int fdL_{4T/\sqrt{p}}$ converges in square mean to $\int fdL$. Thus the empirical spectral distribution $L_{4T/\sqrt{p}}$ converges weakly, in square mean, to the semicircle distribution $L$, as desired.
\end{proof}

\section{Wishart asymptotics: the G-transform point-of-view}\label{sec:wishart-G-transform}

We now turn our attention to the main objective of this paper, namely studying the behavior of Wishart matrices in the various middle-scale regimes. To do this, we exploit the close connection between the Wishart and the symmetric $t$ distributions and make use of the results found Section \ref{sec:symmetrict}. The main result of this section, Theorem \ref{thm:existence-gtransform}, states that we can approximate for every middle-scale regime the G-transform $\psi_\text{NW}$ of a normalized Wishart by a degree-specific function $\psi_K$. This can be seen as an analogue of Theorem \ref{thm:existence-densities} in the G-transform domain.

The reasoning behind the approximations is as follows. We could imagine writing $\psi_\text{NW}$ from Proposition \ref{prop:gtransform-nw} in exponential form, and expanding the terms as a Taylor series would yield
{\setlength{\mathindent}{10pt}
\setlength{\abovedisplayskip}{5pt}
\setlength{\belowdisplayskip}{5pt}\begin{align*}
\psi_{\text{NW}}(T)
&= C_{n, p}
\exp\bigg\{2i\sqrt{n}\Tr  T
-\frac{n+p+1}2\log\bigg|I_p +i\frac{4 T}{\sqrt{n}}\bigg|
\bigg\}
\\&
= C_{n, p}
\exp\bigg\{
2i\sqrt{n}\Tr T
+\frac{n+p+1}2
\sum_{k=1}^\infty \frac{(-i)^k}{k}
\left(\frac{p}{n}\right)^{\frac{k}2}
\Tr \bigg(\frac{4 T}{\sqrt{p}}\bigg)^k
\bigg\}.
\end{align*}}%
Now imagine that the $T$'s appearing in the expression follow a $\text{T}_{n/2}(I_p/8)$ distribution. By Theorem \ref{thm:moments}, we know that $\Tr \big(\frac{4 T}{\sqrt{p}}\big)^k=\Theta(p)$ when $k$ is even, in an $L^2$ sense. When $k$ is odd, the theorem merely proves that $\Tr \big(\frac{4 T}{\sqrt{p}}\big)^k=o(p)$, but for a $\text{GOE}(p)$ matrix $Z$, we know that $\Tr \frac{Z}{\sqrt{p}}^k$ is asymptotically normal for odd $k$ by \citet[Theorem 2.1.31]{anderson10}. This would suggest that $\Tr \big(\frac{4 T}{\sqrt{p}}\big)^k=\Theta(1)$ when $k$ is odd. Thus we would have, in some sense,
{\setlength{\mathindent}{20pt}
\setlength{\abovedisplayskip}{5pt}
\setlength{\belowdisplayskip}{5pt}\begin{align*}
\psi_{\text{NW}}(T)
&= C_{n, p}
\exp\bigg\{
\sum_{\substack{k=2\\\text{even}}}^\infty 
\Theta\left(\frac{p^{k/2+1}}{n^{k/2-1}}\right)
+
\sum_{\substack{k=2\\\text{odd}}}^\infty 
\Theta\left(\frac{p^{k/2}}{n^{k/2-1}}\right)
\bigg\}
\\&
= C_{n, p}
\exp\bigg\{
\sum_{k=0}^\infty 
\Theta\left(\frac{p^{k+2}}{n^{k}}\right)
+ \sum_{k=1}^\infty
\Theta\left(\sqrt{\frac{p^{(2k-1)+2}}{n^{2k-1}}}\right)
\bigg\}.
\end{align*}}%
In other words, terms in the power series would be associated with some degree $K$, such that they would be non-negligible in any middle-scale regime of degree up to $K$, and negligible in higher degrees. In fact, a similar phenomenon occurs with $C_{n,p}$, by Lemma \ref{lem:cnp}. This suggests we should try truncating these power series to derive degree-specific approximations.

\begin{definition}[G-transform approximations]\label{def:G-transform-approximations} For any $K\in\N$, define the $K^\text{th}$ degree approximation $\psi_K: \S_p(\R)\rightarrow\C$ as
{\setlength{\mathindent}{10pt}
\setlength{\abovedisplayskip}{3pt}
\setlength{\belowdisplayskip}{0pt}\begin{align*}&\;\;
\psi_K(T) = C_{n,p}^{(K)}\exp\left\{\rule{0pt}{25pt}\right.\!
	\text{\raisebox{-5pt}{$
	\frac{n}2
	\hspace{-6pt}\mathlarger{\mathlarger{\sum}}_{k=2}^{\substack{2K+3+\\ \Kodd}}\hspace{-10pt}
	{\Big(\frac{4i}{\sqrt{n}}\Big)\!\!}^k
	\frac{\Tr T^k}{k}
	+
	\frac{p+\!1}2\hspace{-7pt}\mathlarger{\mathlarger{\sum}}_{k=1}^{\substack{2K+2-\\ \Kodd}}\hspace{-10pt}
	{\Big(\frac{4i}{\sqrt{n}}\Big)\!\!}^k
	\frac{\Tr T^k}{k}
	$}}
\!\!\left.\rule{0pt}{25pt}\right\}
\notag\\[-5pt]&\hspace{-10pt}\text{with}
\\[-7pt]&\;\;
C_{n,p}^{(K)} = \frac{2^{\frac{p(3p+1)}4}}{\pi^{\frac{p(p+1)}4}}
	\exp\bigg\{\!\!
	-\!\frac12\!\sum_{k=1}^{K+1}\!\frac{\keven}{k(k\!+\!1)(k\!+\!2)}\frac{p^{k+2}}{n^k}
	\notag\\&\hspace{180pt}
	-\!\frac14\!\sum_{k=1}^{K+1}\!\frac{1\!+\!2\keven}{k(k\!+\!1)}\frac{p^{k+1}}{n^k}
	\bigg\}.
	\notag
\end{align*}}%
\end{definition}

Just like the G-transform of a normalized Wishart matrix, these functions implicitly depend on $n$. The first three are
{\setlength{\mathindent}{10pt}
\setlength{\abovedisplayskip}{3pt}
\setlength{\belowdisplayskip}{3pt}\begin{align*}&
\psi_0(T) = \frac{2^{\frac{p(3p+1)}4}}{\pi^{\frac{p(p+1)}4}}
	\exp\bigg\{\!\!
	-\!4\Tr T^2 
	-\!i\frac{32}{3\sqrt{n}}\!\Tr T^3 
	+\!2i\frac{p+1}{\sqrt{n}}\!\Tr T 
	-\!4\frac{p+1}{n}\!\Tr T^2
	\!\bigg\},
\\&
\psi_1(T) = \frac{2^{\frac{p(3p+1)}4}}{\pi^{\frac{p(p+1)}4}}
	\exp\bigg\{\!\!
	-\!\frac{p^2}{8n}
	-\!4\Tr T^2 
	-\!i\frac{32}{3\sqrt{n}}\!\Tr T^3 
	+\!\frac{32}{n}\!\Tr T^4
	+\!i\frac{512}{5n^{3/2}}\!\Tr T^5
\\&\hspace{65pt}
-\!\frac{1024}{3n^2}\!\Tr T^6
	+\!2i\frac{p+1}{\sqrt{n}}\!\Tr T 
	-\!4\frac{p+1}{n}\!\Tr T^2
	-\!i\frac{32(p+1)}{3n^{3/2}}\!\Tr T^3
\!\bigg\}
\\[-3pt]&\hspace{-10pt}\text{and}
\\[-6pt]&
\psi_2(T) = \frac{2^{\frac{p(3p+1)}4}}{\pi^{\frac{p(p+1)}4}}
	\exp\bigg\{
	-\!\frac{p^4}{48n^2}
	-\!\frac{p^2}{8n}
	-\!4\Tr T^2 
	-\!i\frac{32}{3\sqrt{n}}\!\Tr T^3 
	+\!\frac{32}{n}\!\Tr T^4
\\&\hspace{90pt}
+\!i\frac{512}{5n^{3/2}}\!\Tr T^5
	-\!\frac{1024}{3n^2}\!\Tr T^6
	-\!i\frac{8192}{7n^{5/2}}\!\Tr T^7
	+\!2i\frac{p+1}{\sqrt{n}}\!\Tr T 
\\&\hspace{120pt}
-\!4\frac{p+1}{n}\!\Tr T^2
	-\!i\frac{32(p+1)}{3n^{3/2}}\!\Tr T^3
	+\!32\frac{p+1}{n^2}\!\Tr T^4
\\&\hspace{155pt}
+\!i\frac{512(p+1)}{5n^{5/2}}\!\Tr T^5
	-\!\frac{1024(p+1)}{3n^3}\!\Tr T^6
	\!\bigg\}.
\end{align*}}%
These functions have the pleasant property that their modulus is bounded, up to a constant, by the G-conjugate density $|\psi_K|$. Indeed, on one hand we can rewrite Definition \ref{def:G-transform-approximations} into
{\setlength{\mathindent}{10pt}
\setlength{\abovedisplayskip}{3pt}	
\setlength{\belowdisplayskip}{3pt}\begin{align}&
\psi_K(T)
	= \exp\left\{\rule{0pt}{23pt}\right.\!
	\log C_{n, p}^{(K)}
	+ \frac{n}2 \hspace{-5pt}\mathlarger{\mathlarger{\sum}}_{k=1}^{\substack{K + 1 + \\ \Kodd}}\hspace{-7pt}
		(-1)^k\frac{\Tr\left(4T/\sqrt{n}\right)^{2k}}{2k}
\notag\\[-4pt]&\hspace{10pt}
	+ \frac{p+1}2 \hspace{-5pt}\mathlarger{\mathlarger{\sum}}_{k=1}^{\substack{K + 1 - \\ \Kodd}}\hspace{-7pt}
		(-1)^k\frac{\Tr\left(4T/\sqrt{n}\right)^{2k}}{2k}
	- i\frac{n}2 \mathlarger{\mathlarger{\sum}}_{k=1}^{K + 1}
		(-1)^k\frac{\Tr\left(4T/\sqrt{n}\right)^{2k+1}}{2k+1}
\notag\\[-5pt]&\hspace{140pt}
	- i\frac{p+1}2 \mathlarger{\mathlarger{\sum}}_{k=1}^{K}
		(-1)^k\frac{\Tr\left(4T/\sqrt{n}\right)^{2k+1}}{2k+1}
\left.\rule{0pt}{23pt}\right\}\!.
\label{eq:thm-existence-psik}
\end{align}}%
On the other hand, for any $x\in\R$, we can write $1+ix = \sqrt{1+x^2} \exp\big(i \; \text{atan}(x)\big)$ with the arctangent function taking values in $(-\pi/2, \pi/2)$. Thus by Proposition \ref{prop:gtransform-nw} we can rewrite $\psi_{\text{NW}}$ as
{\setlength{\mathindent}{30pt}
\setlength{\abovedisplayskip}{3pt}	
\setlength{\belowdisplayskip}{3pt}\begin{align}&
\psi_{\text{NW}}(T)
	= \exp\bigg\{
	\log C_{n, p}
	- \frac{n+p+1}4\log\left|I_p+\frac{16T^2}{n}\right|
\notag\\&\hspace{90pt}
	- i\frac{n+p+1}2\Tr\text{atan}\left(\frac{4T}{\sqrt{n}}\right)
	+ 2i\sqrt{n}\Tr T
	\bigg\},
\label{eqn:psiNW-real-and-imaginary}
\end{align}}%
with the understanding that the matrix-variate arctangent function operates on eigenvalues by functional calculus. Now, for any $x\in\R$ and odd integer $L$, there is an elementary inequality
{\setlength{\mathindent}{5pt}
\setlength{\abovedisplayskip}{3pt}	
\setlength{\belowdisplayskip}{3pt}\begin{align*}&\hspace{30pt}
\sum_{l=1}^L (-1)^l \frac{x^{2l}}{2l} \;\leq\; -\frac12\log(1+x^2).
\end{align*}}%
Notice that $K+1\pm\Kodd$ is always an odd integer. Thus, from the above inequality and Equations \eqref{eq:thm-existence-psik} and \eqref{eqn:psiNW-real-and-imaginary}, we can derive the bound
{\setlength{\mathindent}{30pt}
	\setlength{\abovedisplayskip}{3pt}	
	\setlength{\belowdisplayskip}{3pt}\begin{align}
	\big|\psi_K\big|(T) 
	&\;\;\leq\;\; C^{(K)}_{n,p} \,\exp\Big\{
	-\Big(\frac{n}4+\frac{p+1}4\Big)\log\Big|I_p+\frac{16T^2}{n}\Big|
	\Big\}
	\notag\\&
	\;\;=\;\; \frac{C^{(K)}_{n,p}}{C_{n,p}}\,
	\big|\psi_{\text{NW}}\big|(T).
	\label{eqn:psiK-bounded-by-psi-NW}
	\end{align}}%
In particular, since $\psi_\text{NW}$ is integrable whenever $n\geq p-2$ by Proposition \ref{prop:gtransform-nw}, Equation \eqref{eqn:psiK-bounded-by-psi-NW} implies that every $\psi_K$ must also be integrable whenever $n\geq p-2$. In particular, for large enough $n$ it makes sense to talk about the asymptotic total variation or Hellinger distance between $\psi_\text{NW}$ and $\psi_K$. 

We now state the main result, which is that each function $\psi_K$ approximates the G-transform of a normalized Wishart for all middle-scale regimes of degree $K$ or lower, but no other.

\begin{theorem}\label{thm:existence-gtransform}\hspace{-10pt}
Let $\lim\limits_{n\rightarrow\infty}\frac{\log p}{\log n}<1$ as $n\rightarrow\infty$. For any $K\in\N$, the total variation distance between the G-transform of the normalized Wishart distribution $\sqrt{n}[\text{W}_p(n,I_p/n)-I_p]$ and the $K^\text{th}$ approximating function $\psi_K$ satisfies
{\setlength{\mathindent}{40pt}
\setlength{\abovedisplayskip}{3pt}
\setlength{\belowdisplayskip}{3pt}\begin{align*}
\mathrm{d}_\text{TV}\big(\psi_{\text{NW}}, \psi_K\big) =
	\bigintsss_{\S_p(\R)}\hspace{-15pt} 
	\big|\psi_{\text{NW}}(T) - \psi_K(T)\big| \,dT
	\rightarrow 0
	&&\text{ as }n \rightarrow\infty
\end{align*}}%
if and only if $p^{K+3}/n^{K+1}\rightarrow0$.
\end{theorem}
\begin{proof}
\textit{If statement.} For the first part of the theorem, remark that by Equation \eqref{eqn:HvsTV-Gtransforms} it is equivalent to show that the Hellinger distance tends to zero, i.e. that
{\setlength{\mathindent}{40pt}
\setlength{\abovedisplayskip}{3pt}
\setlength{\belowdisplayskip}{3pt}\begin{align*}
\mathrm{H}^2\big(\psi_{\text{NW}}, \psi_K\big) =
	\bigintss_{\S_p(\R)}\hspace{-15pt} 
	\Big|\psi_{\text{NW}}^{1/2}(T) - \psi_K^{1/2}(T)\Big|^2 \,dT
	\rightarrow 0
	&&\text{ as }n \rightarrow\infty
\end{align*}}%
when $p^{K+3}/n^{K+1}\rightarrow0$. To control this quantity, we use the Kullback-Leibler inequality for G-transforms. Notice that for any $x\in\R$ and $L\in\N$,
{\setlength{\mathindent}{10pt}
\begin{align}\qquad
\left|-\frac12\log(1+x^2) - \sum_{l=1}^{L-1} (-1)^l \frac{x^{2l}}{2l}\right| 
&\;\leq\; \frac{x^{2L}}{2L},
\label{eq:thm-existence-logineq}
\\\qquad
\left|\text{atan}(x) - \sum_{l=1}^{L-1} (-1)^l \frac{x^{2l+1}}{2l+1}\right| 
&\;\leq\; \frac{x^{2L}}{2L}.
\label{eq:thm-existence-atanineq}
\end{align}}%
Let $\Log$ stand for the principal branch of the complex logarithm, and let us study $\Log\psi_{\text{NW}}/\psi_K$.
Its real part can be bounded by
{\setlength{\mathindent}{5pt}
\setlength{\abovedisplayskip}{3pt}	
\setlength{\belowdisplayskip}{3pt}\begin{align*}&
\bigg|\Re\Log\frac{\psi_{\text{NW}}(T)}{\psi_K(T)} \bigg|
	= \left|\rule{0pt}{23pt}\right.
	\log C_{n, p}
	- \frac{n+p+1}4\log\left|I_p+\frac{16T^2}{n}\right|
	- \log C_{n, p}^{(K)}
\\&\hspace{54pt}
	- \frac{n}2 \hspace{-5pt}\mathlarger{\mathlarger{\sum}}_{k=1}^{\substack{K + 1 + \\ \Kodd}}\hspace{-7pt}
	(-1)^k\frac{\Tr\left(4T/\sqrt{n}\right)^{2k}}{2k}
	- \frac{p+1}2 \hspace{-5pt}\mathlarger{\mathlarger{\sum}}_{k=1}^{\substack{K + 1 - \\ \Kodd}}\hspace{-7pt}
	(-1)^k\frac{\Tr\left(4T/\sqrt{n}\right)^{2k}}{2k}
	\left.\rule{0pt}{23pt}\right|
\\&
\leq \bigg|\log C_{n, p}\!-\!\log C_{n, p}^{(K)}\bigg|
	+\frac{n}2\left|\rule{0pt}{25pt}\right.\!\!
		-\!\frac12\!\log\left|I_p+\frac{16T^2}{n}\right| -
		\hspace{-7pt}\mathlarger{\mathlarger{\sum}}_{k=1}^{\substack{K + 1 + \\ \Kodd}}\hspace{-7pt}
		(-1)^k\frac{\Tr\left(4T/\sqrt{n}\right)^{2k}}{2k}
		\!\!\left.\rule{0pt}{25pt}\right|
\\&\hspace{90pt}
	+ \frac{p+1}2\left|\rule{0pt}{25pt}\right.\!\!
		-\!\frac12\!\log\left|I_p+\frac{16T^2}{n}\right| -
		\hspace{-7pt}\mathlarger{\mathlarger{\sum}}_{k=1}^{\substack{K + 1 - \\ \Kodd}}\hspace{-7pt}
		(-1)^k\frac{\Tr\left(4T/\sqrt{n}\right)^{2k}}{2k}
		\!\!\left.\rule{0pt}{25pt}\right|\!.
\end{align*}}%
By Equation \eqref{eq:thm-existence-logineq}, this can be bounded by
{\setlength{\mathindent}{5pt}
\setlength{\abovedisplayskip}{3pt}	
\setlength{\belowdisplayskip}{3pt}\begin{align}&
\leq \bigg|\log C_{n, p}\!-\!\log C_{n, p}^{(K)}\bigg|
	+ \frac{n}2\frac{\Tr (4T/\sqrt{n})^{2K+4+2\Kodd}}{2K+4+2\Kodd}
\notag\\&\hspace{140pt}
	+ \frac{p+1}2\frac{\Tr (4T/\sqrt{n})^{2K+4-2\Kodd}}{2K+4-2\Kodd}
\notag\\&
= \bigg|\log C_{n, p}\!-\!\log C_{n, p}^{(K)}\bigg| 
	+ \frac{4^{2K+3+2\Kodd}}{K+2+\Kodd} \frac{\Tr T^{2K+4+2\Kodd}}{n^{K+1+\Kodd}}
\notag\\&\hspace{80pt}
	+ \frac{4^{2K+3-2\Kodd}}{K+2-\Kodd} \frac{(p+1)\Tr T^{2K+4-2\Kodd}}{n^{K+2-\Kodd}}.
\label{eq:thm-existence-boundre}
\end{align}}%
We can bound the imaginary part of $\Log\psi_{\text{NW}}/\psi_K$ in a similar way. Define $P_{(-\pi, \pi]}:\R\rightarrow(-\pi, \pi]$ to be the projection mapping $P_{(-\pi, \pi]}x = x-2\pi\lceil\frac{x}{2\pi}-\frac12\rceil$. A plot is given as Figure \ref{fig:modproj}.

\begin{figure}[t]
\centering
\begin{tikzpicture}[scale=0.25]
\def\numpi{3.14};
\draw [help lines,->] (-6*\numpi, 0) -- (6*\numpi, 0);
\node[below] at (6*\numpi, 0){$x$};
\foreach \offset/\ticktext in {{-5*\numpi}/{-$5\pi$}, {-3*\numpi}/{-$3\pi$}, 
												{-\numpi}/{-$\pi$}, {\numpi}/{$\pi$},
												{3*\numpi}/{$3\pi$}, {5*\numpi}/{$5\pi$} 
											   }
	\draw[shift={(\offset, 0)}] (0pt, 4pt) -- (0pt, -4pt) node[below] {\ticktext};
\draw [help lines,->] (-0, -4) -- (-0, 4);
\node[left] at (0, 4) {$y$};
\draw[shift={(0,\numpi)}] (-4pt, 0pt) -- (4pt, 0pt) node[below left] {$\pi$};
\draw[shift={(0,-\numpi)}] (-4pt, 0pt) -- (4pt, 0pt) node[below left] {-$\pi$};

\foreach \offset in {-2, ..., 2}
{
	\draw[line width=0.5pt] (-\numpi+\offset*2*\numpi, -\numpi)--(\numpi+\offset*2*\numpi, \numpi);
	\draw [fill=white] (-\numpi+\offset*2*\numpi, -\numpi) circle[radius= 6pt];
}
\end{tikzpicture}
\caption{Plot of $P_{(-\pi, \pi)}$ on $(-5\pi, 5\pi]$.}
\label{fig:modproj}
\end{figure}
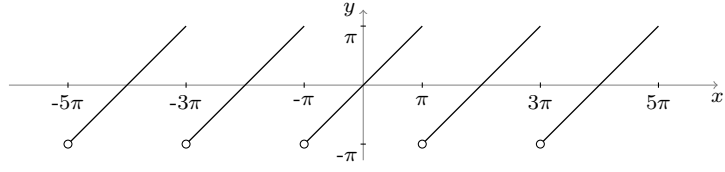

It satisfies $\Im\Log z = P_{(-\pi, \pi]}\Im\log z$ for all branches of $\log z$, as well as the inequality $|P_{(-\pi, \pi]}x|\leq |x|$. Using this mapping, we can see that the imaginary part of $\Log\psi_{\text{NW}}/\psi_K$ can be bounded as
{\setlength{\mathindent}{5pt}
\setlength{\abovedisplayskip}{3pt}	
\setlength{\belowdisplayskip}{3pt}\begin{align*}&
\bigg|\Im\Log\frac{\psi_{\text{NW}}(T)}{\psi_K(T)} \bigg|
	= \left|\rule{0pt}{23pt}\right.\! P_{(-\pi, \pi]}  \!\!\left[\rule{0pt}{23pt}\right.
	- \frac{n+p+1}2\Tr\text{atan}\left(\frac{4T}{\sqrt{n}}\right)
	+ 2\sqrt{n}\Tr T
\\&\hspace{35pt}
	+ \frac{n}2 \mathlarger{\mathlarger{\sum}}_{k=1}^{K + 1}
		(-1)^k\frac{\Tr\left(4T/\sqrt{n}\right)^{2k+1}}{2k+1}
	+ \frac{p+1}2 \mathlarger{\mathlarger{\sum}}_{k=1}^{K}
		(-1)^k\frac{\Tr\left(4T/\sqrt{n}\right)^{2k+1}}{2k+1}
	\left.\rule{0pt}{23pt}\right] \!\!\left.\rule{0pt}{23pt}\right|
\\&\hspace{15pt}
\leq \frac{n}2 \left|\rule{0pt}{23pt}\right. 
	\Tr\text{atan}\left(\frac{4T}{\sqrt{n}}\right)
	- \frac{2\sqrt{n}\Tr T}{n/2}
	- \mathlarger{\mathlarger{\sum}}_{k=1}^{K + 1}
		(-1)^k\frac{\Tr\left(4T/\sqrt{n}\right)^{2k+1}}{2k+1}
	\left.\rule{0pt}{23pt}\right|
\\&\hspace{95pt}
	+ \frac{p+1}2 \left|\rule{0pt}{23pt}\right. 
		\Tr\text{atan}\left(\frac{4T}{\sqrt{n}}\right)
	- \mathlarger{\mathlarger{\sum}}_{k=1}^{K}
		(-1)^k\frac{\Tr\left(4T/\sqrt{n}\right)^{2k+1}}{2k+1}
	\left.\rule{0pt}{23pt}\right|.
\end{align*}}%
By Equation \eqref{eq:thm-existence-atanineq}, this can be bounded by
{\setlength{\mathindent}{5pt}
\setlength{\abovedisplayskip}{3pt}	
\setlength{\belowdisplayskip}{5pt}\begin{align}&\hspace{15pt}
\leq\; \frac{n}2 \frac{\Tr (4T/\sqrt{n})^{2K+4}}{2K+4}
	+ \frac{p+1}2 \frac{\Tr (4T/\sqrt{n})^{2K+2}}{2K+2}
\notag\\&\hspace{15pt}
\leq\; \frac{4^{2K+3}}{K+2}\frac{\Tr T^{2K+4}}{n^{K+1}}
	+ \frac{4^{2K+1}}{K+1}\frac{(p+1)\Tr T^{2K+4}}{n^{K+1}}.
\label{eq:thm-existence-boundim}
\end{align}}%
Recall that the G-conjugate of the normalized Wishart distribution is the $t$ distribution with $n/2$ degrees of freedom and scale matrix $I_p/8$, denoted $T_{n/2}(I_p/8)$ -- see Equation \eqref{eqn:gconjugate-nw} and Section \ref{sec:symmetrict} for details. Let us bound the expectations of these absolute real and imaginary parts under this distribution. By Equations \eqref{eq:thm-existence-cnp}, \eqref{eq:thm-existence-boundre}, \eqref{eq:thm-existence-boundim} and Theorem \ref{thm:moments}, we find that for $T\sim|\psi_{\text{NW}}| = T_{n/2}(I_p/8)$,
{\setlength{\mathindent}{5pt}
\setlength{\abovedisplayskip}{3pt}	
\setlength{\belowdisplayskip}{3pt}\begin{align}&
\E{\bigg|\Re\Log\frac{\psi_{\text{NW}}(T)}{\psi_K(T)} \bigg|}
	\leq\;\; \frac{4^{2K+3+2\Kodd}}{K+2+\Kodd} \frac{\E{\Tr T^{2(K+2+\Kodd)}}}{n^{K+1+\Kodd}}
\notag\\&\hspace{40pt}
	+ \frac{4^{2K+3-2\Kodd}}{K+2-\Kodd} \frac{(p+1)\E{\Tr T^{2(K+2-\Kodd)}}}{n^{K+2-\Kodd}}
	+ o\Big(\frac{p^{K+3}}{n^{K+1}}\Big)
\notag\\&\hspace{20pt}
\leq\;\; O\bigg(\frac{p^{K+2+\Kodd+1}}{n^{K+1+\Kodd}}\bigg) 
	+ O\bigg(\frac{p\cdot p^{K+2-\Kodd+1}}{n^{K+2-\Kodd}}\bigg) 
	+ o\Big(\frac{p^{K+3}}{n^{K+1}}\Big)
\notag\\&\hspace{20pt}
=\;\; O\Big(\frac{p^{K+3}}{n^{K+1}}\Big)
\label{eqn:existence-bound-on-reals}
\end{align}}%
and
{\setlength{\mathindent}{5pt}
\setlength{\abovedisplayskip}{3pt}	
\setlength{\belowdisplayskip}{3pt}\begin{align}&
\E{\bigg|\Im\Log\frac{\psi_{\text{NW}}(T)}{\psi_K(T)} \bigg|}
\leq \frac{4^{2K+3}}{K+2}\frac{\E{\Tr T^{2(K+2)}}}{n^{K+1}}
	+ \frac{4^{2K+1}}{K+1}\frac{(p\!+\!1)\!\E{\Tr T^{2(K+2)}}}{n^{K+1}}
\notag\\&\hspace{20pt}
\leq\;\; O\bigg(\frac{p^{K+2+1}}{n^{K+1}}\bigg) 
+ O\bigg(\frac{p\cdot p^{K+2+1}}{n^{K+2}}\bigg)
\hspace{60pt}
=\;\; O\Big(\frac{p^{K+3}}{n^{K+1}}\Big)
\label{eqn:existence-bound-on-imaginary}
\end{align}}%
as $n\rightarrow\infty$ with $p/n\rightarrow0$.

Moreover, from Lemma \ref{lem:cnp}, we see that
{\setlength{\mathindent}{5pt}
\setlength{\abovedisplayskip}{3pt}	
\setlength{\belowdisplayskip}{3pt}\begin{align}&\hspace{30pt}
C_{n,p} 
= C_{n,p}^{(K)}\exp\left\{o\left(\frac{p^{K+3}}{n^{K+1}}\right)\right\}
\qquad\text{ as }n\rightarrow\infty 
\text{ with } \frac{p}{n}\rightarrow0.
\label{eq:thm-existence-cnp}
\end{align}}%
Thus, from Equations \eqref{eqn:psiK-bounded-by-psi-NW} and \eqref{eq:thm-existence-cnp}, we see that when $p^{K+3}/n^{K+1}\rightarrow0$, the asymptotic $L^1$ norm of $\psi_K$ is bounded by
{\setlength{\mathindent}{5pt}
\setlength{\abovedisplayskip}{3pt}	
\setlength{\belowdisplayskip}{3pt}\begin{align}&
\lim_{n\rightarrow\infty} \bigintss_{\S_p(\R)}\hspace{-20pt}
\big|\psi_K(T)\big| \,dT
\;\leq\; \lim_{n\rightarrow\infty} 
\exp\Big\{ \!-\!o\Big(\frac{p^{K+3}}{n^{K+1}}\Big) \Big\}
\bigintss_{\S_p(\R)} \hspace{-20pt} 
\big|\psi_{\text{NW}}(T)\big| \,dT
= 1.
\label{eq:thm-existence-psikispseudodensity}
\end{align}}%
In fact, at the end of this proof we will see that this bound is sharp and the limit must be exactly 1. 

Using Equations \eqref{eqn:existence-bound-on-reals}, \eqref{eqn:existence-bound-on-imaginary} and  \eqref{eq:thm-existence-psikispseudodensity} with Proposition \eqref{prop:generalizedkl} implies that when $p^{K+3}/n^{K+1}\rightarrow0$,
{\setlength{\mathindent}{5pt}
\setlength{\abovedisplayskip}{3pt}	
\setlength{\belowdisplayskip}{3pt}\begin{align*}
0 
	&\;\;\leq\;\; \lim_{n\rightarrow\infty} \mathrm{H}^2\Big(\psi_{\text{NW}}, \psi_K\Big) 
\\&
	\;\;\leq\;\; \lim_{n\rightarrow\infty} \bigg[
	\bigintss_{\S_{p}(\R)}\hspace{-20pt}
	|\psi_K|(T)\,dT - 1
	\bigg]
	+ 0
	+ 2 \lim_{n\rightarrow\infty}\bigintss_{\S_{p}(\R)}\hspace{-20pt}
	|\psi_K|(T)\,dT^{\frac12} \cdot 0^{\frac12}
\\&
\;\;\leq\;\; [1-1] + 0 + 2\cdot 1^{\frac12}\cdot 0^{\frac12} \; = 0.
\end{align*}}%
Thus $\mathrm{H}^2\big(\psi_{\text{NW}}, \psi_K\big)\rightarrow0$, hence by Equation \eqref{eqn:HvsTV-Gtransforms} we must have the limit $\mathrm{d}_\text{TV}\big(\psi_{\text{NW}}, \psi_K\big)\rightarrow 0$, as desired.

\textit{Only if statement.} For the second part of the theorem, assume that the total variation distance satisfies $\mathrm{d}_\text{TV}(\psi_{\text{NW}}, \psi_K)\rightarrow0$, hence $\mathrm{H}\big(\psi_{\text{NW}}, \psi_K\big)\rightarrow 0$ by Equation \eqref{eqn:HvsTV-Gtransforms}, as $n\rightarrow\infty$. We will show by contradiction this implies that $p^{K+3}/n^{K+1}\rightarrow0$. 

Assume this wasn't the case. Since $\lim_{n\rightarrow\infty}\frac{\log p}{\log n}<1$, there must be an $L\in\N$ such that $p^{L+3}/n^{L+1}\rightarrow0$, and since $p^{K+3}/n^{K+1}\nrightarrow0$, we must have $K<L$. By Equation \eqref{eqn:existence-bound-on-reals}, we must have for $T\sim |\psi_\text{NW}|=\text{T}_{n/2}(I_p/8)$ that
{\setlength{\mathindent}{40pt}
\setlength{\abovedisplayskip}{5pt}	
\setlength{\belowdisplayskip}{5pt}\begin{align*}
\lim_{n\rightarrow\infty} \E{\bigg|\Re\Log\frac{\psi_{\text{NW}}(T)}{\psi_L(T)} \bigg|} \leq \lim_{n\rightarrow\infty} O\Big(\frac{p^{L+3}}{n^{L+1}}\Big) =0,
\end{align*}}%
so
{\setlength{\mathindent}{40pt}
\setlength{\abovedisplayskip}{3pt}	
\setlength{\belowdisplayskip}{3pt}\begin{align}
\frac12\Re\Log\frac{\psi_{\text{NW}}(T)}{\psi_L(T)} \overset{L^1}{\longrightarrow}0.
\label{eqn:existence-RL-tends-to-zero}
\end{align}}%
Now write, by Equation \eqref{eq:thm-existence-psik} and Definition \ref{def:G-transform-approximations},
{\setlength{\mathindent}{10pt}
\setlength{\abovedisplayskip}{3pt}	
\setlength{\belowdisplayskip}{3pt}\begin{align*}&
R(T) \equiv\; \frac12\Re\Log\frac{\psi_L(T)}{\psi_K(T)} 
=\; \frac12\log\frac{C_{n,p}^{(L)}}{C_{n,p}^{(K)}}
+ \frac{n}4 \hspace{-7pt}\mathlarger{\mathlarger{\sum}}_{k=\substack{K + 2 + \\ \Kodd}}^{\substack{L + 1 + \\ \Lodd}}\hspace{-10pt}
	\frac{(-1)^k}{2k}\Tr\left(\frac{4T}{\sqrt{n}}\right)^{2k}
\\&\pushright{
+ \frac{p+1}4 \hspace{-7pt}\mathlarger{\mathlarger{\sum}}_{k=\substack{K + 2 - \\ \Kodd}}^{\substack{L + 1 - \\ \Lodd}}\hspace{-10pt}
	\frac{(-1)^k}{2k}\Tr\left(\frac{4T}{\sqrt{n}}\right)^{2k}
}
\\&
=\; -\frac14\!\sum_{k=K+2}^{L+1}\!\frac{\keven}{k(k\!+\!1)(k\!+\!2)}\frac{p^{k+2}}{n^k}
	-\!\frac18\!\sum_{k=K+2}^{L+1}\!\frac{1\!+\!2\keven}{k(k\!+\!1)}\frac{p^{k+1}}{n^k}
\\&\pushright{
+ \frac{n}4 \hspace{-5pt}\mathlarger{\mathlarger{\sum}}_{k=\substack{K + 2 + \\ \Kodd}}^{\substack{L + 1 + \\ \Lodd}}\hspace{-7pt}
	\frac{(-1)^k}{2k}\left(\frac{p}{n}\right)^{k}\Tr\left(\frac{4T}{\sqrt{p}}\right)^{2k}
+ \frac{p+1}4 \hspace{-5pt}\mathlarger{\mathlarger{\sum}}_{k=\substack{K + 2 - \\ \Kodd}}^{\substack{L + 1 - \\ \Lodd}}\hspace{-7pt}
	\frac{(-1)^k}{2k}\left(\frac{p}{n}\right)^{k}\Tr\left(\frac{4T}{\sqrt{p}}\right)^{2k}.
}
\end{align*}}%
But as $p^{L+3}/n^{L+1}$, we must have $p/n\rightarrow0$, so by Theorem \ref{thm:moments}, we have $\frac1{p}\Tr(\frac{4T}{\sqrt{p}})^{2k}\overset{L^2}{\rightarrow}C_{k}$. Moreover, as we assumed that $p^{K+3}/n^{K+1}\nrightarrow0$, we must have $p\rightarrow\infty$. Thus
{\setlength{\mathindent}{10pt}
\setlength{\abovedisplayskip}{3pt}	
\setlength{\belowdisplayskip}{5pt}\begin{align}&
\frac{n^{K+1}}{p^{K+3}}R(T)
= -\frac14 \sum_{k=K+2}^{L+1}\!\frac{\keven}{k(k\!+\!1)(k\!+\!2)}
\frac{p^{k-K-1}}{n^{k-K-1}}
\notag\\[-5pt]&\pushright{
	- \frac18\!\sum_{k=K+2}^{L+1}\!\frac{1\!+\!2\keven}{k(k\!+\!1)}\frac1{p}\frac{p^{k-K-1}}{n^{k-K-1}}
+ \frac{1}4 \hspace{-5pt}\mathlarger{\mathlarger{\sum}}_{k=\substack{K + 2 + \\ \Kodd}}^{\substack{L + 1 + \\ \Lodd}}\hspace{-7pt}
	\frac{(-1)^k}{2k}\frac{p^{k-K-2}}{n^{k-K-2}}
	\frac1{p}\Tr\left(\frac{4T}{\sqrt{p}}\right)^{2k}
}\notag\\[-5pt]&\pushright{
+ \frac{1}{4}\left(1+\frac1{p}\right) \hspace{-5pt}\mathlarger{\mathlarger{\sum}}_{k=\substack{K + 2 - \\ \Kodd}}^{\substack{L + 1 - \\ \Lodd}}\hspace{-7pt}
	\frac{(-1)^k}{2k}\frac{p^{k-K-1}}{n^{k-K-1}}
	\frac1{p}\Tr\left(\frac{4T}{\sqrt{p}}\right)^{2k}
}
\notag\\&\hspace{30pt}
\overset{L^2}{\longrightarrow}\;\;
0 + 0 + \frac{\Keven}{8(K+2)} C_{K+2} +\frac{\Kodd}{8(K+1)}C_{K+1}
\notag\\&\hspace{30pt}
\;\;=\;\; \frac{C_{K+1+\Keven}}{8(K+1+\Keven)}>0.
\label{eqn:existence-R-tends-to-positive}
\end{align}}%
Then by the reverse triangle inequality,
{\setlength{\mathindent}{30pt}
\setlength{\abovedisplayskip}{3pt}	
\setlength{\belowdisplayskip}{3pt}\begin{align*}&
0 = \lim_{n\rightarrow\infty}\mathrm{H}^2\big(\psi_{\text{NW}}, \psi_K\big) 
=\lim_{n\rightarrow\infty} \bigintss_{\S_p(\R)}\hspace{-15pt} 
	\Big|\psi_{\text{NW}}^{1/2}(T) - \psi_K^{1/2}(T)\Big|^2 \,dT
\\[-5pt]&\hspace{40pt}
\leq\;\; \lim_{n\rightarrow\infty} \bigintss_{\S_p(\R)}\hspace{-15pt} 
	\Big||\psi_{\text{NW}}|^{1/2}(T) - |\psi_K|^{1/2}(T)\Big|^2 \,dT
\\&\hspace{40pt}
=\;\; \lim_{n\rightarrow\infty} \E{\bigg|
\exp\bigg\{
\frac12\Re\Log\frac{\psi_K(T)}{\psi_\text{NW}(T)}\bigg\}-1
\bigg|^2}
\end{align*}}%
for a  $T\sim |\psi_\text{NW}|=\text{T}_{n/2}(I_p/8)$, that is
{\setlength{\mathindent}{30pt}
\setlength{\abovedisplayskip}{5pt}
\setlength{\belowdisplayskip}{5pt}\begin{align*}
\exp\bigg\{\frac12\Re\Log\frac{\psi_K(T)}{\psi_\text{NW}(T)}\bigg\} \overset{L^2}{\longrightarrow}1.
\end{align*}}%
Since $L^p$ convergence implies convergence in probability, by the continuous mapping theorem we must have
{\setlength{\mathindent}{30pt}
\setlength{\abovedisplayskip}{5pt}
\setlength{\belowdisplayskip}{5pt}\begin{align*}
\frac12\Re\Log\frac{\psi_K(T)}{\psi_\text{NW}(T)} \overset{P}{\longrightarrow}0
\end{align*}}%
as $n\rightarrow\infty$, so by Equation \eqref{eqn:existence-RL-tends-to-zero}
{\setlength{\mathindent}{30pt}
\setlength{\abovedisplayskip}{5pt}
\setlength{\belowdisplayskip}{5pt}\begin{align*}
R(T) = -\frac12\Re\Log\frac{\psi_K(T)}{\psi_\text{NW}(T)} 
-\frac12\Re\Log\frac{\psi_{\text{NW}}(T)}{\psi_L(T)}
\overset{P}{\longrightarrow}-0-0=0.
\end{align*}}%
But then, from Equation \eqref{eqn:existence-R-tends-to-positive} and Slutsky's lemma \citep[Lemma 2.8 (iii)]{vandervaart00},
{\setlength{\mathindent}{30pt}
\setlength{\abovedisplayskip}{5pt}
\setlength{\belowdisplayskip}{5pt}\begin{align*}
\frac{p^{K+3}}{n^{K+1}} = \left(\frac{n^{K+1}}{p^{K+3}}R(T)\right)^{-1}R(T) \overset{P}{\longrightarrow} \frac{8(K+1+\Keven)}{C_{K+1+\Keven}}\cdot0 = 0.
\end{align*}}%
as $n\rightarrow\infty$. As $p^{K+3}/n^{K+1}$ is deterministic, this implies that $p^{K+3}/n^{K+1}\rightarrow0$ as $n\rightarrow\infty$, a contradiction. Thus whenever $\mathrm{H}^2\big(\psi_{\text{NW}}, \psi_K\big)\rightarrow0$ as $n\rightarrow\infty$ with $\lim\limits_{n\rightarrow\infty}\frac{\log p}{\log n}<1$, we must have $p^{K+3}/n^{K+1}\rightarrow0$, as desired. This concludes the proof.
\end{proof}

Although Theorem \ref{thm:existence-gtransform} states that the functions $\psi_K$ approximate $\psi_\text{NW}$, there is no guarantee that they are G-transforms of a probability density. In other words, nothing guarantees that their inverse G-transforms $\tilde{f}_K=\mathcal{G}^{-1}\{\psi_K\}$ are real-valued, non-negative and integrate to unity. However, the reverse triangle inequality applied to the $L^2$-norm provides that
{\setlength{\mathindent}{30pt}
\setlength{\abovedisplayskip}{3pt}
\setlength{\belowdisplayskip}{3pt}\begin{align*}
\Big|
	\bigintss_{\S_p(\R)}\hspace{-20pt}     |\psi_{\text{NW}}(T)|\,dT^{\frac12}
	-
	\bigintss_{\S_p(\R)}\hspace{-20pt}     |\psi_K(T)|\,dT^{\frac12}
	\Big|
	\;\;\leq\;\; \mathrm{H}\big(\psi_{\text{NW}}, \psi_K\big),
\end{align*}}%
so Theorem \ref{thm:existence-gtransform} and the Plancherel theorem implies that
{\setlength{\mathindent}{30pt}
\setlength{\abovedisplayskip}{3pt}
\setlength{\belowdisplayskip}{3pt}\begin{align}
\lim_{n\rightarrow\infty}\bigintss_{\S_p(\R)}\hspace{-20pt}     |\tilde{f}_K(T)|\,dT
	= \lim_{n\rightarrow\infty} \bigintss_{\S_p(\R)}\hspace{-20pt}     |\psi_K(X)|\,dX
	= 1
\label{eqn:fK-is-asymptotic-density}
\end{align}}%
when $p^{K+3}/n^{K+1}\rightarrow0$. That is, the theorem at least guarantees that $|\tilde{f}_K|$ is asymptotically a density in its associated regime. We discuss this further in Section \ref{sec:wishart-density}.

We independently know, by the results of \citet{jiang15} and \citet{bubeck16a}, that a Gaussian orthogonal ensemble approximation holds in the classical regime. Although $\psi_0$ is not the G-transform of a $\text{GOE}(p)$, a simple Kullback-Leibler argument is sufficient to prove that it approximates $\psi_\text{GOE}$ for $0^\text{th}$ degree regimes.

\begin{proposition}\label{prop:goe-gtransforms} The total variation distance between the $0^\text{th}$ degree G-transform approximation $\psi_0$ and the Gaussian orthogonal ensemble G-transform $\psi_\text{GOE}$ satisfies $\mathrm{d}_\text{TV}(\psi_{0}, \psi_\text{GOE})\rightarrow0$ as $n\rightarrow\infty$ with $p^3/n\rightarrow0$.
\end{proposition}
\begin{proof}
We use a similar strategy to Theorem \ref{thm:existence-gtransform}: namely, by Equation \eqref{eqn:HvsTV-Gtransforms} is it equivalent to prove that $\mathrm{H}(\psi_{0}, \psi_\text{GOE})\rightarrow0$ as $n\rightarrow\infty$ with $p^3/n\rightarrow0$, and to control that quantity we can use the Kullback-Leibler inequality for G-transforms.

Let $T\sim|\psi_\text{GOE}|=\text{GOE}(p)/4$. Since the Gaussian orthogonal ensemble has been extensively studied, we understand well its empirical moments. For example, according to \citet[Lemma 2.2.2]{anderson10}, we have $\E{\Tr T^2}=O(p^2)$, while from Equation (2.1.45) of the same book we have $\E{|\Tr T|}=O(p^{1/2})$ and $\E{|\Tr T^3|}=O(p^{3/2})$. Then from Definition \ref{def:G-transform-approximations} and Proposition \ref{prop:gtransform-goe}, and using the projection map $P_{(-\pi, \pi)}$ as in the proof of Theorem \ref{thm:existence-gtransform}, we find that
{\setlength{\mathindent}{40pt}
\setlength{\abovedisplayskip}{5pt}	
\setlength{\belowdisplayskip}{5pt}\begin{align*}&
\E{\Re\Log\frac{\psi_\text{GOE}(T)}{\psi_0(T)}}
= -4\frac{p+1}{n}\!\E{\Tr T^2}
= O\bigg(\frac{p^3}{n}\bigg)
\end{align*}}%
and
{\setlength{\mathindent}{40pt}
\setlength{\abovedisplayskip}{5pt}	
\setlength{\belowdisplayskip}{5pt}\begin{align*}&
\E{\bigg|\Im\Log\frac{\psi_\text{GOE}(T)}{\psi_0(T)} \bigg|}
= \E{\bigg|P_{(-\pi, \pi)}\bigg[
	-\frac{32}{3\sqrt{n}}\!\Tr T^3 
	+\frac{p+1}{\sqrt{n}}\!\Tr T \bigg]\bigg|}
\notag\\&\hspace{20pt}
\leq\;\;
\frac{32}{3\sqrt{n}}\!\E{|\Tr T^3|} 
+\!2i\frac{p+1}{\sqrt{n}}\!\E{|\Tr T|}
%
\;\;=\;\; O\bigg(\sqrt{\frac{p^3}{n}}\bigg).
\end{align*}}%
Since $\int_{\S_p(\R)}|\psi_0|(T)dT\rightarrow 0$ when $n\rightarrow\infty$ with $p^3/n\rightarrow 0$ by Equation \eqref{eqn:fK-is-asymptotic-density}, if we apply Proposition \ref{prop:generalizedkl} we find thatt
{\setlength{\mathindent}{40pt}
\setlength{\abovedisplayskip}{5pt}	
\setlength{\belowdisplayskip}{5pt}\begin{align*}&
\lim_{n\rightarrow \infty}\mathrm{H}^2(\psi_0, \psi_\text{GOE}) 
\;\;\leq\;\; 
	0 + 0 + 2\cdot 1^{\frac12}\cdot 0^{\frac12} = 0
\end{align*}}%
for $p^3/n\rightarrow 0$. By Equation \eqref{eqn:HvsTV-Gtransforms}, this concludes the proof.
\end{proof}

As a consequence, $H(f_\text{NW}, f_\text{GOE}) = H(\psi_\text{NW}, \psi_\text{GOE}) \leq H(\psi_\text{NW}, \psi_0) + H(\psi_0, \psi_\text{GOE}) \rightarrow 0$ when $n\rightarrow\infty$ with $p^3/n\rightarrow 0$ by Theorem \ref{thm:existence-gtransform} and Proposition \ref{prop:goe-gtransforms}. Hence $\mathrm{d}_\text{TV}(f_\text{NW}, f_\text{GOE})\rightarrow 0$ by Equation \eqref{eqn:HvsTV-densities} in the classical setting. This provides an alternative proof of the results of \citet{jiang15} and \citet{bubeck16a}.
\section{Wishart asymptotics: the density point-of-view}\label{sec:wishart-density}

In Section \ref{sec:wishart-G-transform}, we studied the asymptotic behavior of the normalized Wishart distribution $\sqrt{n}[\text{W}_p(n, I_p/n)-I_p]$ using its G-transform $\psi_\text{NW}$. In particular, we derived an approximation to $\psi_\text{NW}$ for every middle-scale regime of a given degree. But although it is equivalent to study a probability distribution from a density or a G-transform point of view, it is still natural to wonder if we can find approximations to the \textit{density} of a normalized Wishart for every middle-scale regime of a given degree.

Recall from Theorem \ref{thm:existence-gtransform} that $\mathrm{d}_\text{TV}(\psi_\text{NW}, \psi_K)\rightarrow0$ when $p^{K+3}/n^{K+1}\rightarrow0$. Define $\tilde{f}_K = \mathcal{G}^{-1}\{\psi_K\}$. In general, there is no guarantee that these should be real-valued. On the other hand, we know from Equation \eqref{eqn:psiK-bounded-by-psi-NW} that whenever $n\geq p-2$, $\psi_K$ must be integrable, and since the G-transform maps integrable functions to integrable functions, $\tilde{f}_K$ must also be integrable. In fact, according to Equation \eqref{eqn:fK-is-asymptotic-density}, we know $|\tilde{f}_K|$ must be asymptotically a density when $p^{K+1}/n^{K+1}\rightarrow0$. This suggests we define the following densities.

\begin{definition}[Density approximations]\label{def:density-approximations} For any $K\in\N$ and $n\geq p-2$, define the $K^\text{th}$ degree density approximation as
\begin{ceqn}
{\setlength{\abovedisplayskip}{5pt}	
\setlength{\belowdisplayskip}{5pt}\begin{align*}
f_K(X) = \frac{|\tilde{f}_K|(X)}{\int_{\S_p(\R)} |\tilde{f}_K(Y)|dY},
\end{align*}}%
\end{ceqn}
where $\tilde{f}_K = \mathcal{G}^{-1}\{\psi_K\}$ and $\psi_K$ is as in Definition \ref{def:G-transform-approximations}. The distribution on the real symmetric matrices with density $f_K$ will be denoted $F_K$.
\end{definition}

The main interest is that we can asymptotically approximate the \textit{density} $f_\text{NW}$ of a normalized Wishart by the bona fide densities $f_K$. This was the content of Theorem \ref{thm:existence-densities} from Section \ref{sec:introduction}, which we now prove as a simple corollary of its G-transform analogue Theorem \ref{thm:existence-gtransform} from Section \ref{sec:wishart-G-transform}.

\noindent \begin{proof}[Proof of Theorem \ref{thm:existence-densities}]
As in the rest of this paper, we write the density of the normalized Wishart distribution $\sqrt{n}[W_p(n, I_p/n) - I_p]$ by $f_\text{NW}$, and by Definition \ref{def:density-approximations} the density of $F_K$ is $f_K$.
Notice that by Equation \eqref{eqn:HvsTV-densities}, to prove $\mathrm{d}_\text{TV}(f_\text{NW}, f_K)\rightarrow0$ it is equivalent to prove that $\mathrm{H}(f_\text{NW}, f_K)\rightarrow0$. From the triangle inequality, the reverse triangle inequality, Theorem \ref{thm:existence-gtransform} and Equation \eqref{eqn:fK-is-asymptotic-density},
{\setlength{\mathindent}{20pt}
\setlength{\abovedisplayskip}{5pt}	
\setlength{\belowdisplayskip}{5pt}\begin{align*}&
\lim_{n\rightarrow\infty}H\Big(f_\text{NW}, f_K\Big) 
\leq \lim_{n\rightarrow\infty}H\Big(f_\text{NW}, |\tilde{f}_K|\Big) 
+ \lim_{n\rightarrow\infty}H\Big(|\tilde{f}_K|, f_K\Big)
\\&\hspace{40pt}
=\; \lim_{n\rightarrow\infty}\bigintss_{\S_p(\R)}\hspace{-15pt} \Big|f^{1/2}_\text{NW}(X)-|\tilde{f}^{1/2}_K|(X)\Big|^2\,dX^{\frac12}
\\&\hspace{60pt}
+ \lim_{n\rightarrow\infty}\bigintss_{\S_p(\R)}\hspace{-17pt} \bigg||\tilde{f}^{1/2}_K|(X) - \frac{|\tilde{f}^{1/2}_K|(X)}{\int_{\S_p(\R)} |\tilde{f}_K|(Y)dY^{1/2}}\bigg|^2\,dX^{\frac12}
\\&\hspace{40pt}
\leq\; \lim_{n\rightarrow\infty}\bigintss_{\S_p(\R)}\hspace{-15pt} \Big|f^{1/2}_\text{NW}(X)-\tilde{f}^{1/2}_K(X)\Big|^2\,dX^{\frac12}
\\[-5pt]&\hspace{60pt}
+ \lim_{n\rightarrow\infty}\bigg|1 - \frac{1}{\int_{\S_p(\R)} |\tilde{f}_K|(Y)dY^{1/2}}\bigg|\bigintss_{\S_p(\R)}\hspace{-17pt} |\tilde{f}_K|(X)dX^{\frac12}
\\&\hspace{40pt}
=\; 0^{1/2} + \Big|1-\frac{1}{1^{1/2}}\Big|\cdot 1^{1/2} = 0.
\end{align*}}%
when $p^{K+3}/n^{K+1}\rightarrow0$. Thus $\mathrm{H}(f_\text{NW}, f_K)$, hence $\mathrm{d}_\text{TV}(f_\text{NW}, f_K)$, tends to zero when $n\rightarrow\infty$ with $p^{K+3}/n^{K+1}\rightarrow0$, as desired.
\end{proof}

We defined $f_K$ in terms of the inverse G-transform of the $\psi_K$ functions given by Definition \ref{def:G-transform-approximations}. How can we express this explicitely? By Equation \eqref{eqn:psiK-bounded-by-psi-NW}, we see that $|\psi_K|(T)$ is asymptotically bounded by the $\text{T}_{n/2}(I_p/8)$ density $|\psi_\text{NW}|(T)$, which is integrable whenever $n-p+2\geq0$. But $|\psi_\text{NW}|^{1/2}(T)$ is proportional to a $\text{T}_{m/4}(\frac{n}{4m}I_p)$ density in the sense of Definition \ref{def:tdistribution}, which is integrable for $m/4\geq p/2-1$, that is whenever $n-3p+3\geq0$. Thus  $|\psi_\text{NW}|^{1/2}(T)$ and therefore $|\psi_K|^{1/2}(T)$ is integrable whenever $n-3p+3\geq0$. Hence we can use the Fourer inversion theorem to conclude that $f_K$ is proportional to the integral
{\setlength{\mathindent}{20pt}
\setlength{\abovedisplayskip}{5pt}	
\setlength{\belowdisplayskip}{5pt}\begin{align}
f_K(X) \;\;&\propto\;\; \Big|\mathcal{G}^{-1}\{\psi_K\}\Big|(X)
\notag\\[-7pt]&
\propto\;\; \Bigg|
\bigintss_{\S_p(\R)} \hspace{-18pt}
\exp\Bigg\{
	i\Tr (XT)
	+ \frac{n}4
	\hspace{-6pt}\mathlarger{\mathlarger{\sum}}_{k=2}^{\substack{2K+3+\\ \Kodd}}\hspace{-10pt}
	{\Big(\frac{4i}{\sqrt{n}}\Big)\!\!}^k
	\frac{\Tr T^k}{k}
\notag\\[-5pt]&\hspace{120pt}
	+ \frac{p+\!1}4\hspace{-7pt}\mathlarger{\mathlarger{\sum}}_{k=1}^{\substack{2K+2-\\ \Kodd}}\hspace{-10pt}
	{\Big(\frac{4i}{\sqrt{n}}\Big)\!\!}^k
	\frac{\Tr T^k}{k}
	\Bigg\}
\,dT\Bigg|^2
\label{eqn:fK-as-integral}
\end{align}}%
whenever $n-3p+3\geq0$. In particular, if we do a change of variables $Z = \sqrt{8}T$, we obtain Equation \eqref{eqn:fK-as-expectation}  from Section \ref{sec:introduction} whenever $n\geq 3p-3$, from which we can derive Equations \eqref{eqn:f1-as-expectation}
 and \eqref{eqn:f2-as-expectation}.
 
It would be quite pleasant if there was a way to solve the integral in Equation \eqref{eqn:fK-as-integral} or \eqref{eqn:fK-as-expectation} and obtain a (potentially quite complicated) closed form expression for $f_K$ up to its normalization constant. So far, our efforts have been unfruitful.

We close our discussion with a final remark. At the end of Section \ref{sec:wishart-G-transform}, we showed that $\psi_0$ approximates $\psi_\text{GOE}$ in $0^\text{th}$ degree middle-scale regimes, from which the classical asymptotic normality follows. It is natural to wonder if $f_0$ approximates $f_\text{GOE}$ in the same context. An argument similar to that of Theorem \ref{thm:existence-densities} shows this is the case.

\begin{proposition}\label{prop:goe-density} The total variation distance between the $0^\text{th}$ degree density approximation $f_0$ and the Gaussian orthogonal ensemble G-transform $f_\text{GOE}$ satisfies $\mathrm{d}_\text{TV}(f_{0}, f_\text{GOE})\rightarrow0$ as $n\rightarrow\infty$ with $p^3/n\rightarrow0$.
\end{proposition}
\begin{proof}
The Hellinger distance between $f_0$ and $f_\text{GOE}$ satisfies
{\setlength{\mathindent}{20pt}
\setlength{\abovedisplayskip}{5pt}	
\setlength{\belowdisplayskip}{5pt}\begin{align*}&
\lim_{n\rightarrow\infty} \mathrm{H}\Big(f_0, f_\text{GOE}\Big) 
\leq \lim_{n\rightarrow\infty} \mathrm{H}\Big(f_0, |\tilde{f}_0|\Big) 
+ \lim_{n\rightarrow\infty} \mathrm{H}\Big(|\tilde{f}_0|, f_\text{GOE}\Big)
\\&\hspace{10pt}
\leq \lim_{n\rightarrow\infty} \bigg|1 - \frac{1}{\int_{\S_p(\R)} |\tilde{f}_0|(Y)dY^{1/2}}\bigg|\bigintss_{\S_p(\R)}\hspace{-17pt} |\tilde{f}_0|(X)dX^{\frac12}
+ \lim_{n\rightarrow\infty} \mathrm{H}\Big(\tilde{f}_0, f_\text{GOE}\Big)
\\&\hspace{10pt}
= \Big|1-\frac{1}{1^{1/2}}\Big|\cdot 1^{1/2} + \lim_{n\rightarrow\infty} \mathrm{H}\Big(\psi_0, \psi_\text{GOE}\Big)
=0.
\end{align*}}%
By Equation \eqref{eqn:HvsTV-densities}, the result follows.
\end{proof}

Of course, we could conclude from this that $H(f_\text{NW}, f_\text{GOE}) \leq H(f_\text{NW}, f_0) + H(f_0, f_\text{GOE}) \rightarrow 0$ when $n\rightarrow\infty$ with $p^3/n\rightarrow 0$, offering yet again another proof that a Gaussian orthogonal ensemble approximation holds in the classical setting.
\section{The effect of phase transitions}\label{sec:explanation}

Although we have established the existence of phase transitions, it does not shed much light on \textit{how} the behavior of a normalized Wishart distribution might differ across phase transitions. To do this, it can be very illuminating to study the asymptotics of some of its statistics. For example, we could study its empirical moments. 

For a normalized Wishart matrix $X\sim \sqrt{n}[\text{W}_p(n, I_p/8)-I_p]$, a direct computation yields
{\setlength{\mathindent}{15pt}
\setlength{\abovedisplayskip}{7pt}
\setlength{\belowdisplayskip}{7pt}\begin{align*}
\E{\bigg(\frac1{p}\Tr \Big(\frac{X}{\sqrt{p}}\Big)^2 - 1\bigg)^{\raisebox{-3pt}{$\scriptstyle 2$}}}
&= \frac{5}{p^2}+\frac{4}{p^3}+\frac{8}{np}+\frac{20}{np^2}+\frac{20}{np^3}
= \frac{5}{p^2}+ o\left(\frac1{p^2}\right)
\end{align*}}%
so in every middle-scale regime, that is whenever $n,p\rightarrow\infty$ with $p/n\rightarrow0$,
{\setlength{\mathindent}{20pt}
\setlength{\abovedisplayskip}{7pt}
\setlength{\belowdisplayskip}{7pt}\begin{align*}&
\bigg\|\frac1{p}\Tr \Big(\frac{X}{\sqrt{p}}\Big)^2 - 1\bigg\|_{L^2}
\sim 
\;\dfrac{\sqrt{5}}{p}.
\hspace{110pt}
\Big(\text{\raisebox{0pt}{\parbox{6.2em}{\small all middle-scale regimes}}}\,\Big)
\end{align*}}%
Thus we have $L^2$ convergence of the second empirical moment to $1$, but otherwise nothing very interesting. There doesn't seem to be any change of behavior across the different middle-scale regimes. In contrast, the situation with the symmetric $t$ distribution is striking, and  illustrates yet again that middle-scale regime behavior becomes clearer under a G-transform. Indeed, we know from Theorem \ref{thm:moments} that for a $T\sim \text{T}_{n/2}(I_p/8)$, the quantity $\frac1{p}\Tr(\frac{4T}{\sqrt{p}})^2$ also converges to 1, but we know more. At Equation \eqref{eqn:moments-example-L2distance}, we computed the exact $L^2$ distance between $\frac1{p}\Tr(\frac{4T}{\sqrt{p}})^2$ and $1$, and found that
{\setlength{\mathindent}{0pt}
\setlength{\abovedisplayskip}{7pt}
\setlength{\belowdisplayskip}{7pt}\begin{align*}&\hspace{15pt}
\E{\bigg(\frac1{p}\Tr \Big(\frac{4T}{\sqrt{p}}\Big)^2 - 1\bigg)^{\raisebox{-3pt}{$\scriptstyle 2$}}}
= \frac{m^5}{\left(m\!-\!6\right) \left(m\!-\!2\right) \left(m\!-\!1\right) \left(m\!+\!1\right) \left(m\!+\!3\right)}
\notag\\&\pushright{
\cdot
\bigg[
\frac{5}{p^{2}} + \frac{2}{m} + \frac{p^2}{m^2}
+o\left(\frac1{p^2} + \frac{2}{m}+\frac{p^2}{m^2}\right)
\bigg]}.
\end{align*}}%
Thus the $L^2$ distance must have middle-scale asymptotics
{\setlength{\mathindent}{20pt}
\setlength{\abovedisplayskip}{7pt}
\setlength{\belowdisplayskip}{7pt}\begin{align}&
\bigg\|\frac1{p}\Tr \Big(\frac{4T}{\sqrt{p}}\Big)^2 - 1\bigg\|_{L^2}
\notag\\&\hspace{0pt}
\sim \begin{cases}
\hspace{45pt}\dfrac{\sqrt{5}}{p}
&\text{for}\;\; \dfrac{p^2}{n}\rightarrow 0 \hspace{12pt}
\Big(\,\text{\raisebox{0pt}{\parbox{6.5em}{\small classical or first \\degree}}}\Big)
\\[15pt]
\;\;
\parbox{12em}{$\sqrt{5+2\alpha+\alpha^2}\dfrac{1}{p} 
\\= \sqrt{5\alpha^{-1}+2+\alpha}\dfrac{1}{\sqrt{n}}
\\= \sqrt{5\alpha^{-2}+2\alpha^{-1}+1}\dfrac{p}{n}$}
&\text{for}\;\; \dfrac{p^2}{n}\rightarrow\alpha \hspace{10pt}
\big(\,\text{\parbox{5.7em}{\small second degree}}\big)
\\[25pt]
\hspace{50pt}\dfrac{p}{n}
&\text{for}\;\; \dfrac{p^2}{n}\rightarrow\infty \hspace{8pt}
\Big(\,\text{\raisebox{0pt}{\parbox{6.5em}{\small second or higher \\degree}}}\,\Big)
\end{cases}
\label{eqn:second-moment-t-limit}
\end{align}}%
as $n,p\rightarrow\infty$ with $p/n\rightarrow0$. Thus there is a sharp change in behavior of $\frac1{p}\Tr(\frac{4T}{\sqrt{p}})^2$ when $p$ grows like $\sqrt{n}$, and despite the symmetric $t$ distribution satisfying a semicircle law according to Corollary \ref{cor:semicircle}, it must ultimately behave differently than a Gaussian orthogonal ensemble matrix. The first-order asymptotics look the same: it is rather in the rate of this convergence that they differ.

This matters for both the symmetric $t$ and the Wishart distribution because rates of convergence can be distinguished in the strong topology. As a simple example, consider the sequence of one-dimensional distributions
\begin{ceqn}
{\setlength{\abovedisplayskip}{5pt}
\setlength{\belowdisplayskip}{5pt}
\begin{align*}
F_p = \text{N}\big(0, 1/p\big),
\qquad
\text{and} \; \; \; G_p = \text{N}\big(0, 1/p^2\big).
\end{align*}}%
\end{ceqn}
In the weak topology, these are asymptotically the same, since they converge to the same distribution -- namely $F_p, G_p\Rightarrow \delta_0$ as $p\rightarrow\infty$, for $\delta_0$ the Dirac measure at $0$. In other words, in a metric that induces the weak topology such as the L\'evy metric, 
{\setlength{\mathindent}{30pt}
\setlength{\abovedisplayskip}{5pt}
\setlength{\belowdisplayskip}{5pt}\begin{align*}
d_\text{L\'evy}(F_p, G_p)\rightarrow 0.
\end{align*}}%
Yet, by a direct computation of the Hellinger distance, which induces the strong topology,
{\setlength{\mathindent}{30pt}
\setlength{\abovedisplayskip}{0pt}
\setlength{\belowdisplayskip}{5pt}\begin{align*}
d_\text{Hellinger}(F_p, G_p) &
= H(F_p, G_p) 
= \sqrt{2}\sqrt{1-\Big(\frac{4p}{p^2+2p+1}\Big)^{1/4}}
\\&
\rightarrow \sqrt{2} >0
\end{align*}}%
as $p\rightarrow\infty$. Thus it is clear that the strong topology captures rates of convergence in a way that the weak topology can't. But then, we should expect a phase transition when $p$ grows like $\sqrt{n}$ for the $\text{T}_{n/2}(I_p/8)$ distribution. And since the symmetric $t$ is the G-conjugate of the Wishart, this should imply a phase transition when $p$ grows like $\sqrt{n}$ for the Wishart distribution as well. This is consistent with Theorem \ref{thm:existence-gtransform}, and provides an alternative explanation for the existence of the second phase transition.

A natural question then is to ask whether we can find symmetric $t$ statistics that exemplify all the middle-scale regime phase transitions. It is tempting to look at the $L^2$ error of the other empirical moments of the symmetric $t$ distribution, because we can use the methodology developed in Section \ref{sec:symmetrict} to compute their asymptotics to arbitrary order. As a reference, we compiled a table of the few first few moments as Table \ref{tab:asymptotics}.

\begin{table}
{\renewcommand{\arraystretch}{2}
\begin{tabular}{|c|c|c|}
\hline
Normalized empirical moment & Limit & Asymptotics of its squared $L^2$ error
\\ \hline
$\dfrac1{p}\Tr\Big(\dfrac{4T}{\sqrt{p}}\Big)$ & 0 
& $\dfrac{2}{p^2}$
\\[5pt] \hline
$\dfrac1{p}\Tr\Big(\dfrac{4T}{\sqrt{p}}\Big)^2$ & $C_1=1$ 
& $\dfrac{5}{p^{2}} + \dfrac{2}{m} + \dfrac{p^2}{m^2}$
\\[5pt] \hline
$\dfrac1{p}\Tr\Big(\dfrac{4T}{\sqrt{p}}\Big)^3$ & 0 
& $\dfrac{24}{p^{2}}$
\\[5pt] \hline
$\dfrac1{p}\Tr\Big(\dfrac{4T}{\sqrt{p}}\Big)^4$ & $C_2=2$ 
& $\dfrac{97}{p^{2}} + \dfrac{50}{m} + \dfrac{25p^2}{m^2}$
\\[5pt] \hline
\noalign{\smallskip}
\end{tabular}}%
\caption{Asymptotics of small normalized empirical moments of $T\sim\text{T}_{n/2}(I_p/8)$.}
\label{tab:asymptotics}
\end{table}

As can be seen from the table, the odd moments seem to have uniform behavior across all middle-scale regimes. In contrast, the even moments seem to all change asymptotics at the second phase transition $p=\Theta(\sqrt{n})$, but nowhere else. Hence finding statistics that ``flag'' the other phase transitions remain an open question.
\section{Auxiliary results}\label{sec:auxiliary}

This section compiles several lemmas used elsewhere in the article.

\begin{lemma}[First derivatives lemma]\label{lem:first-derivatives-lemma}\hspace{-10pt}
For any indices $1\leq i_1,\dots, i_{2l}\leq p$ and real symmetric matrix $Z$, there exist polynomials $a_{J,s}(n,m)$ in $n$ and $m=n-p-1$, indexed by $0\leq s\leq l$ and $J=(j_1, \dots, j_{2l})$, such that
{\setlength{\mathindent}{10pt}\begin{align*}&
\frac{\partial_\text{s}}{\partial_\text{s} Z_{i_{2l} i_{2l-1}}}
	\dots \frac{\partial_\text{s}}{\partial_\text{s} Z_{i_4i_3}}\frac{\partial_\text{s}}{\partial_\text{s} Z_{i_2i_1}}
	\exp\bigg\{\!\!-\!\frac{n}4\Tr Z\bigg\}\big|Z\big|^{\frac{m}4}
\\&\hspace{20pt}
=\sum_{s=0}^l\sum_{\substack{J\in\\\{1,\dots,p\}^{2l}}} \hspace{-7pt}
	a_{J,s}(n, m)
	\!\!\prod_{t=s+1}^l \!\!(I_p)_{j_{2t}j_{2t-1}}
	\prod_{t=1}^sZ^{-1}_{j_{2t}j_{2t-1}}
	\exp\bigg\{\!\!-\!\frac{n}4\Tr Z\bigg\}\big|Z\big|^{\frac{m}4}.
\end{align*}}%
\end{lemma}
\begin{proof}
To simplify notation, let
{
\setlength{\abovedisplayskip}{3pt}
\setlength{\belowdisplayskip}{3pt}\begin{align*}
M_{J, s}(Z) = \prod_{t=s+1}^l \!\!(I_p)_{j_{2t}j_{2t-1}}
	\prod_{t=1}^sZ^{-1}_{j_{2t}j_{2t-1}}
	\exp\bigg\{\!\!-\!\frac{n}4\Tr Z\bigg\}\big|Z\big|^{\frac{m}4},
\end{align*}}%
and let $M_l=\{M_{J,s} \,\vert\, J\in\{1,\dots,p\}^{2l}, s\leq l\}$ be the set of all such terms ``on $2l$ indices''. Let $\langle M_l\rangle$ denote the linear span of $M_l$, that is, the space of all linear combinations of elements of $M_l$, with as coefficients real polynomials in $n$ and $m$. Then we are really claiming that
{\setlength{\mathindent}{10pt}\begin{align}&
\frac{\partial_\text{s}}{\partial_\text{s} Z_{i_{2l} i_{2l-1}}}
	\dots \frac{\partial_\text{s}}{\partial_\text{s} Z_{i_4i_3}}\frac{\partial_\text{s}}{\partial_\text{s} Z_{i_2i_1}}
	\exp\bigg\{\!\!-\!\frac{n}4\Tr Z\bigg\}\big|Z\big|^{\frac{m}4}
	\quad\in \langle M_l\rangle.
	\label{eqn:fdl-goal}
\end{align}}%
To see this, let $J=(j_1,\dots, j_{2l-2})\in\{1,\dots,p\}^{2l-2}$ and define the extension $J_{a,b}^q=(j_1,\dots,j_{q-1},a,b,j_{q+1},\dots, j_{2l-2})\in\{1,\dots,p\}^{2l}$ to be $J$ with indices $a$, $b$ inserted (in this order) at the $q^\text{th}$ position. Then using that
{\setlength{\mathindent}{10pt}
\setlength{\abovedisplayskip}{5pt}
\setlength{\belowdisplayskip}{5pt}\begin{align*}&
\frac{\partial_\text{s}}{\partial_\text{s} Z_{i_{2l}i_{2l-1}}}Z_{ab}^{-1}
	= -\frac12\Big[Z_{ai_{2l}}^{-1}Z_{i_{2l-1}b}^{-1}+Z_{ai_{2l-1}}^{-1}Z_{i_{2l}b}^{-1}\Big]
\end{align*}}%
and
{\setlength{\mathindent}{10pt}
\setlength{\abovedisplayskip}{3pt}
\setlength{\belowdisplayskip}{3pt}\begin{align*}&
\frac{\partial_\text{s}}{\partial_\text{s} Z_{i_{2l}i_{2l-1}}} 
	\exp\bigg\{\!\!-\!\frac{n}4\Tr Z\bigg\}\big|Z\big|^{\frac{m}4}
\\&\hspace{100pt}
= \Big[\frac{m}4Z_{i_{2l}i_{2l-1}} - \frac{n}4(I_p)_{i_{2l}i_{2l-1}}\Big]
	\exp\bigg\{\!\!-\!\frac{n}4\Tr Z\bigg\}\big|Z\big|^{\frac{m}4},
\end{align*}}%
we conclude that
{\setlength{\mathindent}{10pt}
\setlength{\abovedisplayskip}{3pt}
\setlength{\belowdisplayskip}{3pt}\begin{align*}&
\frac{\partial_\text{s}}{\partial_\text{s} Z_{i_{2l}i_{2l-1}}} M_{J,s}(Z)
=-\frac12\sum_{r=1}^s M_{J_{i_{2l}i_{2l-1}}^{2r}, s+1}
	-\frac12\sum_{r=1}^s M_{J_{i_{2l-1}i_{2l}}^{2r}, s+1}
\\&\hspace{100pt}
+\frac{m}4 M_{J_{i_{2l}i_{2l-1}}^{2s+1}, s+1}
-\frac{n}4 M_{J_{i_{2l}i_{2l-1}}^{2s+1}, s}
\hspace{50pt}
\in\langle M_l\rangle.
\end{align*}}%
Thus, by linearity, $\partial_\text{s}/\partial_\text{s}Z_{i_{2l}i_{2l-1}}$ maps $\langle M_{l-1}\rangle$ to $\langle M_l\rangle$. But naturally we have $\exp\{-\frac{n}4\Tr Z\}|Z|^{m/4}\in \langle M_0\rangle$, so by induction Equation \eqref{eqn:fdl-goal} must then hold, as desired.
\end{proof}

\begin{lemma}[Second derivatives lemma]\label{lem:second-derivatives-lemma}\hspace{-12pt}
For any $k\in\N$ and any $Z\in\S_p(\R)$,
{\setlength{\mathindent}{10pt}
\setlength{\abovedisplayskip}{0pt}
\setlength{\belowdisplayskip}{-5pt}\begin{align*}&
{\mathlarger{\mathlarger\sum_{i_1,\dots,i_{2k}}^p}}
	\frac{\partial_\text{s}}{\partial_\text{s} Z_{i_1 i_{2k}}}
	\dots \frac{\partial_\text{s}}{\partial_\text{s} Z_{i_3i_2}}
	\frac{\partial_\text{s}}{\partial_\text{s} Z_{i_2i_1}}
e^{-\frac{n}4\Tr Z} \big|Z\big|^{\frac{m}4}
\\[-5pt]&\pushright{
= e^{-\frac{n}4\Tr Z} \big|Z\big|^{\frac{m}4}
	\hspace{-5pt}\sum_{|\kappa|\leq 2k}\hspace{-3pt}
	b_\kappa^{(1)}(n, m, p) r_\kappa(Z^{-1})
} \end{align*}}%
and
{\setlength{\mathindent}{10pt}
\setlength{\abovedisplayskip}{3pt}
\setlength{\belowdisplayskip}{0pt}\begin{align*}&
{\mathlarger{\mathlarger\sum_{\substack{i_1,\dots,i_{k}\\j_1,\dots,j_{k}}}^p}}
	\frac{\partial_\text{s}}{\partial_\text{s} Z_{j_1 j_{k}}}
	\dots \frac{\partial_\text{s}}{\partial_\text{s} Z_{j_3j_2}}
	\frac{\partial_\text{s}}{\partial_\text{s} Z_{j_2j_1}}
	\frac{\partial_\text{s}}{\partial_\text{s} Z_{i_1 i_{2k}}}
	\dots \frac{\partial_\text{s}}{\partial_\text{s} Z_{i_3i_2}}
	\frac{\partial_\text{s}}{\partial_\text{s} Z_{i_2i_1}}
e^{-\frac{n}4\Tr Z} \big|Z\big|^{\frac{m}4}
\\[-5pt]&\pushright{
= e^{-\frac{n}4\Tr Z} \big|Z\big|^{\frac{m}4}
	\hspace{-7pt}\sum_{|\kappa|\leq 2k+1}\hspace{-7pt}
	b_\kappa^{(2)}(n, m, p) r_\kappa(Z^{-1})
} \end{align*}}%
for some polynomials $b_\kappa^{(1)}(n,m,p)$ and $b_\kappa^{(2)}(n,m,p)$ with degrees $\mathrm{deg}\,b_\kappa^{(1)} \leq 2k+1-q(\kappa)$ and $\mathrm{deg}\,b_\kappa^{(2)} \leq 2k+2-q(\kappa)$. The sums at the right hand sides are taken over all integer partitions $\kappa$ of norm at most $2k$ and $2k+1$, including the empty partition.
\end{lemma}
\begin{proof}
We give a spectral proof. Let $OLO^t$ be the spectral decomposition of $Z$, with eigenvalues $\lambda_1\geq \dots\geq\lambda_p$, and notice that
{\setlength{\mathindent}{10pt}
\setlength{\abovedisplayskip}{5pt}
\setlength{\belowdisplayskip}{5pt}\begin{align*}
\frac{\partial_\text{s}O_{hl}}{\partial_\text{s}Z_{ij}} = 
	\frac12\sum_{a\neq l}^p\frac{O_{ha}O_{ai}^t}{\lambda_l-\lambda_a}O_{jl}
	+ \frac12\sum_{a\neq l}\frac{O_{ha}O_{aj}^t}{\lambda_l-\lambda_a}O_{il},
\hspace{40pt}
\frac{\partial_\text{s}\lambda_h}{\partial_\text{s}Z_{ij}} = O_{ih}O_{jh}
\end{align*}}%
for any $1\leq i,j,h,l\leq p$. As a consequence, for any differentiable real-valued functions $F_1(L),\dots,F_p(L)$, we have
{\setlength{\mathindent}{10pt}
\setlength{\abovedisplayskip}{3pt}
\setlength{\belowdisplayskip}{3pt}\begin{align*}
\sum_{j=1}^p\frac{\partial_\text{s}}{\partial_\text{s}Z_{hj}}
	\bigg(\sum_{a=1}^p O_{ja}F_aO_{ai}^t\bigg)
	= \frac12\sum_{\substack{a,b\\b\neq a}}^p O_{ha}\frac{F_b-F_a}{\lambda_b - \lambda_a}O_{ai}^t
	+ \sum_{a=1}^pO_{ha}\frac{\partial F_a}{\partial\lambda_a}O_{ai}^t.
\end{align*}}%
This suggests we define a new operator $D_L$ that would map the space of diagonal matrices $F(L)=\mathrm{diag}(F_1(L),\dots, F_p(L))$ that differentially depends on $L$, to itself, by
{\setlength{\mathindent}{10pt}
\setlength{\abovedisplayskip}{3pt}
\setlength{\belowdisplayskip}{3pt}\begin{align*}
D_L\{F\}_a = \frac12\sum_{b\neq a}^p \frac{F_b-F_a}{\lambda_b - \lambda_a}
	+ \frac{\partial F_a}{\partial\lambda_a}
\quad\text{so that}\quad
\sum_{j=1}^p\frac{\partial_\text{s}}{\partial_\text{s}Z_{hj}} OFO^t_{ji} = OFO^t_{ki}.
\end{align*}}%
In particular,
{\setlength{\mathindent}{10pt}
\setlength{\abovedisplayskip}{3pt}
\setlength{\belowdisplayskip}{3pt}\begin{align}&
{\mathlarger{\mathlarger\sum_{i_1,\dots,i_{2k}}^p}}
	\frac{\partial_\text{s}}{\partial_\text{s} Z_{i_1 i_{2k}}}
	\dots \frac{\partial_\text{s}}{\partial_\text{s} Z_{i_3i_2}}
	\frac{\partial_\text{s}}{\partial_\text{s} Z_{i_2i_1}}
	e^{-\frac{n}4\Tr Z} \big|Z\big|^{\frac{m}4}
\notag\\&\qquad
= {\mathlarger{\mathlarger\sum_{i_1,\dots,i_{2k}}^p}}
	\frac{\partial_\text{s}}{\partial_\text{s} Z_{i_0 i_{2k}}}
	\dots \frac{\partial_\text{s}}{\partial_\text{s} Z_{i_3i_2}}
	\frac{\partial_\text{s}}{\partial_\text{s} Z_{i_2i_1}}
	\Big[e^{-\frac{n}4\Tr Z} \big|Z\big|^{\frac{m}4}I_p\Big]_{i_1i_0}
\notag\\&\qquad
= \Tr D_L^{2k}\Big\{e^{-\frac{n}4\Tr Z} \big|Z\big|^{\frac{m}4}I_p\Big\},
\label{eqn:sdl-DL1}
\end{align}}%
and similarly
{\setlength{\mathindent}{10pt}
\setlength{\abovedisplayskip}{3pt}
\setlength{\belowdisplayskip}{3pt}\begin{align}&
{\mathlarger{\mathlarger\sum_{\substack{i_1,\dots,i_{k}\\j_1,\dots,j_{k}}}^p}}
	\frac{\partial_\text{s}}{\partial_\text{s} Z_{j_1 j_{k}}}
	\dots \frac{\partial_\text{s}}{\partial_\text{s} Z_{j_3j_2}}
	\frac{\partial_\text{s}}{\partial_\text{s} Z_{j_2j_1}}
	\frac{\partial_\text{s}}{\partial_\text{s} Z_{i_1 i_{k}}}
	\dots \frac{\partial_\text{s}}{\partial_\text{s} Z_{i_3i_2}}
	\frac{\partial_\text{s}}{\partial_\text{s} Z_{i_2i_1}}
	e^{-\frac{n}4\Tr Z} \big|Z\big|^{\frac{m}4}
\notag\\&\qquad
= \Tr D_L^{k}\Big\{\Tr D_L^{k}\Big\{e^{-\frac{n}4\Tr Z} \big|Z\big|^{\frac{m}4}I_p\Big\}\Big\}.
\label{eqn:sdl-DL2}
\end{align}}%
Let us look more closely at this operator $D_L$. It satisfies the following.
\begin{enumerate}[label=(\roman*)]
\item $D_L$ is linear, in the sense that for diagonals $F(L)$, $G(L)$ and constants $a$, $b$ with respect to $L$,
{\setlength{\abovedisplayskip}{3pt}
\setlength{\belowdisplayskip}{3pt}\begin{align*}
D_L\{aF+bG\} = aD_L\{F\}+bD_L\{G\}.
\end{align*}}%
\item $D_L$ satisfies a restricted product rule, in the sense that for a diagonal $F(L)$ of the form $F(L)=f(L)I_p$ for some function $f(L)$, and any diagonal $G(L)$,
{\setlength{\abovedisplayskip}{3pt}
\setlength{\belowdisplayskip}{3pt}\begin{align*}
D_L\{FG\} = D_L\{F\}G+FD_L\{G\}.
\end{align*}}%
\end{enumerate}
Moreover, from the definition of $D_L$,
{\setlength{\mathindent}{0pt}
\setlength{\abovedisplayskip}{3pt}
\setlength{\belowdisplayskip}{3pt}\begin{align*}
D_L\big\{e^{-\frac{n}4\Tr L}I_p\big\} &= -\frac{n}4e^{-\frac{n}4\Tr L}I_p, \qquad
D_L\big\{|L|^{\frac{m}4}I_p\big\} = \frac{m}4|L|^{\frac{m}4}I_p, \\
D_L\big\{\Tr(L^{-s})I_p\big\} &= -s L^{-(s+1)}, \\
\text{and}\hspace{40pt} D_L\big\{L^{-s}\big\} &= -\frac{s}2L^{-(s+1)}-\frac12\sum_{t=1}^s\Tr(L^{-[s+1-t]}) L^{-t}.
\end{align*}}%
Now define the spaces
{\setlength{\mathindent}{0pt}
\setlength{\abovedisplayskip}{3pt}
\setlength{\belowdisplayskip}{5pt}\begin{align*}
M_l=\left\{
	b(n,m,p) e^{-\frac{n}4\Tr L} \big|L\big|^{\frac{m}4} r_\kappa(L^{-1}) L^{-s}
	\;\Bigg\vert\;
	\text{\raisebox{0pt}{\parbox{13.5em}{$b(n,m,p)$ is a polynomial with degree at most $l-q(\kappa)$, and $\kappa$ and $s$ satisfy $|\kappa|\leq l-s$.}}}
\right\}
\end{align*}}%
for $l=1,\dots,2k$, and let $\langle M_l\rangle$ denote the linear span of $M_l$, i.e. the space of all real linear combinations of elements of $M_l$. Moreover, for a partition $\kappa$, let $\kappa\pm i$ denote $\kappa$ with the integer $i$ added or removed, respectively. For example, $(3,1,1,1)+2=(3,2,1,1,1)$ and $(3,2,1,1,1)-1=(3,2,1,1)$. Note that $|\kappa\pm i|=|\kappa|\pm i$. Then, for any $F\in M_l$,
{\setlength{\mathindent}{10pt}
\setlength{\abovedisplayskip}{3pt}
\setlength{\belowdisplayskip}{3pt}\begin{align*}
D_L\{F\} &= D_L\Big\{ b(n,m,p)
						e^{-\frac{n}4\Tr L}\big|L\big|^{\frac{m}4} r_\kappa(L^{-1}) L^{-s} 
						\Big\} 
\\&
= b(n,m,p) D_L\Big\{ e^{-\frac{n}4\Tr L}I_p\Big\} \big|L\big|^{\frac{m}4} r_\kappa(L^{-1}) L^{-s}
\\&\quad
	+ b(n,m,p) e^{-\frac{n}4\Tr L} D_L\Big\{ \big|L\big|^{\frac{m}4}I_p\Big\} r_\kappa(L^{-1}) L^{-s}
\\&\quad
	+ b(n,m,p) e^{-\frac{n}4\Tr L} \big|L\big|^{\frac{m}4} D_L\Big\{ r_\kappa(L^{-1})I_p\Big\} L^{-s}
\\&\quad
	+ b(n,m,p) e^{-\frac{n}4\Tr L} \big|L\big|^{\frac{m}4} r_\kappa(L^{-1}) D_L\Big\{ L^{-s}\Big\}
\\&
= \Big[-\frac{n}4b(n,m,p)\Big]e^{-\frac{n}4\Tr L}\big|L\big|^{\frac{m}4} r_\kappa(L^{-1}) L^{-s}
\\&\quad
	+ \Big[\frac{m}4b(n,m,p)\Big]e^{-\frac{n}4\Tr L}\big|L\big|^{\frac{m}4} r_\kappa(L^{-1}) L^{-s}
\\&\quad
	+\sum_{i=1}^{q(\kappa)} \Big[-\kappa_i b(n,m,p)\Big]
		e^{-\frac{n}4\Tr L}\big|L\big|^{\frac{m}4} r_{\kappa-\kappa_i}(L^{-1}) L^{-s}
\\&\quad
	+\Big[-\frac{s}2b(n,m,p)\Big]
		e^{-\frac{n}4\Tr L}\big|L\big|^{\frac{m}4} r_\kappa(L^{-1}) L^{-(s+1)}
\\&\quad
	+\sum_{t=1}^{s} \Big[-\frac12 b(n,m,p)\Big]
		e^{-\frac{n}4\Tr L}\big|L\big|^{\frac{m}4} r_{\kappa+(s+1-t)}(L^{-1}) L^{-t}
\end{align*}}%
Thus $D_L\{F\}\in\langle M_{l+1}\rangle$. It follows by linearity that $D_L$ maps $\langle M_l\rangle$ to $\langle M_{l+1}\rangle$.

Now, $e^{-\frac{n}4\Tr L}\big|L\big|^{\frac{m}4}I_p\in M_0$, so by induction $D^{2k}_L\{e^{-\frac{n}4\Tr L}\big|L\big|^{\frac{m}4}I_p\}\in\langle M_{2k}\rangle$. Hence, for some polynomials $b_{\kappa, s}^{(1)}(n,m,p)$ of degree at most $2k-q(\kappa)$,
{\setlength{\mathindent}{10pt}
\setlength{\abovedisplayskip}{3pt}
\setlength{\belowdisplayskip}{5pt}\begin{align}
\Tr D_L^{2k}\Big\{ e^{-\frac{n}4\Tr L} \big|L\big|^{\frac{m}4} I_p \Big\}
&= \hspace{-5pt}\sum_{|\kappa|+s\leq 2k} \hspace{-5pt}
	b_{\kappa, s}^{(1)}(n,m,p) 
	e^{-\frac{n}4\Tr L} \big|L\big|^{\frac{m}4} r_\kappa(L^{-1}) \Tr(L^{-s})
\notag\\&
= \hspace{-5pt}\sum_{|\kappa'|\leq 2k} \hspace{-5pt}
	b_{\kappa'}^{(1)}(n,m,p) e^{-\frac{n}4\Tr L} \big|L\big|^{\frac{m}4} r_{\kappa'}(L^{-1})
\label{eqn:sdl-first}
\end{align}}%
for $\kappa'=\kappa+s$, $b_{\kappa'}^{(1)} = b_{\kappa, s}^{(1)}$ when $s\neq 0$, while $\kappa'=\kappa$, $b_{\kappa'}^{(1)} = p b_{\kappa, s}^{(1)}$ when $s=0$. Notice that when $s\neq 0$, the degree of the $b_{\kappa'}$'s is at most $2k-q(\kappa)=2k-(q(\kappa')-1)$, while when $s=0$ it is at most $2k-q(\kappa)+1=2k-q(\kappa')+1$. Thus in both cases, $\mathrm{deg}\, b_{\kappa'}^{(1)}\leq 2k-q(\kappa')+1$, which by Equation \eqref{eqn:sdl-DL1} shows the first statement of the lemma.

For the second statement of the lemma, by an argument analoguous to Equation \eqref{eqn:sdl-first} we find that $\Tr D_L^{k}\{e^{-\frac{n}4\Tr L} \big|L\big|^{\frac{m}4}  I_p\}\in\langle M_{k+1}\rangle$. Thus by induction again, we must have $D_L^{k}\{\Tr D_L^{k}\{e^{-\frac{n}4\Tr L} \big|L\big|^{\frac{m}4}  I_p\}\}\in\langle M_{2k+1}\rangle$. Hence for some polynomials $b^{(2)}_{\kappa, s}(n,m,p)$ of degree at most $2k+1-q(\kappa)$,
{\setlength{\mathindent}{10pt}
\setlength{\abovedisplayskip}{3pt}
\setlength{\belowdisplayskip}{5pt}\begin{align*}&
D_L^{k}\Big\{ \Tr D_L^{k}\Big\{ e^{-\frac{n}4\Tr L} \big|L\big|^{\frac{m}4} I_p \Big\}\Big\}
\\&\qquad
= \hspace{-5pt}\sum_{|\kappa|+s\leq 2k+1} \hspace{-15pt}
	b_{\kappa, s}^{(2)}(n,m,p) 
	e^{-\frac{n}4\Tr L} \big|L\big|^{\frac{m}4} r_\kappa(L^{-1}) \Tr(L^{-s})
\\&\qquad
	= \hspace{-5pt}\sum_{|\kappa'|\leq 2k+1} \hspace{-5pt}
		b_{\kappa'}^{(2)}(n,m,p) e^{-\frac{n}4\Tr L} \big|L\big|^{\frac{m}4} r_{\kappa'}(L^{-1})
\end{align*}}%
for again $\kappa'=\kappa+s$, $b^{(2)}_{\kappa'}=b^{(2)}_{\kappa, s}$ when $s\neq0$, while $\kappa'=\kappa$, $b_{\kappa'}^{(2)} = p b_{\kappa, s}^{(2)}$ when $s=0$. By the same argument as before, $\mathrm{deg}\, b_{\kappa'}^{(2)}\leq 2k-q(\kappa')+2$, which by Equation \eqref{eqn:sdl-DL2} shows the second statement of the lemma. This concludes the proof.
\end{proof}

We will also need in our proof a result about the asymptotics of inverse moments of the Wishart distribution. Because we couldn't find anything like it in the literature, we think it is worthwhile to provide some context.

Let $f: (0,4)\rightarrow\R$ be the restriction to the positive reals of a complex function analytic in a neighborhood of $(0,4)$. We are often interested in the linear spectral statistic $\frac1p\Tr f(Y)$ for $Y\sim\text{W}_p(n, I_p/n)$. Much is known about its distributional properties in the high-dimensional regime where $p\rightarrow\infty$ such that $\lim\limits_{n\rightarrow\infty}\frac{p}{n} = \alpha <1$. For example, if $0<\alpha<1$, there must be an $\epsilon>0$ such that $p/n\in[\epsilon, 1-\epsilon]$ for all $n$ large enough, so \citet[Theorem 9.10]{bai2010} and the dominated convergence theorem yield that
{\setlength{\mathindent}{20pt}
\setlength{\abovedisplayskip}{3pt}
\setlength{\belowdisplayskip}{0pt}\begin{align}
\frac1{p}\Tr f(Y) 
\overset{\mathcal{P}}{\longrightarrow} &
	\lim_{n\rightarrow\infty} 
	\text{\raisebox{2pt}{$\bigints_{\;[p/n]_-}^{[p/n]_+}$}}\hspace{-25pt}
		f(t) \frac{\sqrt{([p/n]_+-t)(t-[p/n]_-)}}{2\pi[p/n] t}dt
\notag\\&
= \text{\raisebox{2pt}{$\bigintss_{\;\alpha_-}^{\alpha_+}$}}\hspace{-10pt}
		f(t) \frac{\sqrt{(\alpha_+-t)(t-\alpha_-)}}{2\pi\alpha t}dt
\label{eqn:lss-probability}
\end{align}}%
as $n\rightarrow\infty$. Here, $\overset{\mathcal{P}}{\rightarrow}$ stands for convergence in probability and $x_\pm$ for $(1\pm \sqrt{x})^2$. In fact, the theorem states more, namely a central limit theorem, but what we want to draw to attention is the class of functions for which this result was proven.

This is sometimes enough, but often we would like to understand the expectation of this linear spectral statistic. If $f$ is bounded, then Equation \eqref{eqn:lss-probability} implies that
{\setlength{\mathindent}{20pt}
\setlength{\abovedisplayskip}{0pt}
\setlength{\belowdisplayskip}{3pt}\begin{align}
\lim_{n\rightarrow\infty}\frac1{p}\E{\vphantom{\Big\vert}\Tr f(Y)} 
= \text{\raisebox{2pt}{$\bigintss_{\;\alpha_-}^{\alpha_+}$}}\hspace{-10pt}
		f(t) \frac{\sqrt{(\alpha_+-t)(t-\alpha_-)}}{2\pi\alpha t}dt.
\label{eqn:lss-for-bounded}
\end{align}}%
This is nice for a function $f(z)$ like $e^z$ or $\sin z$ that happens to be bounded on a neighborhood of $(0, 4)$, but it unfortunately excludes many interesting unbounded functions, such as $\log z$ or $1/z$. In fact, for unbounded $f$, it is in general not even clear if $\lim\limits_{n\rightarrow\infty} \frac1{p}\E{\Tr f(Y)}$ will be finite!

The following result shows that, at least in the case $f(z)=1/z^s$ for $s\in\N$, we can use Stein's lemma to obtain Equation \eqref{eqn:lss-for-bounded} and its $\alpha=0$ analogue.

\begin{lemma}\label{lem:hd-wishart}
Let for $Y\sim\text{W}_p(n, I_p/n)$ and $s$ be any integer $s\geq1$. Then as long as $n\geq p+4s+2$, the $s^\text{th}$ inverse moment satisfies the recursive bound
{\setlength{\mathindent}{20pt}
\setlength{\abovedisplayskip}{7pt}
\setlength{\belowdisplayskip}{3pt}\begin{align*}&
\left(1-\frac{(p+1)s}{n}\right) \E{\vphantom{\Big\vert} \Tr Y^{-s}}
\;\;\leq\;\;
\E{\Tr Y^{-(s-1)}}.
\end{align*}}%
In particular, as $p\rightarrow\infty$ such that $\lim\limits_{n\rightarrow\infty}\frac{p}{n} = \alpha <1$, if $s < \alpha^{-1}-1$ then
{\setlength{\mathindent}{20pt}
\setlength{\abovedisplayskip}{2pt}
\setlength{\belowdisplayskip}{2pt}\begin{align*}&
\lim_{n\rightarrow\infty}\frac1p\E{\vphantom{\Big\vert} \Tr Y^{-s}} = \begin{dcases}
\text{\raisebox{2pt}{$\bigintss_{\;\alpha_-}^{\alpha_+}$}}\hspace{-10pt}
\frac{\sqrt{(\alpha_+-t)(t-\alpha_-)}}{2\pi\alpha t^{s+1}}dt
& \quad\text{if}\quad 0<\alpha<1, \\
\hspace{55pt}1 & \quad\text{if}\quad\alpha=0.
\end{dcases}
\end{align*}}%
for $\alpha_\pm=(1\pm\sqrt{\alpha})^2$.
\end{lemma}
\begin{proof}
The classical Stein's lemma states that for any differentiable function $f:\R\rightarrow \R$ such that $\E{\big|\big(\frac{\partial}{\partial Z}-Z\big)f(Z)\big|}<\infty$ for $Z\sim\text{N}(0,1)$ and  $\lim\limits_{z\rightarrow\pm\infty}f(z)e^{-z^2/2}=0$, we must have
{\setlength{\abovedisplayskip}{2pt}
\setlength{\belowdisplayskip}{2pt}\begin{align*}&
\E{\Big(\frac{\partial}{\partial Z}-Z\Big)f(Z)}=0.
\end{align*}}%
Let $Z\sim\text{N}_{n\times p}(0, I_n\otimes I_p)$ be an $n\times p$ matrix of i.i.d. standard normal random variables, and let $Y=\frac1nZ^tZ\sim\text{W}_p(n, I_p/n)$. For any $1\leq \alpha\leq n$ and $1\leq \beta, i, j \leq p$,
{\setlength{\mathindent}{0pt}
\setlength{\abovedisplayskip}{3pt}
\setlength{\belowdisplayskip}{3pt}\begin{align*}&
\frac{\partial}{\partial Z_{\alpha\beta}}
\!=\! \frac{2}{n}\sum_{i=1}^p Z_{\alpha i}\frac{\partial_\text{s}}{\partial_\text{s} Y_{i\beta}}
\;\;\;\text{and}\;\;\;
\frac{\partial_\text{s} Y_{\beta j}^{-s} }{\partial_\text{s}Y_{i\beta}}
\!=\! -\frac12\sum_{l=1}^s\Big[Y_{\beta i}^{-l}Y_{\beta j}^{-(s-l+1)} \!\!+ Y_{\beta \beta}^{-l}Y_{ij}^{-(s-l+1)} \Big],
\end{align*}}%
so for $\delta$ the Kronecker delta,
{\setlength{\mathindent}{0pt}
\setlength{\abovedisplayskip}{3pt}
\setlength{\belowdisplayskip}{3pt}\begin{align}&
\bigg( \frac{\partial}{\partial Z_{\alpha\beta}} - Z_{\alpha\beta}\bigg)\big(ZY^{-s}\big)_{\alpha\beta}
\notag\\&\qquad
= \sum_{j=1}^p\bigg[
		\delta_{\beta j}Y^{-s}_{\beta j} + \frac{2}{n}\sum_{i=1}^p Z_{\alpha j}Z_{\alpha i} \frac{\partial_\text{s}}{\partial_\text{s} Y_{i\beta}} Y^{-s}_{\beta j} - Z_{\alpha\beta}Z_{\alpha j}Y^{-s}_{\beta j}
	\bigg]
\notag\\&\qquad
= Y^{-s}_{\beta\beta} 
	- \!\frac1{n}\sum_{l=1}^s \big(ZY^l\big)_{\alpha\beta}\big(ZY^{-(s-l+1)}\big)_{\alpha\beta}
	- \!\frac1{n}\sum_{l=1}^s Y_{\beta\beta}^{-l} \big(ZY^{-(s-l+1)} Z^t\big)_{\alpha\alpha}
\notag\\&\pushrightn{
	-Z_{\alpha\beta}\big(ZY^{-s}\big)_{\alpha\beta}
\label{eqn:hdwishart-componentwise}
} \end{align}}%
Let us first show that this expression is integrable. For any matrix $X$, $|X_{ij}|\leq \|X\|_2=\|X^tX\|_2^{1/2}$. Thus by Equation \eqref{eqn:hdwishart-componentwise},
{\setlength{\mathindent}{10pt}
\setlength{\abovedisplayskip}{3pt}
\setlength{\belowdisplayskip}{3pt}\begin{align*}&
\E{\bigg|\bigg( \frac{\partial}{\partial Z_{\alpha\beta}} - 	
	Z_{\alpha\beta}\bigg)\big(ZY^{-s}\big)_{\alpha\beta}
	\bigg|}
\\&\qquad
\leq \E{\vphantom{\bigg\vert} 
	\|Y^{-s}\|_2
	+ \sum_{l=1}^s\|Y^{2l+1}\|^{\frac12}_2\|Y^{2s+2l-1}\|^{\frac12}_2
\right.\\&\left.\hspace{120pt}\vphantom{\bigg\vert} 
	+ \sum_{l=1}^s\|Y^{-2l}\|^{\frac12}_2\|Y^{-2s+2l}\|_2^{\frac12}
	+ n\|Y\|^{\frac12}_2\|Y^{-2s+1}\|_2^{\frac12}
	}
\end{align*}}%
As $Y$ is positive definite, $\|Y^{\pm a}\|_2\leq \Tr Y^{\pm a}$ for any $a\in \N$, so by the Cauchy-Schwarz inequality,
{\setlength{\mathindent}{10pt}
\setlength{\abovedisplayskip}{3pt}
\setlength{\belowdisplayskip}{3pt}\begin{align*}&\qquad
\leq \E{\vphantom{\bigg\vert} \Tr Y^{-s}}
	+ \sum_{l=1}^s\E{\vphantom{\Big\vert} \Tr Y^{-2l+1}}^{\frac12}  
		\E{\vphantom{\Big\vert} \Tr Y^{-2s+2l-1}}^{\frac12}
\\&\hspace{40pt}
	+ \sum_{l=1}^s\E{\vphantom{\Big\vert} \Tr Y^{-2l}}^{\frac12}  
		\E{\vphantom{\Big\vert} \Tr Y^{-2s+2l}}^{\frac12}
	+ n\E{\vphantom{\Big\vert} \Tr Y}^{\frac12}
			\E{\vphantom{\Big\vert} \Tr Y^{-2s+1}}^{\frac12},
\end{align*}}%
which is finite for $n\geq p+4s+2$.

Moreover, $(ZY^{-s})_{\alpha\beta}$ can be expressed using minors and determinants as a rational function of the entries of $Z$, so
{
\setlength{\abovedisplayskip}{3pt}
\setlength{\belowdisplayskip}{3pt}\begin{align*}
\lim_{Z_{\alpha\beta}\rightarrow\pm\infty}(ZY^{-s})_{\alpha\beta}e^{-Z_{\alpha\beta}^2/2}
=0.
\end{align*}}
So all conditions are fulfilled to apply Stein's lemma to Equation \ref{eqn:hdwishart-componentwise} and obtain
{\setlength{\mathindent}{10pt}
\setlength{\abovedisplayskip}{3pt}
\setlength{\belowdisplayskip}{3pt}\begin{align*}
0 &= \E{\frac1{n}\sum_{\alpha=1}^n\sum_{\beta=1}^p 
	\bigg( \frac{\partial}{\partial Z_{\alpha\beta}} - 	Z_{\alpha\beta}\bigg)
	\big(ZY^{-s}\big)_{\alpha\beta}
	}
\\&
= \E{\Tr Y^{-s} 
	-\frac{s}{n}\Tr Y^{-s} 
	-\frac1n\sum_{l=1}^s\Tr(Y^{-l})\Tr(Y^{-(s-l)})
	-\Tr Y^{-(s-1)}
	}
\end{align*}}
As $\Tr(Y^{-l})\Tr(Y^{-(s-l)}) \leq p\Tr Y^{-s}$ for any $1\leq l \leq s$, and every term is integrable as $n\geq p+4s+2$, this means that
{\setlength{\mathindent}{20pt}
\setlength{\abovedisplayskip}{3pt}
\setlength{\belowdisplayskip}{3pt}\begin{align}
\left(1-\frac{(p+1)s}{n}\right) \E{\vphantom{\Big\vert} \Tr Y^{-s}}
	\leq \E{\vphantom{\Big\vert} \Tr Y^{-(s-1)}}.
\label{eqn:hdwishart-tracial}
\end{align}}
This shows the first part of the proof.

For the second part, let $S\in\N$. If we let $n\rightarrow\infty$ such that $\lim\limits_{n\rightarrow\infty} \frac{p}{n}=\alpha <1$, then any $S<\alpha^{-1}$ we will have $n\geq p + 4S+2$ and $n\geq (p+1)S$ for $n$ large enough. So by repeatedly applying Equation \eqref{eqn:hdwishart-tracial} for $s=S,\dots,1$ and dividing by $p$, we obtain
{\setlength{\mathindent}{20pt}
\setlength{\abovedisplayskip}{3pt}
\setlength{\belowdisplayskip}{3pt}\begin{align*}
\prod_{l=1}^S \left(1-\frac{(p+1)l}{n}\right) \cdot 
	\frac1{p}\E{\vphantom{\Big\vert} \Tr Y^{-S}}
\;\;\leq\;\; \frac1{p}\E{\vphantom{\Big\vert} \Tr Y^{-0}} = 1.
\end{align*}}%
Taking a limit in the above yields
{\setlength{\mathindent}{20pt}
\setlength{\abovedisplayskip}{3pt}
\setlength{\belowdisplayskip}{3pt}\begin{align*}
\prod_{l=1}^S \left(1-\alpha l\right) 
	\lim_{n\rightarrow\infty}\frac1{p}\E{\vphantom{\Big\vert} \Tr Y^{-S}}
\;\;\leq\;\; 1.
\end{align*}}%
Thus for any $S < \alpha^{-1}$, we have
{\setlength{\mathindent}{20pt}
\setlength{\abovedisplayskip}{3pt}
\setlength{\belowdisplayskip}{3pt}\begin{align}
\lim_{n\rightarrow\infty}\frac1{p}\E{\vphantom{\Big\vert} \Tr Y^{-S}}
\;\;\leq\;\; \prod_{l=1}^S \frac1{1-\alpha l} \;\;<\;\; \infty.
\label{eqn:hdwishart-finitelimit}
\end{align}}%

In the case $0<\alpha<1$, if $s+1<\alpha^{-1}$ then by Jensen's inequality and Equation \eqref{eqn:hdwishart-finitelimit} applied to $S=s+1$, we have
{\setlength{\mathindent}{20pt}
\setlength{\abovedisplayskip}{3pt}
\setlength{\belowdisplayskip}{3pt}\begin{align*}
\lim_{n\rightarrow\infty}\frac1{p}\E{\vphantom{\Big\vert} \Big(\Tr Y^{-s}\big)^{1+\frac1{s}}}
\;\;\leq\;\; \lim_{n\rightarrow\infty}\frac1{p}\E{\vphantom{\Big\vert} \Tr Y^{-(s+1)}} \;\;<\;\; \infty.
\end{align*}}%
Thus $\frac1p\Tr Y^{-s}$ is uniformly integrable, and by Equation \eqref{eqn:lss-probability},
{\setlength{\mathindent}{20pt}
\setlength{\abovedisplayskip}{0pt}
\setlength{\belowdisplayskip}{3pt}\begin{align*}
\lim_{n\rightarrow\infty}\frac1{p}\E{\vphantom{\Big\vert}\Tr T^{-s}} 
= \text{\raisebox{2pt}{$\bigintss_{\;\alpha_-}^{\alpha_+}$}}\hspace{-10pt}
		\frac{\sqrt{(\alpha_+-t)(t-\alpha_-)}}{2\pi\alpha t^{s+1}}dt
\end{align*}}%
for $\alpha_\pm=(1\pm\sqrt{\alpha})^2$.

In contrast, by applying Jensen's inequality twice,
{\setlength{\mathindent}{20pt}
\setlength{\abovedisplayskip}{5pt}
\setlength{\belowdisplayskip}{5pt}\begin{align*}
1 = \E{\frac1{p}\Tr Y}^{-s} \leq \E{\Big(\frac1{p}\Tr Y\Big)^{-s}} \leq \E{\frac1{p}\Tr Y^{-s}}
\end{align*}}%
so when $\alpha=0$, by applying Equation \eqref{eqn:hdwishart-finitelimit} with $S=s$, we obtain that $\lim\limits_{n\rightarrow\infty}\frac1{p}\E{\Tr Y^{-s}}=1$, as desired.
\end{proof}
\section{Conclusion}\label{sec:conclusion}

The results of this paper raise more questions than they answer. We enumerate some that we found particularly interesting.

\begin{enumerate}[label=(\arabic*)]
\item The univariate $t$ distribution with $\nu$ degrees of freedom is often defined as the distribution of $Z/\sqrt{s}$, for $Z\sim\text{N}(0,1)$ and $s\sim\chi^2_\nu/\nu$ independent. In the real symmetric matrix case, we could imagine studying the distribution of $S^{1/4}ZS^{1/4}$, for $Z\sim\text{GOE}(p)$ and $S\sim\text{W}_p(\nu, I_p/\nu)$ independent. Is this the $T_{\nu}(2I_p)$ distribution in the sense of Section \ref{sec:symmetrict}?
\item By Theorem \ref{thm:moments} and Corollary \ref{cor:semicircle}, it is clear the empirical moments of a symmetric $t$ distribution are quite similar to those of a Gaussian orthogonal ensemble matrix, except perhaps in their rates of convergence. From \citet[Theorem 2.1.31]{anderson10}, we know the empirical moments of a Gaussian orthogonal ensemble are asymptotically normal. Are the empirical moments of the symmetric $t$ distribution also asymptotically normal?
\item In Section \ref{sec:symmetrict}, we showed that the rate of convergence of the even normalized empirical moments of a symmetric $t$ distribution change when $p$ grows like $\sqrt{n}$. 
Can we find analogue symmetric $t$ statistics that change their rates of convergence when $p$ grows like $n^{(K+1)/(K+3)}$ for every $K\in\N$? This would establish phase transitions for the symmetric $t$ distribution. If so, can we find approximating densities between every two transitions, just like in the Wishart case?
\item As a counterpart of Theorem \ref{thm:existence-densities}, could we prove that $\mathrm{d}_\text{TV}(f_\text{NW}, f_K)\nrightarrow0$ whenever $p^{K+3}/n^{K+1}\nrightarrow0$ as $n\rightarrow\infty$? This is delicate because we have no guarantee that the $L^1$ norm of $\psi_K$ is asymptotically bounded for regimes of degree $K+1$ or higher.
\item Can we find the normalization constant or, better, solve the expectation of Equation \eqref{eqn:fK-as-expectation} in closed form?
\item What asymptotics hold for the symmetric $t$ or the Wishart distribution in a middle-scale regime of infinite degree? How do these asymptotics differ from the other middle-scale regimes, or the high-dimensional regime?
\item The symmetric $t$ distribution was discovered as the G-conjugate of the Wishart distribution. What other distributions can be realized as the G-conjugate of some well-known distribution?
\item In Lemma \ref{lem:moments-as-expectation}, we expressed the characteristic function of the G-conjugate $F^*$ of a distribution $F$ as $f^{1/2}\star (f^{1/2}\circ R)$, for $f$ the density of $F$ and $R$ the flip operator. To obtain the moments, we then repeatedly differentiated under the convolution integral at zero, and obtained an expression of the moments as an expectation with respect to $f$. The argument worked when $F^*$ was the symmetric $t$ distribution. Can this argument be generalized to other $F^*$? If $F^*$ is a well-known distribution, does this give rise to novel and nontrivial expressions for its moments?
\item The G-transform of a distribution encodes all the information relative to that distribution. However, taking a modulus removes some information, and so in some sense the G-conjugate distribution is ``less informative'' than the original distribution. What happens when we repeatedly apply the G-conjugation operator, destroying information every time? For example, is there an attractor distribution $G$ that is the limit of this process regardless of the initial distribution?
\item Can we find distinct random operators which can be regarded, in some sense, as the total variation limit of a normalized Wishart matrix between every two phase transitions?
\end{enumerate}

It appears to us that some of these questions might be very difficult to answer. We would be pleased if future work were able to shed light on any of them.

\bibliographystyle{plainnat}
\bibliography{bibliography}

\end{document}